\documentclass[12pt,centertags,oneside]{amsart}
\usepackage{amsmath,amstext,amsthm,amscd,typearea,hyperref}
\usepackage{amssymb}
\usepackage{a4wide}
\usepackage[mathscr]{eucal}
\usepackage{mathrsfs}
\usepackage{typearea}
\usepackage{charter}
\usepackage{pdfsync}
\usepackage[a4paper,width=16.5cm,top=3cm,bottom=3cm]{geometry}

\numberwithin{equation}{section}

\allowdisplaybreaks
\emergencystretch=\maxdimen
\usepackage{multicol}

\usepackage{xcolor}

\newtheorem{theorem}{Theorem}[section]
\newtheorem{definition}[theorem]{Definition}
\newtheorem{proposition}[theorem]{Proposition}
\newtheorem{corollary}[theorem]{Corollary}
\newtheorem{lemma}[theorem]{Lemma}
\newtheorem{remark}[theorem]{Remark}

\newtheorem*{definition*}{Definition}

\newtheorem{mainthm}{Theorem}

\newcommand{\cali}[1]{\mathscr{#1}}

\newcommand{\SL}{{\rm SL}}

\newcommand{\supp}{{\rm supp}}

\newcommand{\diff}{{\rm d}}

\newcommand{\del}{\partial}

\newcommand{\dbar}{\overline\partial}

\newcommand{\ep}{\epsilon}

\newcommand{\Cc}{\cali{C}}

\newcommand{\W}{\cali{W}}

\newcommand{\D}{\mathbb{D}}
\newcommand{\C}{\mathbb{C}}

\newcommand{\N}{\mathbb{N}}
\newcommand{\Z}{\mathbb{Z}}
\newcommand{\R}{\mathbb{R}}

\renewcommand\P{\mathbb{P}}

\newcommand{\E}{\mathbf{E}}

\newcommand{\omegaFS}{ \omega_{\mathrm{FS}}}

\newcommand{\norm}[1]{\lVert#1\rVert}

\newcommand{\sgn}{{\rm sgn}}
\newcommand{\g}{{\mathbf g}}

\newcommand{\oF}{\mathcal{F}}
\newcommand{\oP}{\mathcal{P}}
\newcommand{\oN}{\mathcal{N}}
\newcommand{\oR}{\mathcal{R}}
\newcommand{\oE}{\mathcal{E}}
\newcommand{\oU}{\mathcal{U}}
\newcommand{\oQ}{\mathcal{Q}}
\newcommand{\oS}{\mathcal{S}}
\newcommand{\oL}{\mathcal{L}}

\newcommand{\bN}{\mathbf{N}}
\newcommand{\bM}{\mathbf{M}}
\newcommand{\bT}{\mathbf{T}}
\newcommand{\bS}{\mathbf{S}}
\newcommand{\bA}{\mathbf{A}}
\newcommand{\bB}{\mathbf{B}}

\newcommand{\bQ}{\mathbf{Q}}
\newcommand{\bC}{\mathbf{C}}
\newcommand{\bP}{\mathbf{P}}

\usepackage{fancyhdr}

\title{Decay of Fourier coefficients for Furstenberg measures}

\author{Tien-Cuong Dinh}
\address{Department of Mathematics,  National University of Singapore - 10, Lower Kent Ridge Road - Singapore 119076}
\email{matdtc@nus.edu.sg}

\author{Lucas Kaufmann}
\address{Center for Complex Geometry - Institute for Basic Science (IBS) - 55 Expo-ro Yuseong-gu Daejeon 34126 South Korea}
\email{lucaskaufmann@ibs.re.kr}

\author{Hao Wu}
\address{Department of Mathematics,  National University of Singapore - 10, Lower Kent Ridge Road - Singapore 119076}
\email{matwu@nus.edu.sg}

\thanks{This work was supported by the NUS and MOE grants  AcRF Tier 1 R-146-000-259-114, R-146-000-299-114 and MOE-T2EP20120-0010. L. Kaufmann was supported by the Institute for Basic Science (IBS-R032-D1)}

\begin{document}

\begin{abstract}
Let $\nu$ be the Furstenberg measure  associated with a non-elementary probability measure $\mu$ on $\SL_2(\R)$. We show that, when $\mu$ has a finite second moment, the Fourier coefficients of $\nu$ tend to zero at infinity. In other words, $\nu$ is a Rajchman measure. This improves a recent result of Jialun Li.
\end{abstract}

\clearpage\maketitle
\thispagestyle{empty}


\noindent\textbf{Mathematics Subject Classification 2010:} \texttt{60B15,60B20,60K15,42A16,37C30}.

\setcounter{tocdepth}{1}
\tableofcontents

\section{Introduction}

Let $\bC$ be a circle parametrized by $e^{i\theta}$ with $0\leq\theta< 2\pi$. For a probability measure $\nu$ on $\bC$,  we define its  \textit{$k$-th Fourier coefficient} by $$\widehat\nu (k):=\int_{\bC } e^{i k \theta} \,\diff \nu(\theta) \quad\text{for}\quad k\in\Z.$$

The behavior of $\widehat\nu (k)$ when $|k| \to \infty$ and in particular, the question of whether $\widehat\nu (k)$ tends to zero are widely studied. This is linked with different problems in analysis, geometry, number theory, etc. In particular, this question appears naturally in the study of sets of uniqueness for trigonometric series, a topic that has raised interest since Riemann, Cantor, etc., and in the study of various problems in geometric measure theory. The reader may refer to the survey \cite{korner} and the book  \cite{mattila:book} for an overview of these topics. 

It is a standard principle in Fourier analysis that regularity properties of a function can be obtained from decay properties of its Fourier transform. This principle can be carried to probability measures, but in this case, properties other than regularity (often of arithmetic nature) play a role in the decay of the Fourier transform. For instance, it can be shown that the Fourier coefficients of any measure on the standard middle-thirds Cantor set do not decay to zero as $k$ tends to infinity. The same is true for measures on the standard middle $\lambda$-Cantor set, $0< \lambda < 1 / 2$,  precisely when $\lambda^{-1}$ is a Pisot number, highlighting the role of arithmetic properties in this problem, see \cite[Chapter 8]{mattila:book} and \cite{kahane-salem}. We refer to the introduction of \cite{li-sahlsten} and the references therein for a nice overview of this topic.

In this context, we say that a probability measure $\nu$ on $\bC$ is a \textit{Rajchman measure} if its Fourier transform vanishes at infinity, that is, $\widehat\nu (k) \to 0$ as $|k| \to \infty$. They form an important class of measures and appear in many situations, see for instance the survey \cite{lyons:rajchman}.  The aim of this work is to obtain the Rajchman property for probability measures arising as the Furstenberg measure associated with a low-moment random walk on $\SL_2(\R)$.
\vskip5pt

We now make our result precise. Denote by $G_\R:= \SL_2(\R)$ (resp. $G:= \SL_2(\C)$) the space of real (resp. complex) $2$ by $2$ matrices of determinant one. The standard linear action of $G_\R$ on $\R^2$ descends to the real projective line $\R\P^1$. We can also see a matrix in $G_\R$ as an element of $G$ acting on the complex projective line $\C \P^1 =: \P^1$  (i.e., the Riemann sphere) and  preserving the real submanifold $\R\P^1$. Using a stereographic projection, we naturally identify  $\R\P^1$ with a circle $\bC$. In particular, measures on $\P^1$ supported by $\R\P^1$ can be identified with measures on $\bC$ and their Fourier coefficients can be defined as above.

Let $\mu$ be a probability measure on $G$. Then, $\mu$ induces a random walk on $G$ by setting $$S_n: = g_n \cdots g_1,$$ where the $g_j$'s are independent random elements of $G$ with law given by $\mu$. 

A measure $\nu$ on $\P^1$ is called \textit{stationary} if
$$\int_{G} g_* \nu \, \diff \mu(g)= \nu,$$ where $g_*$ stands for the push-forward under the map induced by $g$. An easy compactness argument shows that stationary measures always exist.

We say that $\mu$ is \textit{non-elementary} if its support does not preserve a finite subset of $\P^1$ and if the semi-group it generates is not relatively compact in $G$.  In this case, $\mu$ admits a unique stationary measure which is called the \textit{Furstenberg measure} associated with $\mu$. When $\mu$ is supported by $G_\R$, the associated Furstenberg measure $\nu$ is supported by $\bC$. In particular, one can define the  Fourier coefficients of $\nu$.

We can now state our main theorem. 

\begin{mainthm}\label{thm:main-fourier}
	Let $\mu$ be a non-elementary probability measure on $G_\R=\SL_2(\R)$. Assume that $\mu$ has a finite second moment, that is, $\int_{G_\R} \log^2 \|g\| \, \diff \mu(g) < \infty$. Let $\nu$ be the associated stationary measure on $\bC$. Then 
	$$\widehat\nu(k) \longrightarrow 0 \quad \text{ as } \quad |k|\to \infty.$$
	In other words, $\nu$ is a Rajchman measure.
\end{mainthm}

 Under the stronger \textit{exponential moment condition}, i.e., $\int_{G_\R} \|g\|^\alpha \, \diff \mu(g) < \infty$ for some $\alpha > 0$,  the above result is due to Li \cite{li:fourier}. The sequel \cite{li:fourier-2} provides a polynomial decay rate, that is, $|\widehat \nu (k)| \lesssim 1/|k|^\beta$ for some $\beta >0$.  Various related results can be found in recent literature.  For Patterson-Sullivan measures associated with convex co-compact subgroups of $\SL_2(\R)$, Bourgain-Dyatlov \cite{bourgain-dyatlov} obtained a polynomial Fourier decay. This is related to zero-free regions of the Selberg zeta functions. Li-Naud-Pan \cite{li-naud-pan} showed that the same type of decay is observed for Patterson-Sullivan measures associated with Zariski dense Schottky subgroups of $\SL_2(\C)$. They also note that such measures arise as the Fusternberg measure associated with some random walk on the corresponding group having a finite exponential moment. In the context of fractal geometry, Li-Sahlsten \cite{li-sahlsten} obtained the Fourier decay for a class of self-affine measures.

Theorem \ref{thm:main-fourier} will be a consequence of a more general result, see Theorem \ref{thm:fourier-general} below. Under a $(2+\epsilon)$-moment condition, that is $\int_{G_\R} \log^{2+\varepsilon} \|g\| \, \diff \mu(g) < \infty$ for some $\varepsilon >0$, our techniques can be used to show that $|\widehat \nu (k)| \lesssim 1 / \log \log |k|$. The proof is not presented here in order to make the paper less technical.

\vskip5pt

 The proof of Theorem \ref{thm:main-fourier} follows the strategy of \cite{li:fourier} and draw from techniques of our recent works \cite{DKW:IJM,DKW:PAMQ,DKW:LLT}. The key idea is to use stopping times and Cauchy-Schwarz inequality in order to bound the integral defining $\widehat\nu(k)$ by a quantitity that can be translated to fit the framework of renewal theory for Markov chains, as initiated by Kesten \cite{kesten:renewal} and further developed in \cite{boyer,guivarch-lepage,li:fourier,li:fourier-2}.  We refer to Section \ref{sec:fourier} for more details on the structure of the proof. In the setting of random matrix products, most known results about renewal operators require exponential moment conditions on the measure $\mu$ and break down in our situation. The main difficulties are the fact that natural pertubations of the Markov operator display very low regularity (see Section \ref{sec:markov}) and that the large deviation estimates, as obtained in \cite{benoist-quint:CLT}, are weaker in the finite second moment case.  Therefore, in order to be able to apply Li's method, we need to further develop the spectral analysis of the perturbed Markov operators initiated in \cite{DKW:LLT} and generalize most of the estimates of \cite{li:fourier} to cover the case of finite second moment. Due to the low regularity of the perturbed Markov operators, new arguments are needed. This is one of the main technical contributions of this work.
\vskip5pt

The text is structured as follows. In Section \ref{sec:prelim}, we recall some basic fact about the theory of random walks on $\SL_2(\C)$ and state a few basic results on Fourier analysis needed in our work. In Section \ref{sec:markov}, we study the spectral analysis of the Markov operator and its perturbations in the case of low moments as initiated in  \cite{DKW:PAMQ,DKW:LLT}. We prove new regularity results that are necessary in the analysis of the renewal operators. Section \ref{sec:renewal} is the technical core of this work, where we introduce many ``renewal operators'' and derive their asymptotic. Finally, in Section \ref{sec:fourier}, we use these results to prove our main theorem.

\medskip
 
 \noindent\textbf{Notations:} Throughout this paper, the symbols $\lesssim$ and $\gtrsim$ stand for inequalities up to a positive multiplicative constant. The dependence of these constants on certain parameters, or lack thereof, will be clear from the context.

\section{Preliminary results} \label{sec:prelim}

In this article, we'll work with $G=\SL_2(\C)$. In the end, some results will specialize to $G_\R = \SL_2(\R)$. 

\subsection{Norm cocycle,  Lyapunov exponent and density points}
Let $\mu$ be a probability measure on $G$.
   For $n \geq 1$, we define the convolution measure  $\mu^{*n} := \mu * \cdots * \mu$ ($n$ times) as the push-forward of the product measure $\mu^{\otimes n}$ on $G^n$ by the map $(g_1, \ldots, g_n) \mapsto g_n \cdots g_1$. This is the law of the product $S_n:= g_n \cdots g_1$ when the $g_j$'s are i.i.d. with law $\mu$.

Denote by $\norm{g}$ the operator norm of the matrix $g$. Note that since $g \in \SL_2(\C)$, we have that $\|g\| \geq 1$.

\begin{definition} \rm We say that  $\mu$  has a \textit{finite exponential moment} if  $\int_G \|g\|^\alpha \, \diff \mu(g) < \infty$  for some $\alpha > 0$, and a \textit{finite moment of order} $p >0$ (or a finite $p$-th moment) if  $$\int_G \log^p \|g\| \, \diff \mu(g) < \infty.$$
\end{definition}

For a probability measure $\mu$ with a finite first moment, we define the  \textit{(upper) Lyapunov exponent} of $\mu$ by 
$$\gamma := \lim_{n \to \infty} \frac1n \,\E \big( \log\|S_n\| \big)=\lim_{n \to \infty} \frac1n \int \log\|g_n \cdots g_1\| \, \diff \mu(g_1) \cdots \diff \mu(g_n)$$
and the \textit{norm cocycle} by $$\sigma(g,x) = \sigma_g(x):= \log \frac{\norm{gv}}{\norm{v}}, \quad \text{for }\,\, v \in \C^2 \setminus \{0\}, \, x = [v] \in \P^1   \, \text{ and } g \in G.$$
Observe that $\|\sigma_g\|_\infty = \log \|g\|$ and $g_* \sigma_g := \sigma_g \circ g^{-1} = - \sigma_{g^{-1}}$.

We say that $\mu$ is \textit{non-elementary} if its support does not preserve a finite subset of $\P^1$ and if the semi-group it generates is not relatively compact in $G$.  A non-elementary measure admits a unique \textit{stationary measure}, called the \textit{Furstenberg measure}. This is a probability measure $\nu$ on $\P^1$ satisfying $$\int_G g_* \nu \, \diff \mu(g)= \nu.$$ One can show that for non-elementary measures we always have $\gamma > 0$. See \cite{bougerol-lacroix} for the proofs of the above facts.

\textit{Cartan's decomposition} says that every matrix $g \in \SL_2(\C)$ can be written as
\begin{equation*} 
g=k_g a_g \ell_g, \quad \text{ where } \, \, k_g,\ell_g \in \text{SU}(2) \,\, \text{ and } \,\, a_g=\left(\begin{smallmatrix} \lambda & 0 \\ 0 & \lambda^{-1} \end{smallmatrix}\right) \,\, \text{ for some } \,\, \lambda \geq 1.
\end{equation*}
Observe that $\|g\| = \|g^{-1}\| = \lambda$.

Let $e_1,e_2 \in \P^1$ be the points corresponding to the standard basis of $\C^2$. We define the \textit{density points} of $g$ by $$z^m_g:=\ell_g^{-1} e_2 \quad \text{ and }\quad z^M_g:=k_g e_1.$$
These points have the property that, when $\lambda$ is large, there are small neighborhoods $V_g^m$ (resp. $V_g^M$) of $z^m_g$ (resp. $z^M_g$) such that $g$ maps $\P^1 \setminus V_g^m$ into $V_g^M$. Observe that $z^m_g=z^M_{g^{-1}}$.  These points are uniquely defined when $\lambda  > 1$. When $\lambda = 1$ these points are not uniquely determined and we set by convention $z^m_g=z^M_{g} := e_1$.  This doesn't change our estimates, see Proposition \ref{prop:BQLDT} below.
\medskip

We equip $\P^1$ is with the natural distance given by 
\begin{equation} \label{eq:distance-def}
d(x,y) : = \frac{| \det(v,w) \, |}{\|v\| \|w\|}, \quad \text{where} \quad v,w \in \C^2 \setminus \{0\}, \, x = [v]\in\P^1, \, y = [w] \in \P^1.
\end{equation}

Then,  $(\P^1, d)$ has diameter one and $\text{SU}(2)$ acts by holomorphic isometries. We denote by $\D(x,r)$ the associated disc of center $x$ and radius $r$. Since $\det(g)=1$, the above distance can be read from the norm cocycle by the formula
\begin{equation} \label{d-gxgy}
d(gx,gy)=\frac{ \|v\| \|w\|}   {    \|gv\| \|gw\|}d(x,y) =e^{-\sigma_g(x)-\sigma_g(y)}d(x,y).
\end{equation}

The distance to the density point $z_g^m$ is related to the value of the cocycle $\sigma_g$ by
	\begin{equation} \label{g^-2}
	d(z^m_g,x)\leq \frac{\norm{gv}}{\norm{g}\norm{v}} \leq d(z^m_g,x)+\norm{g}^{-2},
	\end{equation}
which can be easily proven using Cartan's decomposition. See \cite[Lemma 14.2]{benoist-quint:book}.

\subsection{Large deviation estimates} In the proof of Theorem \ref{thm:main-fourier}, the following large deviation estimates, due to Benoist-Quint, will be used frequently. We note that it also holds in the more general setting of linear groups over local fields. We refer to  \cite{benoist-quint:book} for the corresponding version for measures with a finite exponential moment.

\begin{proposition}[\cite{benoist-quint:CLT}--Proposition 4.1 and Lemma 4.8] \label{prop:BQLDT} 
	Let $\mu$ be a probability measure on $G=\SL_2(\C)$. Assume that $\mu$ is non-elementary and $\int_G \log^p \|g\| \, \diff \mu(g) < \infty$ for some $p > 1$. Let $\gamma > 0$ be the Lyapunov exponent of $\mu$. Then, for every $\epsilon>0$ there exist positive constants $C_{n,\epsilon}$ satisfying $\sum_{n\geq 1} n^{p-2}C_{n,\epsilon}< \infty$ and $\lim_{n\to \infty} n^{p-1} C_{n,\epsilon}=0$, such that 
	\begin{equation*}
	\mu^{*n} \big\{g\in G:\, \big|\log\|g\|-n\gamma\big|\geq \epsilon n\big\}\leq C_{n,\epsilon},
	\end{equation*}
	and for every $x\in\P^1$,
	\begin{equation*}
	\mu^{*n}\big\{g\in G:\, | \sigma_g(x)-n\gamma| \geq \epsilon n\big\}\leq C_{n,\epsilon}.
	\end{equation*}
	
	Moreover, for every $0<\epsilon < \gamma$ and $x,y\in\P^1$,  the following subsets of $G$ have $\mu^{*n}$ measure bounded by $C_{n,\epsilon}$:
	\begin{equation} \label{BQregular1}
	\begin{aligned}[c]
	\big\{g\in G:\,d(g^{-1}x,z^m_g)\geq e^{-(2\gamma-\ep) n}\big\} \\
	\big\{g\in G:\,d(gx,z^M_g)\geq e^{-(2\gamma-\ep) n}\big\}
	\end{aligned}
	\quad  \quad 
	\begin{aligned}[c] 
	&\big\{g\in G:\,d(z^m_g,y)\leq e^{-\ep n}\big\} \\
	&\big\{g\in G:\,d(z^M_g,y)\leq e^{-\ep n}\big\}\\
	&\big\{g\in G:\,d(g^{-1}x,y)\leq e^{-\ep n}\big\}.
	\end{aligned}
	\end{equation}	
\end{proposition}

We note that the fact that  $\lim_{n\to \infty} n^{p-1} C_{n,\epsilon}=0$ is not explicitly stated in \cite{benoist-quint:CLT}, but it can be  easily deduced from the proof of \cite[Theorem 2.2]{benoist-quint:CLT}.  In order to simplify the notation we'll often use the abbreviation $C_n$, omitting the dependence on $\ep$.  This shouldn't cause any confusion since $\ep$ will always be fixed, with its value depending on the context.
\medskip

In the course of the proof of our main theorem, we'll often need to control quantities depending on a parameter $t>0$ when $t$ is large. The following notion will ease some notation. 

\begin{definition}\rm
	A non-negative function $\varepsilon(t)$ on $\R_{\geq 0}$ is said to be a  \textit{rate function} if it satisfies $\lim_{t\to \infty} \varepsilon(t)=0$.
\end{definition}

Notice that, for any rate function $\varepsilon(t)$,  the function $t\mapsto \sup_{s\geq t}\varepsilon(s)$ is a decreasing rate function that dominates  $\varepsilon(t)$. Therefore, we can always assume that our rate functions are decreasing.

We now give a few useful consequences of Proposition \ref{prop:BQLDT}.  We mostly work with the case $p=2$,  as this corresponds to the assumption  of Theorem \ref{thm:main-fourier}.  Some of the results below can be improved when $\mu$ has a finite moment of order $p > 2$. We leave the details to the reader.

\begin{lemma} \label{large-n}
  Let $\mu$ be as in Proposition \ref{prop:BQLDT} and $p=2$. Then, there exists a decreasing rate function $\varepsilon_0(t)$ such that
		$$\sum _{n\geq \lceil 2t/ {\gamma}\rceil}\mu^{*n}\big\{g\in G:\, \log\|g\|\leq t\big\} + \sum _{n\leq \lfloor 2t/ (3\gamma)\rfloor}\mu^{*n}\big\{g\in G:\, \log\|g\|\geq t\big\}  \leq \varepsilon_0(t) $$
		and, for every $x \in \P^1$,
	$$	\sum _{n\geq \lceil 2t/ {\gamma}\rceil} \mu^{*n}\big\{g\in G:\,\sigma_g(x)  \leq t\big\} + \sum _{n\leq \lfloor 2t/ (3\gamma)\rfloor} \mu^{*n}\big\{g\in G:\,\sigma_g(x)  \geq t\big\} \leq \varepsilon_0(t) $$
	 for all $t>0$.
\end{lemma}

\begin{proof}
	We'll only prove the first assertion. The second one can be obtained in a similar way for the first term and using that $\sigma_g(x) \leq \log \|g\|$ for the second term.
	
Take $\ep=\gamma/2$ in Proposition \ref{prop:BQLDT}. For the first sum in the first  inequality,	notice that for $n\geq \lceil 2t/ {\gamma}\rceil$, we have $n\ep \geq t$ and hence
 $$\big\{g\in G:\, \log\|g\|\leq t\big\} \subseteq \big\{g\in G:\, |\log\|g\|-n\gamma|\geq \epsilon n\big\}.$$
By Proposition \ref{prop:BQLDT}, we deduce that $$\mu^{*n}\big\{g\in G:\, \log\|g\|\leq t\big\}\leq 	\mu^{*n}\big\{g\in G:\, |\log\|g\|-n\gamma|\geq \epsilon n\big\}\leq C_{n,\epsilon}.$$
Therefore,
	 $$\sum _{n\geq \lceil 2t/ {\gamma}\rceil}\mu^{*n}\big\{g\in G:\, \log\|g\|\leq t\big\}  \leq \sum_{n\geq \lceil 2t/ {\gamma}\rceil}C_{n,\epsilon},$$ 
	 which tends to $0$ as $t\to\infty$ since $\sum C_{n,\ep}$ is finite.
	 
	We now estimate the second sum in the statement of the lemma. Set  $M_n:=\int \log^2\|g\|  \, \diff \mu^{*n}(g)$. Then, $\mu^{*n}\big\{\log\|g\|\geq t\big\}\leq M_n/t^2$ by Chebyshev's inequality. On the other hand, the sub-additivity of $\log\|g\|$ and Cauchy-Schwarz inequality imply that $M_n\leq n^2M_1$.
	 
	 When $n\leq \sqrt t$, we obtain
	 $$\sum_{1\leq n\leq \sqrt t}\mu^{*n}\big\{\log\|g\|\geq t\big\}\leq\sum_{1\leq n\leq \sqrt t} M_n/t^2\leq \sum_{1\leq n\leq \sqrt t} n^2M_1/t^2\leq  (\sqrt t)^3M_1/t^2=M_1/\sqrt t.$$ 
	 
	 When $\sqrt t<n \leq \lfloor 2t/ (3\gamma)\rfloor$, we have $t  \geq n(\gamma+\ep)$. Then, Proposition \ref{prop:BQLDT} yields $$\mu^{*n}\big\{g\in G:\, \log\|g\|\geq t\big\}  \leq 	\mu^{*n}\big\{g\in G:\, |\log\|g\|-n\gamma|\geq \epsilon n\big\}\leq C_{n,\epsilon},$$
	 so
	 $$\sum_{\sqrt t<n\leq \lfloor 2t/ (3\gamma)\rfloor}\mu^{*n}\big\{g\in G:\, \log\|g\|\geq t\big\}\leq \sum_{\sqrt t<n\leq \lfloor 2t/ (3\gamma)\rfloor}C_{n,\epsilon}\leq \sum_{n>\sqrt t}C_{n,\epsilon}.$$ 
	 The last quantity tends to zero as $t \to \infty$ because $\sum C_{n,\epsilon}$ is finite.
	 Set
	 $$\varepsilon_0(t):=\sum_{n\geq \lceil 2t/ {\gamma}\rceil}C_{n,\epsilon}+2M_1/\sqrt t+\sum_{n>\sqrt t}C_{n,\epsilon},$$ which is  decreasing rate function. 	From the above estimates, the desired inequality follows. This concludes the proof.
\end{proof}

We will use in Sections \ref{sec:renewal} and \ref{sec:fourier} the following refined large deviation estimate. A version for more general linear groups can be proved similarly. This is the analogue of \cite[Lemma 17.8]{benoist-quint:book}, which holds when $\mu$ has a finite exponential moment. 

\begin{lemma}\label{diff-log} 	Let $\mu$ be a probability measure on $G=\SL_2(\C)$. Assume that $\mu$ is non-elementary and $\int_G \log^p \|g\| \, \diff \mu(g) <  \infty$ for some $p>1$. Let $100l\geq n> l\geq 1$ be integers and let $x\in\P^1$. Then there exist positive constants $D_n$ independent of $x$ satisfying $\sum_{n\geq 1} n^{p-2}D_n< \infty$ and $\lim_{n\to \infty}n^{p-1}D_n=0$, and  subsets $\bS_{n,l,x}\subseteq G \times G$ such that 
	\begin{equation} \label{refined-estimate}
	\big|\log\norm{g_2g_1}-\sigma_{g_2}(g_1x)-\log\norm{g_1}\big|\leq 4e^{-\gamma l} \quad \text{for all} \quad (g_2,g_1)\in \bS_{n,l,x},
	\end{equation}
	and $\mu^{*(n-l)} \otimes \mu^{*l}(\bS_{n,l,x})\geq 1-D_l- D_n$.
\end{lemma}

\begin{proof}
		Let $\ep>0$ be a small constant  and take  $C_{n,\ep}$ as in Proposition \ref{prop:BQLDT}. Define $$D_n:=3\max(C_{n,\ep},C_{n,\ep/100}).$$
		Clearly, we have $\sum_{n\geq 1} n^{p-2}D_n< \infty$ and $\lim_{n\to \infty}n^{p-1}D_n=0$.
		\vskip 3pt
		
	Let $\bS_{n,l,x}$ be the set of all $(g_2,g_1)\in G \times G$ such that the following inequalities hold:
	\begin{equation*} 
  \begin{aligned}[c]
  & \text{(A1)} \quad  -\ep l\leq\log\norm{g_1}-l\gamma \leq\ep l \\
  & \text{(A2)} \quad d(z_{g_1}^m,x)\geq e^{-\ep l} \\
  & \text{(A3)} \quad -\ep n\leq\log\norm{g_2g_1}-n\gamma \leq\ep n   
  \end{aligned}
\quad \text{ and } \quad 
  \begin{aligned}[c] 
  & \text{(A4)}  \quad d(z_{g_2g_1}^m,x)\geq e^{-\ep l} \\
  & \text{(A5)} \quad  d\big((g_2g_1)^{-1}x, z^m_{g_2g_1}\big)\leq e^{-(2\gamma -\ep)n} \\
  & \text{(A6)} \quad  d(g_1^{-1}g_2^{-1}x,z^m_{g_1})\leq e^{-(2\gamma-\ep)l}.
\end{aligned}
\end{equation*}

	\noindent \textbf{Claim 1.} Inequality \eqref{refined-estimate} holds for all $(g_2,g_1)\in\bS_{n,l,x}$.
\proof[Proof of Claim 1]
Fix $(g_2,g_1)\in\bS_{n,l,x}$.  
 Let $v\in\C^2\setminus \{0\}$ with $[v]=x$.
Observe that 
\begin{align*}
\big|\log\norm{g_2g_1}-\sigma_{g_2}(g_1x)-\log\norm{g_1}\big|=\Big|\log\frac{\norm{g_1v}}{\norm{g_1}\norm{v}} - \log\frac{\norm{g_2g_1v}}{\norm{g_2g_1}\norm{v}}  \Big|
\leq  B_1+B_2+B_3,
\end{align*}
where $$B_1:=\Big| \log \frac{\norm{g_1v}}{\norm{g_1}\norm{v}} - \log d(z^m_{g_1},x)\Big|,\quad
B_2:=\Big| \log\frac{\norm{g_2g_1v}}{\norm{g_2g_1}\norm{v}} - \log d(z^m_{g_2g_1},x)\Big| $$
and $B_3:=\big| \log d(z^m_{g_1},x)-\log d(z^m_{g_2g_1},x)\big|$.

\vskip 3pt
Using \eqref{g^-2} and (A1), we get 
$$0\leq {\norm{g_1v}\over\norm{g_1}\norm{v}} - d(z^m_{g_1},x)\leq e^{-2\gamma l+2\ep l} \quad \text{and}\quad  {\norm{g_1}\norm{v}\over \norm{g_1v}}\leq {1\over d(z^m_{g_1},x)}.$$
By the mean value theorem applied to the function $t\mapsto \log t$ between ${\norm{g_1v}\over\norm{g_1}\norm{v}}$ and $d(z^m_{g_1},x)$ and using (A2), we have
\begin{equation*}\label{diff-log-1}
B_1=\Big| \log{\norm{g_1v}\over\norm{g_1}\norm{v}} - \log d(z^m_{g_1},x)\Big| \leq  e^{\ep l}\cdot e^{-2\gamma l+2\ep l}\leq e^{-\gamma l}.
\end{equation*}

Repeating the similar arguments as above and using (A3) and (A4), we get
\begin{equation*}
B_2=\Big| \log{\norm{g_2g_1v}\over\norm{g_2g_1}\norm{v}} - \log d(z^m_{g_2g_1},x)\Big| \leq e^{-\gamma n}\leq e^{-\gamma l}.
\end{equation*} 

It remains to estimate $B_3$.  By using the triangle inequality, (A5) and (A6), we get  
\begin{equation*}
d(z_{g_1}^m,z_{g_2g_1}^m)\leq d\big((g_2g_1)^{-1}x, z^m_{g_1}\big)+d\big((g_2g_1)^{-1}x, z^m_{g_2g_1}\big)\leq e^{-(2\gamma-\ep)l}+e^{-(2\gamma-\ep)n}. 
\end{equation*}
From (A2), (A4) and the mean value theorem applied to the function $t\mapsto \log t$ between $d(z^m_{g_1},x)$ and $d(z^m_{g_2g_1},x)$, one has
\begin{equation*}
B_3=\big|\log d(z^m_{g_1},x)-\log d(z^m_{g_2g_1},x)\big|\leq e^{\ep l}\cdot \big(e^{-(2\gamma-\ep)l}+e^{-(2\gamma-\ep)n}\big)\leq 2e^{-\gamma l}.
\end{equation*}
The proof of Claim 1 is finished. \endproof

\noindent \textbf{Claim 2.} $\mu^{*(n-l)} \otimes \mu^{*l}(\bS_{n,l,x})\geq 1-D_n- D_l$.
\proof[Proof of Claim 2] From Proposition \ref{prop:BQLDT} we see that (A1),(A2),(A3),(A5) and (A6) hold for $(g_2,g_1)$ outside a set of $\mu^{*(n-l)} \otimes \mu^{*l}$-measure $3C_{l,\ep}+2C_{n,\ep}$ in $G \times G$.

Since $100l \geq n$, we have $e^{-\ep l}\leq e^{-\ep n/100}$.
Now applying Proposition \ref{prop:BQLDT} with $\ep/100$ instead of $\ep$, we obtain that
$$  \mu^{*(n-l)} \otimes \mu^{*l} \big\{ (g_2,g_1):\, d(z_{g_2g_1}^m,x)\leq e^{-\ep l}  \big\} \leq C_{n,\ep/100}.       $$ 
Hence (A4) holds for $(g_2,g_1)$ outside a set of $\mu^{*(n-l)} \otimes \mu^{*l}$ measure $C_{n,\ep/100}$. Therefore, we deduce that $$\mu^{*(n-l)} \otimes \mu^{*l}(\bS_{n,l,x})\geq 1- 3C_{l,\ep}-2C_{n,\ep}- C_{n,\ep/100} \geq  1-D_l- D_n,$$
which finishes the proof of Claim 2. \endproof

The lemma follows directly from the two claims.
\end{proof}

\subsection{Fourier transform}
The \textit{Fourier transform} of an integrable function $f$, denoted by $\widehat f$, is defined by
$$\widehat f(\xi):=\int_{-\infty}^{\infty}f(u)e^{-i u\xi} \,\diff u$$ 
and the inverse Fourier transform is $$f(u)={1\over {2\pi}}\int_{-\infty}^{\infty} \widehat f (\xi) e^{ i u\xi} \,\diff \xi.$$
With these definitions, the Fourier transform of $\widehat f(\xi)$ is $2\pi f(-u)$. The convolution formula says that $\widehat{f_1*f_2}=\widehat f_1\cdot \widehat f_2$.

We'll frequently deal with functions on $\R$ having non-integrable singularities at $u=0$ and  $u=\pm \infty$. In this case, we'll work with the \textit{Cauchy principal value}, defined by
$$\mathrm{p.v.} \int_{-\infty}^{\infty} f(u)\,\diff u:=\lim_{\epsilon \to 0^+ \atop M \to \infty} \Big(\int_{-M}^{-\epsilon} f(u)\,\diff u +\int^{M}_\epsilon f(u)\,\diff u\Big)$$ 
provided that the limit exists. When the function is integrable at infinity, one can replace the limit when $M \to \infty$ in the above definition by $\int_{-\infty}^{-\epsilon}$ and  $\int^{\infty}_\epsilon$, respectively.
With this definition, we have 
\begin{equation}\label{cauchy-principal}
\mathrm{p.v.} \int_{-\infty}^{\infty}{1\over u}e^{- iu\xi}\,\diff u =-i\pi \, \sgn (\xi),
\end{equation}
where $\sgn$ is the sign function.  This identity implies the following useful formula, which is related to the Hilbert transform, see e.g.\ \cite{frederick-book}.  

\begin{lemma}\label{f/x}
	Let $\varphi$ be a function in $L^1(\R)$ such that $\widehat\varphi\in L^1(\R)\cap\Cc^1(\R)$.  Then  
	$$\mathrm{p.v.}\int_{-\infty}^{\infty} {\widehat\varphi(\xi)\over \xi} e^{-it\xi} \,\diff \xi =i\pi\int_{-\infty}^{-t} \varphi(u)\,\diff u -i\pi\int_{-t}^{\infty} \varphi(u)\,\diff u  \quad\text{for all}\quad t\in\R.$$
\end{lemma}

\begin{proof}
	By definition of Cauchy principal value, we have
	\begin{align*}
	\mathrm{p.v.}\int_{-\infty}^{\infty} {\widehat\varphi(\xi)\over \xi} e^{-it\xi}\, \diff \xi= \lim_{\epsilon \to 0^+ \atop M\to\infty}  \int_\ep^M {\widehat\varphi(\xi) e^{-it\xi} -\widehat\varphi(-\xi)e^{it\xi}\over \xi} \, \diff \xi .
	\end{align*}
	Since the function $\widehat\varphi(\xi) e^{-it\xi}$ is $\Cc^1$, the  integrand from the last integral is bounded for $0<\xi\leq 1$ by the mean value theorem. When $\xi>1$, the integrand is dominated by $2 |\widehat\varphi(\xi)|$, which is integrable by assumption. Lebesgue's dominated convergence theorem implies that the above limit exists.
	
	By definition of Fourier transform, we have
	\begin{align*} 
	\int_\ep^M {\widehat\varphi(\xi) e^{-it\xi}\over \xi}  \diff \xi=\int_\ep^M {e^{-it\xi}\over \xi} \int_{-\infty}^{\infty} \varphi(u)e^{-iu\xi}   \diff u \diff \xi =\int_{-\infty}^\infty \int_{\ep}^M {  \varphi(u)e^{-i(u+t)\xi} \over \xi}   \diff \xi \diff u,
	\end{align*}
	by  Fubini's theorem. This is is justified by the fact that $\ep$ is fixed and $\varphi \in L^1(\R)$. Similarly,
	$$\int_\ep^M {\widehat\varphi(-\xi) e^{it\xi}\over \xi} \, \diff \xi=\int_{-\infty}^\infty  \int_{\ep}^M {\varphi(u) e^{i(u+t)\xi} \over \xi}  \, \diff \xi \diff u.$$
	It follows that
	\begin{align*}
	\mathrm{p.v.}\int_{-\infty}^{\infty} {\widehat\varphi(\xi)\over \xi} e^{-it\xi}\, \diff \xi&=
	\lim_{\ep\to 0^+ \atop M\to\infty}\int_{-\infty}^\infty  \varphi(u)\int_{\ep}^M {e^{-i(u+t)\xi}- e^{i(u+t)\xi} \over \xi} \, \diff \xi \diff u  \\
	&=\lim_{\ep\to 0^+ \atop M\to\infty}\int_{-\infty}^\infty  \varphi(u)\int_{\ep(u+t)}^{M(u+t)} {-2i\sin(\xi)\over \xi  } \,\diff \xi \diff u .
	\end{align*}
	
	Using an integration by parts, it is easy to see that  $\int_{b_1}^{b_2} {\sin( \xi)\over\xi} \,\diff \xi$ is bounded uniformly in $0 < b_1<b_2 < \infty$. Using this and the fact that the function ${\sin \xi\over\xi}$ is even, we get that $ \big|\varphi(u)\int_{\ep(u+t)}^{M(u+t)} {-2i\sin(\xi)\over \xi  } \diff \xi\big|$	is dominated by $C|\varphi(u)|$ for some constant $C>0$ independent of $\ep,M,u$ and $t$. Then, using  Lebesgue's dominated convergence theorem again and \eqref{cauchy-principal}, we obtain
	$$\mathrm{p.v.}\int_{-\infty}^{\infty} {\widehat\varphi(\xi)\over \xi} e^{-it\xi} \diff \xi= \int_{-\infty}^\infty \varphi (u) \cdot\mathrm{p.v.} \int_{-\infty}^\infty {e^{-i(u+t)\xi} \over \xi} \diff \xi \diff u =-i\pi\int_{-\infty}^\infty \varphi(u) \,\sgn(u+t) \diff u .$$ The desired result follows.
\end{proof}

The following lemma will be used in Subsection \ref{subsec:R}. We denote by $\|\cdot\|_{\Cc^\alpha}$ the $\alpha$-H\"older norm for $0<\alpha < 1$.

\begin{lemma} \label{lemma:fourier-decay-holder}
	Let $f$ be a  function on $\R$ with support in a compact set $K$ and  $M > 0$ be a constant. Assume that $f$ is locally ${1\over 2}$-H\"older continuous on $\R \setminus \{0\}$ with  $\|f\|_{\Cc^{1/2}(K \setminus (-\delta,\delta))} \leq  M \delta^{-q}$ for every $0<\delta< 1$ and some $q>1$. Assume further that $|f(u)| \leq  M |u|^{-1/2}$. Then, there exists a constant $C_K>0$, independent of $f$ and $M$, such that 
	$$|\widehat{f} (\xi)| \leq  C_K M \, (1+|\xi|)^{-{1\over 4q-2}}\quad\text{for every}\quad \xi \in \R.$$
\end{lemma}

\begin{proof}
	We can assume that $M=1$. From the assumption that $|f(u)| \leq |u|^{-1/2}$ and the fact that $f$ is supported by $K$, we see that $f$ is integrable. Hence, $\widehat f(\xi) =\int_{-\infty}^{\infty}f(u)e^{-i u\xi} \diff u$ is a well-defined bounded function. Therefore, it is enough to prove the desired bound for $|\xi|$ large enough.
	
	We may assume that $\xi>0$. The case $\xi<0$ can be treated in the same way. Applying the change of coordinates $u \mapsto u - \pi \slash \xi$ gives $\widehat f(\xi) = -\int_{-\infty}^{\infty}f(u - \pi \slash \xi)e^{-i u\xi} \diff u$. So,  $$\widehat f(\xi)  = \frac12 \int_{-\infty}^{\infty} \Big[ f(u) - f\Big( u - \frac{\pi}{\xi} \Big) \Big] e^{-i u\xi}\, \diff u.$$

 For $\delta >0$ small and $\xi > 0$ large, we have
	\begin{align*}
   \big|\widehat f(\xi)\big| & \leq \frac12 \int_{\R \setminus (-\delta,\delta+{\pi\over\xi})} \Big| f(u) - f\Big( u - \frac{\pi}{\xi} \Big) \Big|  \diff u + \frac12 \int_{-\delta}^{\delta+{\pi\over\xi}}\Big| f(u) - f\Big( u - \frac{\pi}{\xi} \Big) \Big|   \diff u \\
	&\lesssim  \int_{ K \setminus (-\delta,\delta+{\pi\over\xi})}  \delta^{-q}\Big({\pi\over \xi}\Big)^{1/2} \diff u +\int_{-\delta-{\pi\over\xi}}^{\delta+{\pi\over\xi}} |u|^{-1/2} \diff u \lesssim   \delta^{-q}\xi^{-1/2} + \delta^{1/2},
	\end{align*}
	where the implicit constants in the last two steps are independent of $\xi, \delta$ and $f$. Letting $\delta:=\xi^{-{1\over 2q-1}}$ yields $|\widehat f(\xi)| \lesssim  \xi^{-{1\over 4q-2}} \lesssim (1+\xi)^{-{1\over 4q-2}}$ for $\xi$ large. This finishes the proof.
\end{proof}

 We will often have to approximate functions by ones whose Fourier transforms are compactly supported. In this case, convolution with the functions $\vartheta_\delta$ below will give such approximations.

\begin{lemma}[\cite{DKW:LLT}--Lemma 2.2] \label{l:vartheta}
	There is a smooth strictly positive even function $\vartheta$ on $\R$ with $\int_\R \vartheta(u) \diff u=1$ such that its Fourier transform $\widehat\vartheta$ is  a smooth function supported by $[-1,1]$.
	Moreover, for $0<\delta< 1$ and $\vartheta_\delta(u):=\delta^{-2}\vartheta(u/\delta^2)$, we have that $\widehat{\vartheta_\delta}$ is supported by $[-\delta^{-2},\delta^{-2}]$, $\norm{\widehat{\vartheta_\delta}}_{\Cc^1}\leq c$ and $\int_{|u|\geq \delta} \vartheta_\delta (u)\diff u\leq c\delta$ for some constant $c>0$ independent of $\delta$.
\end{lemma}

The following simple result will be used later. Its proof is left to the reader. See \cite[Lemma 4.26]{li:fourier} for a similar estimate.

\begin{lemma} \label{renewal-3-lemma-3}
	Let $\varphi_C(u) :=\mathbf1_{u\in[b_1,b_2]}\cdot\varphi(u)$, where $b_2>b_1$ and $\|\varphi\|_{\Cc^1}\leq 1$. Then, when $\delta \to 0$,
		 \begin{align*}
	\big|\varphi_C*\vartheta_\delta(u)-\varphi_C(u)\big|\lesssim &
	\begin{cases}
	\delta        & \text{for}     \quad     u\in[b_1+\delta,b_2-\delta]                ,\\
	1        & \text{for}   \quad             u\in[b_1-\delta,b_1+\delta]\cup [b_2-\delta,b_2+\delta]                 ,\\
	\mathbf 1_{u\in[b_1,b_2]} *\vartheta_\delta        & \text{for}  \quad u\notin [b_1-\delta,b_2+\delta].
	\end{cases}
	\end{align*}
\end{lemma}

\section{The Markov operator and its perturbations} \label{sec:markov}

In this section, we recall the main properties of the Markov operator and its purely imaginary perturbations obtained in \cite{DKW:PAMQ} and \cite{DKW:LLT} in the low moment case. We also prove some new results about these operators that will be needed later. We refer to \cite{benoist-quint:book,lepage:theoremes-limites,li:fourier-2} for analogous results when the random walk has a finite exponential moment.

\subsection{Markov operator}
The \textit{Markov operator} $\oP$ and its purely imaginary perturbations $\oP_\xi$ are the operators acting on functions on $\P^1$ given by
\begin{equation} \label{eq:makov-op}
\oP u := \int_G g^* u \, \diff \mu(g) \quad \text{and} \quad  \oP_\xi u  (x) :=   \int_G e^{i \xi \sigma_g(x)} g^*u (x) \, \diff \mu(g) \,\, \text{ for } \,\, \xi \in \R,
\end{equation}
where $\sigma_g$ is the norm cocycle.

Notice that $\oP_0= \oP$. A straightforward computation using the cocycle relation $\sigma_{g_2g_1}(x) = \sigma_{g_2}(g_1  x) + \sigma_{g_1}(x)$ shows that $\oP^n_\xi$ corresponds to the perturbed Markov operator associated with the convolution power $\mu^{\ast n}$, that is,
$$\oP^n_\xi u  (x) = \int_G e^{i \xi \sigma_g(x)} g^*u (x) \, \diff \mu^{*n}(g).$$

We will consider the action of the above operators on several spaces of functions on $\P^1$. 

\subsection{Action on $W^{1,2}$} If $\phi$ (resp. $\psi$) is a $(1,0)$-form (resp. $(0,1)$-form ) with $L^2$ coefficients, we define their norms by $$\|\phi\|_{L^2}:= \Big( \int_{\P^1} i \phi \wedge \overline \phi \Big)^{1 \slash 2} \text{ and } \|\psi\|_{L^2}:= \Big( \int_{\P^1} i\, \overline \psi \wedge \psi \Big)^{1 \slash 2}.$$

We denote by $L^2_{(1,0)}$ and $L^2_{(0,1)}$ the space of forms with finite $L^2$-norm. Observe that any $g \in G=\SL_2(\C)$ acts unitarily on these spaces.

 The  \textit{Sobolev space} $W^{1,2}$ is the space of complex-valued measurable functions on $\P^ 1$ with finite  $\|\cdot\|_{W^{1,2}}$-norm, where
$$\|u\|_{W^{1,2}} := \|u\|_{L^2} + \frac12 \|\partial u\|_{L^2}+ \frac12 \|\dbar u\|_{L^2}.$$

As usual, we identify functions that are equal almost everywhere. Here $\|\cdot \|_{L^2}$ stands for the $L^2$-norm with respect to the Fubini-Study form $\omegaFS$ on $\P^1$, the unique unitarily invariant $(1,1)$-form of mass one. The quantity $\|u\|_{L^2}$ in the above definition can be replaced by $\|u\|_{L^1}$ or  $\big| \int_{\P^1}  u \, \omegaFS \big|$ and produce equivalent norms, see \cite[Proposition 2.1]{DKW:PAMQ}.

The main properties of $\oP_\xi$ acting on $W^{1,2}$ are summarized in the following theorem. Recall that the essential spectral radius $\rho_{\rm ess}$ is the radius of the smallest disc centered at the origin containing the spectrum deprived of the isolated eigenvalues of finite multiplicity.

\begin{theorem}[\cite{DKW:PAMQ,DKW:LLT}] \label{thm:spectrum-Pxi-sobolev}
Let $\mu$ be a non-elementary probability measure on $G=\SL_2(\C)$. Assume that $\mu$ has a finite moment of order $p>0$, i.e.,  $\int_G \log^p \|g\| \, \diff \mu(g) <\infty$.
\begin{enumerate}
\item If $p \geq \frac12$, then $\oP_0$ defines a bounded operator on $W^{1,2}$.  Moreover, $\oP_0$ has a spectral gap, that is,  $\rho_{\rm ess}(\oP_0)<1$ and $\oP_0$ has a single eigenvalue of modulus $\geq 1$ located at $1$. It is an isolated eigenvalue of multiplicity one.

\item If $p \geq 1$, then  $\oP_\xi: W^{1,2} \to W^{1,2}$ is a bounded operator for every $\xi\in \R$.

\item If $p \geq 1$, then  the family $\xi \mapsto \oP_\xi$ of bounded operators on $W^{1,2}$ is continuous in $\xi$. If $p \geq 2$ this  family is $\lfloor p-1 \rfloor$-times differentiable and $\frac{\diff^k}{\diff \xi^k} \oP_\xi = \oP^{(k)}_\xi$ for $k=0,1,\ldots, \lfloor p-1 \rfloor$, where
\begin{equation*} \label{eq:def-P_t^k}
\oP^{(k)}_\xi u := \int_G (i  \sigma_g)^k \, e^{i\xi \sigma_g} g^* u \, \diff \mu(g).
\end{equation*}

\item If $p \geq 2$, then, for every compact subset $K \subset \R \setminus \{0\}$, there exist $0 < \delta < 1$ and $C>0$  such that $\norm{\oP_\xi^n}_{W^{1,2}} <C( 1 - \delta)^n$ for every $\xi \in K$ and $n \geq 1$. 
\end{enumerate}
\end{theorem}

The regularity of $\xi \mapsto \oP_\xi$ together with the spectral gap for $\xi=0$ yield a spectral gap for $\oP_\xi$ for small values of $\xi$, see \cite[Corollary 4.5]{DKW:PAMQ} and Corollary \ref{cor:P_t-decomp-W} below. This is a consequence of the general theory of perturbations of linear operators, see \cite{kato:book}.  When $\mu$ has a finite exponential moment, the family $\xi \mapsto \oP_\xi$ acting on some H\" older space is analytic in $\xi$, which simplifies the study considerably. See \cite{benoist-quint:book,bougerol-lacroix,lepage:theoremes-limites}.

Part (4) of the above theorem implies that the spectral radius of $\oP_\xi$  is strictly smaller than one for \textbf{all non-zero values} of $\xi$. This is a strong property which is related to the absence of some strong invariance of the cocycle $\sigma(g,x)$. This is a central ingredient in the proof of the Local Limit Theorem for $\sigma(g,x)$ in \cite{DKW:LLT} and will also be indispensable here.

\medskip

As a consequence of the spectral gap for $\oP_0$, we get the following equidistribution theorem. When $p > 1$, this is Theorem 1.2 in \cite{DKW:PAMQ}. The fact that the result still holds when $p > 1 \slash 2$ was observed in \cite[Remark 3.13]{DKW:LLT}.

\begin{theorem} \label{thm:equi-dis}
		Let $\mu$ be a non-elementary probability measure on $G=\SL_2(\C)$. Assume that $\int_G \log^p \|g\| \, \diff \mu(g) <\infty$ for some $p > 1 \slash 2$.  Let $\nu$ be the associated stationary measure. Then, for every $u \in \Cc^1(\P^1)$, we have 
	$$\Big\| \oP^n u - \int_{\P^1} u \,\diff \nu\Big\|_\infty \leq C \lambda^n\norm{u}_{\Cc^1}$$ for some constants $C>0$ and $0<\lambda<1$ independent of $u$.
\end{theorem}

The following regularity property of the stationary measure will be used later on. It is a direct consequence of \cite[Proposition 4.5]{benoist-quint:CLT}. See also \cite[Corollary 5.9]{DKW:PAMQ}.

\begin{proposition}  \label{regularity}
	Let $\mu$ be a non-elementary probability measure on $G=\SL_2(\C)$. Assume that $\int_G \log^p \|g\| \, \diff \mu(g) <\infty$ for some $p >1$. Let $\nu$ be the associated stationary measure. Then, there is a constant $C>0$ such that  $$\nu\big(\D(x,r)\big) \leq C/|\log r|^{p-1} \quad\text{for every} \,\, x\in \P^1 \, \text{ and every } \,\, 0< r < 1 .$$
\end{proposition}

\subsection{Action on $\Cc^{\log^p}$}

Let $u$ be a function on $\P^1$. Denote $\log^\star t: = 1 + |\log t \, |$ for $t > 0$ and let $d$ be the distance defined in \eqref{eq:distance-def}. Define for $p>0$,
$$\|u\|_{\log^p}: = \|u\|_{\infty} + [u]_{\log^p},\quad \text{where} \quad [u]_{\log^p} := \sup_{x \neq y \in \P^1} \big|u(x) - u(y)\big| \big(\log^\star d(x,y)\big)^p.$$
We say that $u$ is \textit{$\log^p$-continuous} if $\|u\|_{\log^p} <  \infty$ and denote by $\Cc^{\log^p}$ the space of $\log^p$-continuous functions. The space $\Cc^{\log^p}$ is closed under multiplication and it is easy to see that 
\begin{equation} \label{eq:log^p-product}
[uv]_{\log^p} \leq \|u\|_\infty [v]_{\log^p} + [u]_{\log^p} \|v\|_\infty \quad \text{and} \quad \|uv\|_{\log^p} \leq \|u\|_\infty \|v\|_{\log^p} + \|u\|_{\log^p} \|v\|_\infty
\end{equation}
for all $u,v \in \Cc^{\log^p}$. See \cite[Lemma 3.7]{DKW:LLT} for a proof.

It follows directly from the definition of $[\, \cdot \,]_{\log^p}$ and the mean value theorem that, if $u \in \Cc^{\log^p}$ and  $\chi: \C \to \C$  a $\Cc^1$-function, then
\begin{equation}  \label{eq:C^1-log^p}
[\chi(u)]_{\log^p} \leq \|\diff \chi\|_\infty [u]_{\log^p}.
\end{equation}
Observe that the quantity $\|\diff \chi\|_\infty$ above can be replaced by the supremum of $|\diff \chi|$ over the range of the function $u$.

Concerning the action of the operators $\oP_\xi$, we have the following result.  

\begin{proposition}[\cite{DKW:LLT}--Propositions 3.9 and 4.9] \label{prop:P_t-logp}
Let $\mu$ be a non-elementary probability measure on $G=\SL_2(\C)$. Let $p > 1$ and assume that  $M_{p}(\mu):= \int_G \log^{p} \|g\| \, \diff \mu(g)$ is finite. Then,  $$\|\oP u\|_{\log^{p-1}} \leq c \big(1+M_{p-1}(\mu)\big) \|u\|_{\log^{p-1}} \quad \text{and}  \quad \|\oP_\xi u\|_{\log^{p-1}} \leq c \, (1 + |\xi|) M_{p}(\mu) \|u\|_{\log^{p-1}}$$ for some some constant $c>0$ independent of $u,\xi$ and $\mu$.  In particular, for every $\xi \in \R$, the operator $\oP_\xi$ acts continuously on $\Cc^{\log^{p-1}}$.  Moreover, for every $\xi \in \R$ and every $n \geq 1$, we have $$\|\oP^n_\xi\|_{\log^{p-1}} \leq c \, ( 1 + |\xi|) \, n^{p} \, M_{p}(\mu).$$
\end{proposition}

The proofs  of Theorem \ref{thm:spectrum-Pxi-sobolev}-(2)-(3) and Proposition \ref{prop:P_t-logp} make use of some basic estimates on the norm cocycle and on the action of elements of $\SL_2(\C)$ with respect to the norms introduced above. These estimates will also be useful here, so we state them below. See \cite[Lemmas 3.8 and 4.1]{DKW:LLT} for a proof.

\begin{lemma} \label{lemma:sigma-estimates}
There exists a constant $C > 0$ such that the following estimates hold for any $g \in \SL_2(\C)$:
\begin{itemize}
\item[(1)] $\|\del \sigma_g \|_{L^2} \leq C \,  (1+\log^{1 \slash 2} \|g\|)$;

\item[(2)] For $u \in W^{1,2}$, we have $\|g^* u \, \del \sigma_g \|_{L^2}  \leq C \,   (1 + \log \|g\|) \cdot \|u\|_{W^{1,2}}$.
\end{itemize}

Moreover, for $p > 1$, there is a constant $C_p >0$ such that
\begin{itemize}
\item[(3)] $[\sigma_g]_{\log^{p-1}} \leq C_p \, (1+\log^{p} \|g\|)$;

\item[(4)] $[g^* u]_{\log^{p-1}} \leq C_p (1 + \log^{p-1} \|g\|) [u]_{\log^{p-1}}$.
\end{itemize}
\end{lemma}

Concerning the regularity of the family $\xi \mapsto \oP_\xi$ acting on $\Cc^{\log^{p-1}}$ we have the following result, which improves \cite[Proposition 4.11]{DKW:LLT}.

\begin{proposition} \label{prop:P_t-logp-regularity}
Let $p > 1$  and assume that $\int_G \log^p \|g\| \, \diff \mu(g) < \infty $. Then, the family $\xi \mapsto \oP_\xi$ acting on $\Cc^{\log^{p-1}}$ is locally $\frac12$-H\"older continuous on $\R$.
\end{proposition}

\begin{proof}
	It is enough to work with $\xi$ in a fixed bounded interval $J$ of $\R$. We need to show that $\|(\oP_\xi - \oP_\eta) u\|_{\log^{p-1}} \lesssim |\xi- \eta|^{1 \slash 2}$ for $\xi, \eta \in J$ uniformly in $u$ in the unit ball of $\Cc^ {\log^{p-1}}$. By definition, we have $$(\oP_\xi - \oP_\eta) u  = \int_G \big(e^{i \xi \sigma_g } - e^{i \eta \sigma_g} \big)  g^*u  \,  \diff \mu(g).$$
	
	From now on, assume that $\|u\|_{\log^{p-1}}=\norm{u}_{\infty}+[u]_{\log^{p-1}} \leq 1$. Using that $\big|\frac{\diff}{\diff \xi} e^{i\xi \sigma_g} \big| \leq |\sigma_g| \leq \log \|g\|$, we have $$\big\|\big(e^{i\xi \sigma_g} - e^{i\eta \sigma_g}\big) g^*u \big\|_\infty \leq \|e^{i\xi \sigma_g} - e^{i\eta \sigma_g}\|_\infty \leq |\xi- \eta|\cdot \log \|g\|.$$ 
 In particular, the last quantity is $\lesssim  |\xi- \eta|^{1/2} \cdot \log \|g\|$. It remains to bound $[(e^{i\xi \sigma_g} - e^{i\eta \sigma_g}) g^*u]_{\log^{p-1}}$. From \eqref{eq:log^p-product}
and the fact that $\|g^* u\|_\infty \leq 1$, we get $$[(e^{i\xi \sigma_g} - e^{i\eta \sigma_g}) g^*u]_{\log^{p-1}} \leq [(e^{i\xi \sigma_g} - e^{i\eta \sigma_g})]_{\log^{p-1}} + \|e^{i\xi \sigma_g} - e^{i\eta \sigma_g}\|_\infty [g^*u]_{\log^{p-1}}.$$
		As above $\|e^{i\xi \sigma_g} - e^{i\eta \sigma_g}\|_\infty \leq |\xi- \eta|\cdot \log \|g\|$. Also, $[g^*u]_{\log^{p-1}} \lesssim 1+\log^{p-1}\|g\|$  by Lemma \ref{lemma:sigma-estimates}-(4). Therefore, the second term above is bounded by a constant times $$|\xi- \eta| \cdot (1+\log^p \|g\|) \lesssim |\xi- \eta|^{1/2} \cdot (1+\log^p \|g\|).$$
	
	We now estimate the quantity $[(e^{i\xi \sigma_g} - e^{i\eta \sigma_g})]_{\log^{p-1}}$. 
		For fixed $g \in G$ and $x,y \in \P^1$,  consider the function $$\phi(\xi) := e^{i\xi \sigma_g(x)} - e^{i\xi \sigma_g(y)}.$$
	Let $q >0$ and $0< \alpha \leq 1$ be numbers that will be fixed later. 
	We'll use the interpolation inequality $\|\phi\|_{\Cc^\alpha} \leq C_\alpha \|\phi\|_{\Cc^0}^{1 - \alpha} \|\phi\|_{\Cc^1}^\alpha$, where $C_\alpha>0$ is a constant independent of $\phi$, see for instance \cite{triebel}.

	We have $\phi'(\xi) = i\sigma_g(x) e^{i\xi \sigma_g(x)} -i\sigma_g(y) e^{i\xi \sigma_g(y)}$. Recall that $[\sigma_g]_{\log^q} \lesssim 1+\log^{q+1} \|g\|$ from Lemma \ref{lemma:sigma-estimates}-(3). Observe that $\sigma_g$ takes values in the interval $I_g:=[- \log \|g\|, \log \|g\|]$. Applying \eqref{eq:C^1-log^p} to the function $\chi(s) = s e^{i \xi s}$ for $s \in I_g$ and noting that $|\diff \chi| \lesssim 1 + \log \|g\|$ on $I_g$ (recall that $\xi$ belongs to a bounded interval), we obtain $[\sigma_g e^{i\xi\sigma_g}]_{\log^q} \lesssim 1+\log^{q+2} \|g\|$. We deduce that $$|\phi'(\xi)| \leq c_1 (1 + \log^{q+2} \|g\|) (\log^\star d(x,y))^{-q}$$
	for some constant $c_1 > 0$. Therefore, $\|\phi\|_{\Cc^1} \leq c_2 (1 + \log^{q+2} \|g\|) (\log^\star d(x,y))^{-q}$ for some constant $c_2 > 0$, since $\|\phi\|_{\Cc^0} \leq 2$. Inserting this in the interpolation inequality and using again that $\|\phi\|_{\Cc^0} \leq 2$ yields $$\|\phi\|_{\Cc^\alpha} \leq c_3 (1 + \log^{q+2} \|g\|)^ \alpha (\log^\star d(x,y))^{- \alpha q}$$
	 for some constant $c_3>0$.
	We conclude that
	\begin{align*}
	\big| \big( e^{i\xi \sigma_g(x)}  - e^{i\xi \sigma_g(y)} \big) - \big( e^{i\eta \sigma_g(x)} - e^{i\eta \sigma_g(y)} \big) \big|   = \big| \phi(\xi) - \phi(\eta)  \big| \leq \|\phi\|_{\Cc^\alpha}  |\xi-\eta|^\alpha \\  \leq  c_3 (1 + \log^{q+2} \|g\|)^ \alpha  (\log^\star d(x,y))^{- \alpha q} |\xi-\eta|^\alpha.
	\end{align*} 
	
	Since $x$ and $y$ are arbitrary, we get that
	\begin{equation*}  
	[(e^{i\xi \sigma_g} - e^{i\eta \sigma_g})]_{\log^{\alpha q}} \leq c_3 (1 + \log^{\alpha(q+2)} \|g\|) |\xi-\eta|^\alpha.
	\end{equation*}
	
	Now, set $\alpha := 1 \slash 2$  and $q:= 2(p-1)$, so that $\alpha q = p-1$ and $\alpha(q+2) = p$. Then, the preceding estimate gives 
	$$[(e^{i\xi \sigma_g} - e^{i\eta \sigma_g})]_{\log^{p-1}} \leq c_3 (1 + \log^p \|g\|) |\xi-\eta|^{1 \slash 2}.$$
	
	Gathering all the above estimates, we obtain a constant $c_4 > 0$ independent of $u$, $g$ and $\xi,\eta \in J$ such that $$\big\|\big(e^{i \xi \sigma_g } - e^{i \eta \sigma_g} \big)  g^*u\big\|_{\log^{p-1}} \leq c_4 (1 + \log^p \|g\|) |\xi-\eta|^{1 \slash 2}.$$
	We conclude by the triangle inequality that $$\|(\oP_\xi - \oP_\eta) u\|_{\log^{p-1}} \leq c_4(1+M_p(\mu)) |\xi- \eta|^{1 \slash 2}$$ uniformly in $u$ in the unit ball of $\Cc^{\log^{p-1}}$. Recall that $M_{p}(\mu) = \int_G \log^{p} \|g\| \, \diff \mu(g)$, which is finite by assumption. We conclude that $\oP_\xi : \Cc^{\log^{p-1}}  \to \Cc^{\log^{p-1}}$ is locally $\frac12$-H\"older continuous in $\xi$. This finishes the proof of the proposition.
\end{proof}

\subsection{The space $\W$} 

In \cite{DKW:LLT}, a new function space adapted to our situation was introduced. This  space mixes the norms from $W^{1,2}$ and $\Cc^{\log^{p-1}}$, and its definition takes into account the spectral gap of $\oP_0$ on $W^{1,2}$. The advantage of working with $\W$ is that its norm automatically yields uniform pointwise estimates. See \cite[Section 6]{DKW:LLT} for the precise definition of the norm $\| \cdot \|_\W$ and more details. 

\begin{theorem} \label{thm:spectral-gap-W}
Let $\mu$ be a non-elementary probability measure on $G =\SL_2(\C)$. Assume that $\int_G \log^p \|g\| \, \diff \mu(g) <\infty$, where $p > 3 \slash 2$. Then, there exists a Banach space $( \W, \|  \cdot \|_\W)$ of functions on $\P^1$ with the following properties: 

\begin{enumerate}
\item The following inequalities hold: $\| \cdot \|_{W^{1,2}} \lesssim \| \cdot \|_\W $, $\| \cdot \|_{\infty} \lesssim \| \cdot \|_\W $ and  $ \| \cdot \|_\W \lesssim \max \lbrace \| \cdot \|_{W^{1,2}}, \| \cdot \|_{\log^{p-1}} \rbrace$. In particular, we have continuous embeddings $$ W^ {1,2} \cap \Cc^{\log^{p-1}} \subset\W \subset W^ {1,2}\cap \Cc^0;$$

\item  If an operator $\mathcal T$ acts continuously both on $W^{1,2}$ and $\Cc^{\log^{p-1}}$, then  it also acts continuously on $\W$ and $\| \mathcal T\|_\W \leq \max \big\{ \| \mathcal T\|_{W^{1,2}}, \| \mathcal T\|_{\log^{p-1}} \big\}$;

\item The operator $\oP_0$ has a spectral gap on $\W$ as in Theorem \ref{thm:spectrum-Pxi-sobolev}-(1);
\item For every $\xi \in \R$, $\oP_\xi$ defines a bounded operator on $\W$; 
\item If moreover $p \geq 2$, then the family $\xi \mapsto \oP_\xi$ acting on  $\W$ is locally $\frac12$-H\" older continuous.
\end{enumerate}
\end{theorem}

\begin{proof}
Statements (1)--(4) are contained in \cite[Theorem 6.1]{DKW:LLT}. They all follow directly from the definition $\W$, with the exception of the inclusion $\W \subset W^ {1,2}\cap \Cc^0$, which is a consequence of Lemma \ref{lemma:logp-pre-equidistribution} below. Therefore, we only need to prove (5).  We need to show that $\| \oP_\xi - \oP_\eta \|_\W \lesssim |\xi-\eta|^{1 \slash 2}$ whenever $\xi,\eta$ belong to a bounded interval in $\R$. In view of (2), this will follow from the analogous property for $\| \oP_\xi - \oP_\eta \|_{W^{1,2}}$ and $\| \oP_\xi - \oP_\eta \|_{\log^{p-1}}$. The former follows from Theorem \ref{thm:spectrum-Pxi-sobolev}-(3), while the latter is the content of Proposition \ref{prop:P_t-logp-regularity}. This proves (5).
\end{proof}

As mentioned before, the regularity of the family $\xi \mapsto \oP_\xi$ acting on $\W$ provided by the above theorem allows us to apply the theory of perturbations of linear operators (see \cite{kato:book}), yielding the following decomposition of $\oP_\xi$ on $\W$ for $\xi$ near zero. When $\mu$ has a finite exponential moment a similar decomposition holds on some H\"older space, see \cite{benoist-quint:book,bougerol-lacroix,lepage:theoremes-limites}.

\begin{corollary} \label{cor:P_t-decomp-W}
Let $\mu$ be a non-elementary probability measure on $G =\SL_2(\C)$. Assume that $\int_G \log^p \|g\| \, \diff \mu(g) <\infty$ for some $p \geq 2$. Let $\W$ be as in Theorem \ref{thm:spectral-gap-W}. Then there exists an $\epsilon_0 > 0$ such that, for $\xi \in [-\epsilon_0,\epsilon_0]$, one has a decomposition
\begin{equation} \label{eq:P_t-decomp}
\oP_\xi = \lambda_\xi \oN_\xi + \oQ_\xi,
\end{equation}
where $\lambda_\xi \in \C$, and $\oN_\xi,\oQ_\xi$ are bounded operators on $\W$ with the following properties: 
\begin{enumerate}
\item $\lambda_0 = 1$ and $\oN_0 u = \int u \, \diff \nu$, where $\nu$ is the unique $\mu$-stationary measure;
\item $\rho:= \displaystyle \lim_{n \to \infty } \|\oP_0^n - \oN_0\|_\W^{1 \slash n} < 1$;

\item $\lambda_\xi$ is the unique eigenvalue of maximum modulus of $\oP_\xi$ in $\W$, $\oN_\xi$ is a rank-one projection and $\oN_\xi \oQ_\xi = \oQ_\xi \oN_\xi = 0$;

\item the maps  $\xi \mapsto \oN_\xi$ and $\xi \mapsto \oQ_\xi$ are locally $\frac12$-H\"older continuous;

\item  $\frac{2 + \rho}{3}<|\lambda_\xi| \leq 1 $ and there exists a constant $c > 0$ such that $\|\oQ_\xi^n\|_\W \leq c \big( \frac{1 + 2 \rho}{3} \big)^n$  for every  $n \geq 0$;

\item The map $\xi \mapsto \lambda_\xi$ is differentiable and has an expansion
\begin{equation} \label{eq:lambda-expansion-2}
\lambda_\xi = 1 + i  \gamma \xi - (a^2 + \gamma^2) \frac{\xi^2}{2}+ \psi(\xi) 
\end{equation}
for a  continuous function $\psi$  such that $
\lim_{\xi\to 0 }\xi^{-2}\psi(\xi) =0$ and a number $a > 0$, where $\gamma$ is the Lyapunov exponent of $\mu$.
\end{enumerate}
\end{corollary}

\begin{proof}
All statements except (4) are included in \cite[Corollary 6.4]{DKW:LLT}, so we only need to prove (4). The operator $\oN_\xi$ is given by the integral $(2\pi i)^{-1} \int_\Gamma R_\xi(z) \diff z$ where $\Gamma$ is a small circle enclosing $1$ and $R_\xi(z) = (z - \oP_\xi)^{-1}$ is the resolvent of $\oP_\xi$, see \cite{kato:book}. Since $\xi \mapsto \oP_\xi$ is locally $\frac12$-H\" older continuous by Theorem \ref{thm:spectral-gap-W} and the resolvent of an operator $\mathcal T$ depends analytically in $\mathcal T$ (cf. \cite[\S 4.3.3]{kato:book}), it follows that $\xi \mapsto \oN_\xi$ is locally $\frac12$-H\" older continuous. Finally, from the fact that $\oQ_\xi = \oP_\xi - \lambda_\xi \oN_\xi$ and (6), the H\" older regularity of $\xi \mapsto \oQ_\xi$ follows. This finishes the proof.
\end{proof}

\begin{remark}
The number $a > 0$ appearing in the expansion \eqref{eq:lambda-expansion-2} coincides with the variance in the CLT for the norm cocycle. See e.g.\ \cite[Lemma 4.7]{DKW:LLT}.
\end{remark}

\subsection{Asymptotic behavior of Markov operator}
 The next result concerns the behavior of the series $\sum_{n=0}^\infty \oP_\xi^n$ as $\xi$ approaches zero. This asymptotic behavior  will be central in the study of the renewal operators in Section \ref{sec:renewal}. We equip $W^{1,2} \cap \Cc^{\log^{p-1}}$ with the norm $\max \big\{\|\cdot\|_{W^{1,2}},\|\cdot\|_{\log^{p-1}}\big\}$ and  $W^{1,2} \cap \Cc^0$ with the norm $\max \big\{\|\cdot\|_{W^{1,2}},\|\cdot\|_{\infty}\big\}$.

\begin{proposition}\label{prop:sum-Pxi-n-decomposition}
Let $\mu$ be a non-elementary probability measure on $G =\SL_2(\C)$. Assume that $\int_G \log^p \|g\| \, \diff \mu(g) <\infty$ for some $p \geq 2$.  Then, for every $\xi \neq 0$ there is a bounded operator $\oU_\xi: W^{1,2} \cap \Cc^{\log^{p-1}} \to W^{1,2} \cap \Cc^0 $ such that
 $$\sum_{n=0}^\infty \oP_\xi^n={\oN_0 \over -i\gamma\xi} +\oU_\xi \quad \text{and} \quad \|\oU_\xi\|_{W^{1,2} \cap \Cc^{\log^{p-1}} \to W^{1,2} \cap \Cc^0} = O(|\xi|^{- 1 \slash 2}) \quad \text{as} \quad \xi \to 0.$$
Moreover, the family of operators $\xi \mapsto \oU_\xi$ from  $W^{1,2} \cap\Cc^{\log^{p-1}}$ to $W^{1,2} \cap \Cc^0$ is locally $\frac12$-H\"older continuous  on $\R \setminus \{0\}$.
\end{proposition}

When $\mu$ has a finite exponential moment an analogous decomposition holds in some H\"older space. In this case, the situation is considerably better and the family $\oU_\xi$ extends analytically through $\xi = 0$, see for instance \cite{boyer,li:fourier,li-sahlsten}. The weak control  from Proposition \ref{prop:sum-Pxi-n-decomposition} makes our analysis more involved.

In the proof, we will use the following lemma contained in \cite{DKW:LLT}. It is a consequence of Moser-Trudinger's estimate in $W^{1,2}$.

\begin{lemma} \label{lemma:logp-pre-equidistribution}
Let $p > 3 \slash 2$. Let $u_n$, $n \geq 1$, be a sequence of continuous functions. Assume that there are constants $c_1,c_2 > 0$, $0<\lambda<1$ and $q > 0$ such that $\|u_n\|_{W^{1,2}} \leq c_1 \lambda^n$ and $[u_n]_{\log^{p-1}} \leq c_2 n^q$. Then, there are constants $c > 0$ and $0< \tau < 1$ depending only on $c_1,c_2,\lambda,p$ and $q$ such that  $$\|u_n \|_\infty \leq c \tau^n \quad \text{ for every } n \geq 1.$$ 
\end{lemma}

\begin{proof}[Proof of Proposition \ref{prop:sum-Pxi-n-decomposition}]
Fix $\xi \neq 0$. Let $u$ be a function in the unit ball of $W^{1,2} \cap \Cc^{\log^{p-1}}$. From Theorem \ref{thm:spectrum-Pxi-sobolev}-(4), there is a number $0< \rho_\xi < 1$ and a constant $c_1 >0$ such that $\|\oP_\xi^n u \|_{W^{1,2}} \leq c_1 \rho_\xi^n$ and Proposition \ref{prop:P_t-logp} gives that $\|\oP_\xi^n u \|_{\log^{p-1}} \leq c_2 n^p$ for some constant $c_2>0$. It follows from Lemma \ref{lemma:logp-pre-equidistribution} that  $\|\oP_\xi^n u \|_{\infty} \leq c_\xi \tau_\xi^n$ for some $0< \tau_\xi < 1$ and $c_{\xi} >0$. These estimates imply that the series $\sum_{n=0}^\infty \oP_\xi^n u$ converges in $W^{1,2} \cap \Cc^0$ and that the operator defined by $\oU_\xi : =  \sum_{n=0}^\infty \oP_\xi^n  - {\oN_0\over -i\gamma\xi}$ is bounded from $W^{1,2} \cap \Cc^{\log^{p-1}}$ to $W^{1,2} \cap \Cc^0$.

\vskip3pt

Let us show that $\|\oU_\xi\|_{W^{1,2} \cap \Cc^{\log^{p-1}} \to W^{1,2} \cap \Cc^0}$ is   $O(|\xi|^{-1 \slash 2})$ as  $\xi \to 0$. Let $u$ be as above. Recall from Theorem \ref{thm:spectral-gap-W} that there are continuous embeddings $W^ {1,2} \cap \Cc^{\log^{p-1}} \subset\W \subset W^ {1,2}\cap \Cc^0$. In particular, $u$ has finite $\W$-norm. Using the decomposition $\oP_\xi=\lambda_\xi \oN_\xi+\oQ_\xi$ for $\xi$ near zero as in \eqref{eq:P_t-decomp}, one has
 	$$ \sum_{n=0}^\infty \oP_\xi^n u =\sum_{n=0}^\infty \big( \lambda_\xi^n \oN_\xi u  +\oQ_\xi^n u\big) ={\oN_\xi u \over 1-\lambda_\xi} +\sum_{n=0}^\infty \oQ_\xi^n u.$$
 	 
 	We have that both $\|\oQ_\xi^n u \|_\infty$ and $\|\oQ_\xi^n u \|_{W^{1,2}}$ are $\lesssim \|\oQ_\xi^n u \|_\W$ and the last quantity is exponentially small in $n$, uniformly in $u$ by Corollary \ref{cor:P_t-decomp-W}-(5). Therefore, the operator $\sum_{n=0}^\infty \oQ_\xi^n$ from  $W^{1,2} \cap \Cc^{\log^{p-1}}$ to $W^{1,2} \cap \Cc^0$  is well-defined and its norm is uniformly bounded as $\xi \to 0$. Therefore,  the desired bound will follow if we show that  
 	\begin{equation}\label{N-xi-N-0}
 	\Big\| \frac{\oN_\xi}{1-\lambda_\xi}  - \frac{\oN_0}{-i\gamma\xi} \Big \|_{W^{1,2} \cap \Cc^{\log^{p-1}} \to W^{1,2} \cap \Cc^0} = O(|\xi|^{-1 \slash 2}) \quad \text{as } \,\, \xi \to 0.
 	\end{equation} 

 From  Corollary \ref{cor:P_t-decomp-W}-(4), we can write $\oN_\xi =\oN_0 + \mathcal M_\xi$ with  $\|\mathcal M_\xi\|_\W = O(|\xi|^{1 \slash 2})$ and by (\ref{eq:lambda-expansion-2}), we have $\lambda_\xi =1+i\gamma \xi +O(|\xi|^2)$. This gives
 	$$ {\oN_\xi\over 1-\lambda_\xi} - {\oN_0 \over -i\gamma\xi}={\oN_0 + \mathcal M_\xi \over -i\gamma \xi+O(|\xi|^2) }+ {\oN_0\over i\gamma\xi}={i\gamma \xi \mathcal M_\xi + O(|\xi|^2) \oN_0 \over \gamma^2 \xi^2+O(|\xi|^3) } =  \frac{\mathcal M_\xi}{-i \gamma \xi + O(|\xi|^2)} + O(1) \oN_0 .$$
 	As $\gamma>0$ and $\|\mathcal M_\xi\|_{\W} = O(|\xi|^{1 \slash 2})$, we conclude that the $\W$-norm of the above operator is $O(|\xi|^{-1 \slash 2})$. From the inclusions $ W^ {1,2} \cap \Cc^{\log^{p-1}} \subset\W \subset W^ {1,2}\cap \Cc^0$, the same estimate holds true for the norm $\| \cdot \|_{W^{1,2} \cap \Cc^{\log^{p-1}} \to W^{1,2} \cap \Cc^0}$,   thus giving \eqref{N-xi-N-0}.
 	
 	\vskip3pt
 	
 	We now show that $\xi \mapsto \oU_\xi$ is locally $\frac12$-H\"older continuous in $\R \setminus \{0\}$.  By the definition of $\oU_\xi$, it is enough to show the same property for $\sum_{n=0}^\infty \oP_\xi^n$ instead. Let $\xi, \eta$ belong to a fixed compact subset $K$ of $\R \setminus \{0\}$. For every $n \geq 1$, we can write $$ \oP_\xi^n - \oP_\eta^n = \sum_{k=0}^{n-1} \oP_\xi^{n-k-1}(\oP_\xi - \oP_\eta)\oP_\eta^k.$$
 	
 	From Theorem \ref{thm:spectrum-Pxi-sobolev}-(3), we have that $\|\oP_\xi - \oP_\eta\|_{W^{1,2}} \lesssim |\xi - \eta|$ and from Theorem  \ref{thm:spectrum-Pxi-sobolev}-(4), one has $\|\oP_\zeta^\ell\|_{W^{1,2}} \leq c_1 \rho^\ell$ for some  $0< \rho < 1$ independent of $\zeta \in K$. This gives, for $u$ in the unit ball of $W^{1,2} \cap \Cc^{\log^{p-1}}$,
 	 $$\| \oP_\xi^n u - \oP_\eta^n u \|_{W^{1,2}} \leq \sum_{k=0}^{n-1} \rho^{n-k-1} |\xi-\eta|\rho^k = |\xi - \eta | \sum_{k=0}^{n-1} \rho^{n-1} =n \rho^{n-1} |\xi - \eta | \lesssim \theta^n \, |\xi - \eta|$$ 
 	 for any $\rho<\theta<1$. 
 	Proposition \ref{prop:P_t-logp-regularity} gives that $\|\oP_\xi - \oP_\eta\|_{\log^{p-1}} \lesssim |\xi - \eta|^{1 \slash 2}$ and, by Proposition \ref{prop:P_t-logp}, one has $\|\oP_\zeta^\ell\|_{\log^{p-1}} \leq c_3 \ell^{p}$ for some constant $c_3>0$ independent of of $\zeta \in K$. Therefore, 
 	$$\| \oP_\xi^n u - \oP_\eta^n u\|_{\log^{p-1}} \lesssim |\xi - \eta|^{1 \slash 2} \sum_{k=0}^{n-1}  (n-k-1)^{p} k^{p}\leq |\xi - \eta|^{1/2} n^{2p+1}.$$ 
  Coupling with the above estimate on the Sobolev norm and using Lemma \ref{lemma:logp-pre-equidistribution}, we conclude that 
  $$\| \oP_\xi^n u - \oP_\eta^n u \|_{\infty} \leq c \lambda^n |\xi - \eta|^{1 \slash 2}$$ 
  for some $c>0$ and $0<\lambda<1$ independent of $u$. By the last estimate and triangle inequality,  we conclude that $|\xi - \eta|^ {-1 \slash 2} \big\| \sum_{n=0}^\infty  \oP_\xi^n u - \sum_{n=0}^\infty  \oP_\eta^n u \big\|_{\infty}$ is bounded in the unit ball of $W^{1,2} \cap \Cc^{\log^{p-1}}$ uniformly in $\xi,\eta$. This implies that  $ \xi \mapsto \sum_{n=0}^\infty \oP_\xi^n$ is a locally $\frac12$-H\"older continuous family of bounded operators from  $W^{1,2} \cap \Cc^{\log^{p-1}}$ to $W^{1,2} \cap \Cc^0$. This finishes the proof of the proposition.
\end{proof}

We now estimate the $\Cc^{1/2}$-norm of $\xi \mapsto \oU_\xi$ near $\xi = 0$. Fix a compact set $K \subset \R$. For $\delta >0$ small enough, set 
$$H_K(\delta):=\sup_{\xi\neq\eta\in K \setminus (-\delta,\delta)}  { \norm{\oU_\xi-\oU_\eta}_{W^{1,2} \cap \Cc^{\log^{p-1}} \to W^{1,2} \cap \Cc^0}  \over |\xi-\eta|^{1 \slash 2}}.$$
By Proposition \ref{prop:sum-Pxi-n-decomposition},  $H_K(\delta)$ is finite for all $\delta\neq 0$.

\begin{lemma}\label{holder-delta-2}
	Under the same assumptions of Proposition \ref{prop:sum-Pxi-n-decomposition}, we have that $H_K(\delta)\leq C \delta^{-3/2}$ for some constant $C>0$ independent of $\delta$.
\end{lemma}

\begin{proof} 
	It is enough to consider $\delta$ small. Since we want an upper bound for $H_K(\delta)$ and the family $\xi \mapsto U_\xi$ is locally $\frac12$-H\"older continuous away from the origin, it is enough to assume that $\xi,\eta\in (-\ep_0,\ep_0)\setminus (-\delta,\delta)$ and $K \supset (-\ep_0,\ep_0)$.  By Corollary \ref{cor:P_t-decomp-W}, there is a decomposition  $\oP_\xi=\lambda_\xi \oN_\xi+\oQ_\xi$ for $|\xi| \leq \ep_0$ as operator in $\W$ and the families $\xi\mapsto\oN_\xi,\xi\mapsto \oQ_\xi$ acting on $\W$ are locally $\frac12$-H\"older continuous. From the embeddings $W^ {1,2} \cap \Cc^{\log^{p-1}} \subset\W \subset W^ {1,2}\cap \Cc^0$, the same is true when seeing them as  families of bounded operators from  $W^{1,2} \cap \Cc^{\log^{p-1}}$ to $W^{1,2} \cap \Cc^0$. We may assume that $\delta < \ep_0 \slash 2$.
	
	  From the proof of Proposition \ref{prop:sum-Pxi-n-decomposition}, we have
	$$  \oU_\xi   = {\oN_0 \over i\gamma\xi}+ {\oN_\xi  \over 1-\lambda_\xi} +\sum_{n=0}^\infty \oQ_\xi^n   \quad \text{for } \,\, 0< |\xi| <\ep_0  . $$
	
	Observe that, since the  $\Cc^{1/2}$-norm of the function $t\mapsto 1/t$ on  $K\setminus (-\delta,\delta)$ is $O(\delta^{-3/2})$,  the  $\Cc^{1/2}$-norm of $\xi\mapsto\oN_0 /( i\gamma\xi)$ on $K\setminus (-\delta,\delta)$ is  $O(\delta^{-3/2})$. Therefore, we only need to handle the last two terms in the above expression for $\oU_\xi$.

For the last term, we can write $$\oQ_\xi^n - \oQ_\eta^n = \sum_{k=0}^{n-1} \oQ_\xi^{n-k-1}(\oQ_\xi - \oQ_\eta)\oQ_\eta^k.$$ Recall, from Corollary \ref{cor:P_t-decomp-W}-(5),  that  $\|\oQ_\xi^n\|_{\mathscr W}$ is exponentially small, when $n \to \infty$, uniformly in $\xi$. Arguing as in the proof of Proposition \ref{prop:sum-Pxi-n-decomposition}, we obtain that	$\norm{\oQ_\xi^n - \oQ_\eta^n}_{\mathscr W}\leq c \theta^n|\xi-\eta|^{1 \slash 2} $ for some $c>0$ and $0<\theta<1$ independent of $\xi,\eta$.  Hence  $$\Big\| \sum_{n=0}^\infty \oQ_\xi^n-\sum_{n=0}^\infty \oQ_\eta^n \Big\|_{\mathscr W} \leq {c\over 1-\theta} |\xi-\eta|^{1 \slash 2}$$ and, from the inclusions $W^ {1,2} \cap \Cc^{\log^{p-1}} \subset\W \subset W^ {1,2}\cap \Cc^0$, the same estimate holds  for  the norm $\| \cdot \|_{W^{1,2} \cap \Cc^{\log^{p-1}} \to W^{1,2} \cap \Cc^0}$. Therefore, the $\Cc^{1/2}$-norm of $\sum_{n=0}^\infty \oQ_\xi^n$  over $K\setminus (-\delta,\delta)$ is bounded independently of $\delta$. In particular, it is $O(\delta^{-3/2})$.

For the remaining term, we write
	$$ { \oN_\xi \over 1-\lambda_\xi}    -{\oN_\eta\over 1-\lambda_\eta } =\Big(  { \oN_\xi \over 1-\lambda_\xi} - {\oN_\eta\over 1-\lambda_\xi }  \Big)  +   \Big( {\oN_\eta\over 1-\lambda_\xi }   -{\oN_\eta\over 1-\lambda_\eta }   \Big).$$
	 
	From the expansion $1-\lambda_\xi=-i\gamma\xi +O(\xi^2)$, we see that 
	$$   \Big\| { \oN_\xi \over 1-\lambda_\xi} - {\oN_\eta\over 1-\lambda_\xi } \Big\|_{\mathscr W} \lesssim { |\xi-\eta|^{1/2} \over |\xi|        } \leq    { |\xi-\eta|^{1/2} \over \delta  } \leq     { |\xi-\eta|^{1/2} \over \delta ^{3/2}       }    $$
	and
	$$    \Big\|    {\oN_\eta\over 1-\lambda_\xi }   -{\oN_\eta\over 1-\lambda_\eta }            \Big\|_{\mathscr W}    \lesssim \Big| {   1   \over  1-\lambda_\xi  }    - { 1 \over  1-\lambda_\eta }   \Big|  \lesssim      { |\xi-\eta|^{1/2} \over \delta^{3/2} },$$ where we have use again that the  $\Cc^{1/2}$-norm of $t\mapsto 1/t$ on $K\setminus (-\delta,\delta)$ is $O(\delta^{-3/2})$. As before, the same estimate holds for  the norm $\| \cdot \|_{W^{1,2} \cap \Cc^{\log^{p-1}} \to W^{1,2} \cap \Cc^0}$. This finishes the proof of the lemma.
\end{proof}

\subsection{Asymptotic behavior of the leading eigenvalue}
We end this section with  the following result on the asymptotic behavior of the powers of the leading eigenvalue $\lambda_\xi$ of $\oP_\xi$ for $\xi$ near $0$. This result will be used in Subsection \ref{subsec:R}.

\begin{lemma}\label{lemma-b-ep}
	Let $\mu$ be a non-elementary probability measure on $G =\SL_2(\C)$. Assume that $\int_G \log^p \|g\| \, \diff \mu(g) <\infty$ for some $p \geq 2$. 
	Let $\lambda_\xi$  be the leading eigenvalue of $\oP_\xi$ as in Corollary \ref{cor:P_t-decomp-W}. For $k\in \N$ and $\ep >0$ small enough, denote 
	$$B_\ep^k:=\int_0^\ep \Big(  {\lambda_\xi^{k}\over i\xi} +{\lambda_{-\xi}^{k}\over -i\xi}\Big) \diff \xi.$$ Then, there exists an $\ep_1>0$ independent of $k$ such that, for all $0<\ep<\ep_1$, one has $\lim_{k\to \infty}  B_\ep^k=\pi$.
\end{lemma}

We state a technical lemma first.

\begin{lemma}\label{xi-k-lemma}
	Let $\phi$ be a real valued function on $[0,1]$ such that $\lim_{\xi\to 0}\phi(\xi)=0$. There exists a sequence of positive numbers $\xi_k$ decreasing to $0$, such that $\lim_{k\to\infty} k\xi_k^2=\infty$, and for every $k$, one has $k\xi^2|\phi(\xi)|\leq 1$ for all $0 \leq \xi\leq \xi_k$.
\end{lemma}

\begin{proof}
	 Take a non-negative continuous increasing function $\Phi$ on $[0,1]$ such that $\Phi(0)=0$ and $|\phi|\leq \Phi$ outside the origin. It is enough to find $\xi_k$ satisfying $\lim_{k\to\infty} k\xi_k^2=\infty$ and $k\xi_k^2\Phi(\xi_k)\leq 1$.
	
	Since $\Phi$ is continuous,  there exists a sequence of positive numbers $\{\xi_k\}$ decreasing to $0$,  such that $\xi_k^2 \Phi(\xi_k)=1/k$ for $k$ large enough. Hence $k\xi_k^2\Phi(\xi_k)=1$ for large $k$. From the fact that $\Phi(0)=0$, we get $\lim_{k\to\infty} k\xi_k^2 =\infty$. 
\end{proof}

\begin{proof} [Proof of Lemma \ref{lemma-b-ep}]
	Recall from (\ref{eq:lambda-expansion-2}) that $\lambda_\xi =1+i\gamma \xi -(a^2+\gamma^2)\xi^2/2+o(\xi^2)$ as $\xi\to 0$. This gives that $|\lambda_\xi|^ 2 =1-a^2\xi^2+o(\xi^2)$ and
	$$\mathrm{Arg}\lambda_\xi=\arctan {\gamma\xi +o(\xi^2)\over 1+O(\xi^2)}=\arctan\big(\gamma\xi+o(\xi^2)\big)=\gamma\xi +o(\xi^2).$$ 
	It follows that
	$$|\lambda_\xi|^k=\big(1-a^2\xi^2+o(\xi^2)\big)^{k/2}= e^{\frac{k}{2} \log (1-a^2\xi^2+o(\xi^2))} = e^{-ka^2\xi^2/2+o(k\xi^2)}=e^{-ka^2\xi^2/2+ k\xi^2 \phi(\xi)}$$ for some $\phi$ satisfying $\lim_{\xi\to 0} \phi(\xi)=0$, and
	\begin{equation} \label{eq:exp-arg-lambda}
	e^{ik \mathrm{Arg}\lambda_{\xi}}=e^{ik\gamma \xi}+o(k\xi^2).
	\end{equation}

	By Lemma \ref{xi-k-lemma} applied to the function $\widetilde\phi(\xi):=\max\big(|\phi(\xi)|,|\phi(-\xi)|\big)$,	there exists a sequence of positive numbers $\xi_k$  bounded by $\ep$ and decreasing to $0$, such that $\lim_{k\to \infty}k\xi_k^2=\infty$, and for every $k$,  $k\xi^2 |\phi(\xi)|\leq 1$ for all $|\xi|\leq \xi_k$. We split the integral $B_\ep^k$ into two parts, one with $0\leq\xi\leq \xi_k$ and the other with $\xi_k< \xi\leq \ep$ and denote them by $B_1^k$ and $B_2^k$, respectively.
   \vskip 3pt
	
	For the first part, since $k\xi^2|\phi(\xi)|\leq 1$ and $e^s \lesssim 1 + |s|$ when $|s| \leq 1$, we have 
	$$|\lambda_\xi|^k=e^{-ka^2\xi^2/2}e^{ k\xi^2 \phi(\xi)}=e^{-ka^2\xi^2/2}\big(1+O(k\xi^2\phi(\xi))\big)=e^{-ka^2\xi^2/2}+O\big( k\xi^2\phi(\xi)e^{-ka^2\xi^2/2} \big).$$
	So, $\lambda_\xi^k-\lambda_{-\xi}^k=|\lambda_\xi|^k  e^{ik \mathrm{Arg}\lambda_\xi} -|\lambda_{-\xi}|^k e^{ik \mathrm{Arg}\lambda_{-\xi}}$, which from \eqref{eq:exp-arg-lambda} equals
	\begin{align*}
	2i\sin(k\gamma\xi)e^{-ka^2\xi^2/2}+O\big(k\xi^2 \widetilde\phi(\xi)e^{-ka^2\xi^2/2}\big)+o\big(k\xi^2e^{-ka^2\xi^2/2}\big).
	\end{align*}
	Therefore, $$ B_1^k=\int_{0}^{\xi_k} {2\sin(k\gamma\xi)\over \xi} e^{-ka^2\xi^2/2}\diff \xi+\int_{0}^{\xi_k} O\big(k\xi\widetilde\phi(\xi) e^{-ka^2\xi^2/2}\big) \diff \xi+\int_{0}^{\xi_k} o\big(k\xi e^{-ka^2\xi^2/2}\big) \diff \xi .$$
	By  changing variables with $\eta=\sqrt{k}a \xi$, it is easy to see that the second and third terms are $o(1)$ as $k\to\infty$ because  $\eta e^{-\eta^2/2}$ is integrable and $\widetilde\phi$ is small on $[0,\xi_k]$.
		The first term is equal to 
	\begin{align*}
	\int_0^{\sqrt{k}a\xi_k} {2\sin(\sqrt{k}\gamma \eta/a)\over \eta} e^{-\eta^2/2}   \diff \eta&=\Big(\int_0^\infty  -\int_ {\sqrt{k}a\xi_k} ^\infty\Big)  {2\sin(\sqrt{k}\gamma \eta/a)\over \eta} e^{-\eta^2/2}   \diff \eta\\
	&=\int_0^\infty    {2\sin(\sqrt{k}\gamma \eta/a)\over \eta} e^{-\eta^2/2}   \diff \eta +  o(1)
	\end{align*}
	as $k\to \infty$, where the last equality holds because  $\lim_{k\to\infty}\sqrt k\xi_k=\infty$.
	Using the formula $2i\sin(t)=e^{it}-e^{-it}$ and the fact that ${e^{i(\sqrt{k}\gamma \eta/a)}\over i\eta}e^{-\eta^2/2}$ is integrable over $[\ep,\infty)$ for every $\ep>0$, we can write the above integral as  
	$$\lim_{\ep\to 0^+}\int_\ep^\infty    {2\sin(\sqrt{k}\gamma \eta/a) \over \eta} e^{-\eta^2/2}  \diff \eta=\lim_{\ep\to 0^+} \Big(\int_{-\infty}^{-\ep}{e^{i(\sqrt{k}\gamma \eta/a)}\over i\eta} e^{-\eta^2/2}d\eta+\int_\ep^\infty {e^{i(\sqrt{k}\gamma \eta/a)}\over i\eta} e^{-\eta^2/2}  \diff \eta\Big), $$
	which is the Cauchy principal value of $\int_{-\infty}^\infty {e^{i(\sqrt{k}\gamma \eta/a)}\over i\eta}e^{-\eta^2/2}  \diff \eta$. 
	Recall that the Fourier transform of $e^{-u^2/2}$  is $\sqrt{2\pi}e^{-\xi^2/2}$. Applying Lemma \ref{f/x} with $t=-\sqrt k \gamma/a$ and $\xi=\eta$, we get 
	$$\mathrm{p.v.}\int_{-\infty}^\infty {e^{i(\sqrt{k}\gamma \eta/a)}\over i\eta}e^{-\eta^2/2}  \diff \eta=\pi\int_{-\infty}^{{\sqrt{k}\gamma \over a}} {e^{-u^2/2} \over\sqrt{2\pi}}\diff u -\pi\int^{\infty}_{{\sqrt{k}\gamma \over a}} {e^{-u^2/2}\over\sqrt{2\pi}} \diff u=\pi+o(1)$$
		as $k\to \infty$.
		Therefore,  we conclude that  $\lim_{k\to\infty}B_1^k=\pi$.
  \vskip 3pt
	
	In order to finish the proof, it remains to show that $B_2^k$ tends to zero as $k \to \infty$. From the choice of $\xi_k$, one has $\xi_k\gtrsim 1/\sqrt k$. On the other hand, from the expansion of $|\lambda_\xi|^k$ we get that $|\lambda_\xi|^k\lesssim e^{-ka^2\xi^2/4}$ for $|\xi|\leq\ep_1$ and $\ep_1 >0$ small enough. Therefore, we have
	$$|B_2^k|\lesssim \int_{\xi_k}^\ep {1\over \xi_k} e^{-ka^2\xi^2/4} \diff \xi\lesssim \int_{\xi_k}^\ep \sqrt k  e^{-ka^2\xi^2/4}\diff \xi \lesssim \int_{\sqrt k a\xi_k}^{\infty} e^{-\eta^2/4}\diff \eta = o(1)$$
	as $k\to\infty$, because $\sqrt k \xi_k \to \infty$. This concludes the proof of the proposition.
\end{proof}

\section{Renewal operators}\label{sec:renewal}

In this section, we follow \cite{li:fourier} by introducing several ``renewal operators'' and studying their asymptotic behaviors.  These results are the technical core of the present work and will be used to obtain the key estimates in the proof of our main theorem and its more general version Theorem \ref{thm:fourier-general}, see the beginning of Subsection \ref{subsec:fourier-A}. The probabilistic interest of these operators is that they can be used to study the stopping time function $n_t(x):= \inf \{n \in \N : \sigma(S_n,x) > t\}$ and associated random processes, such as $\sigma(S_{n_t(x)},x) - t$ and so on. Renewal operators associated to products of random matrices and more general Markov chains have been studied by various authors. The reader may refer to \cite{kesten:renewal,boyer,guivarch-lepage,li:fourier,li:fourier-2}.

The available results in the setting of random walks on Lie groups only work under exponential moment conditions. In this case, one works with H\"older continuous observables and use the available spectral analysis of the operators $\oP_\xi$.  In order to be able to prove renewal theorems in our setting of low moments, we are required to work with the spaces $W^{1,2}$, $\Cc^{\log^{p-1}}$ and $\W$ that appeared above and to reprove most of the estimates for the corresponding norms, although in a weaker form.  As we dispose of a low regularity for the family of operators $\xi \mapsto \oP_\xi$ (in contrast with the analyticity observed in the case of finite exponential moment), our analysis requires some effort and new ideas. 
 
 \medskip
 
As in the end of Section \ref{sec:markov}, we equip the spaces $W^{1,2} \cap \Cc^{\log^1}$ and $W^{1,2} \cap \Cc^0$ with the norms $\max \big\{\|\cdot\|_{W^{1,2}},\|\cdot\|_{\log^1}\big\}$ and $\max \big\{\|\cdot\|_{W^{1,2}},\|\cdot\|_{\infty}\big\}$ respectively.

 \medskip

\noindent\textbf{Standing assumption:} Throughout this section, $\mu$ is a non-elementary measure on $G = \SL_2(\C)$ having a finite second moment, i.e., \  $\int_G \log^2 \|g\| \,\diff \mu(g) < \infty$.

\subsection{First renewal operator} \label{subsec:R}
For a bounded function $f$ on $\P^1\times \R$, we define the renewal operator $\oR$ by
$$\oR f(x,t):=\sum_{n\geq 0}\int_{G} f\big(gx,\sigma_g(x)-t\big) \diff \mu^{*n}(g),$$
where, by convention, we set the measure  $\mu^{*0}$ to be the point mass at the identity matrix. Note that the series defining $\oR f$ may be divergent. However, we will only work with functions $f$ for which the above series converge.

When $f(x,u) = \mathbf 1_{u \in [a,b]}$ is the indicator function of an interval $[a,b]$ in $\R$, the value $\oR f(x,t)$ counts, for all $n \geq 0$, the random products $g_n \cdots g_1$ for which the cocycle value $\sigma(g_n \cdots g_1,x)$ falls in the interval $[a+t,b+t]$. Various estimates of such quantities are obtained below.

\vskip3pt

For a function $f$ on $\P^1\times \R$ such that  $f(x,\cdot)\in L^1( \R)$ for all $x\in \P^1$, we write $$\oF_uf(x,\xi):=\int_{-\infty}^{\infty}f(x,u)e^{- iu\xi}  \,\diff u,$$
 which is the Fourier transform with respect to the second coordinate.

The first result of this section concerns the asymptotic behavior of $\oR f(x,t)$ for large values of $t$. Its main feature is the appearance of an absolutely continuous measure in the $u$ variable in the limit distribution. This is the most fundamental result in this section and the asymptotic estimates for the various operators defined later will follow from this one. 

 The following proposition is analogous to results from \cite{li:fourier,li:fourier-2}, but our proof is largely different from the ones there. The main reasons are that we do not dispose of analytic families of operators and that, when working with the spaces $W^{1,2}$ and $\W$ and under low moments conditions, it is not possible to perturb the Markov operator $\oP_0$ along real directions. For these reasons, we need to use new arguments and deal with some singular families of operators.

\begin{proposition}\label{renewal-1}
	Let $f$ be a  bounded continuous function on $\P^1\times \R$ such that $\oF_uf(x,\cdot)$ is well-defined and of class $\Cc^1$ for all $x$. Assume that $\oF_uf(\cdot,\xi)$ and $\partial_\xi\oF_uf(\cdot,\xi)$ are in  $W^{1,2}\cap \Cc^{\log^1}$ with $W^{1,2}\cap\Cc^{\log^1}$-norms  bounded by a positive constant $M$ for all $\xi\in\R$. Assume moreover   $\supp(\oF_uf)$  is contained $\P^1\times K$ for some compact set $K$. Then,  for all $x\in\P^1$ and $t>0$, we have
	$$\Big|\oR f(x,t)-{1\over \gamma}\int_{\P^1}\int_{-t}^\infty f(y,u)\, \diff u \diff \nu(y)\Big|\leq {C_KM\over (1+t)^{1/4}},$$
	where $C_K>0$ is a constant independent of $f,M,x,t$.
\end{proposition}

\begin{proof}
	Fix an $x\in\P^1$. We can assume $M = 1$ for simplicity.	We begin by using the inversion formula for $\oF_u$ to rewrite
	\begin{align*}f\big(gx,\sigma_g(x)-t\big)&={1\over {2\pi}}\int_{-\infty}^\infty \oF_u f(gx,\xi)e^{ i(\sigma_g(x)-t)\xi}\diff \xi
	={1\over {2\pi}}\int_{-\infty}^\infty e^{ i\xi\sigma_g(x)}\oF_u f(gx,\xi)e^{- it\xi}\diff \xi \\
	&={1\over {2\pi}} \int_0^{\infty} \Big[e^{ i\xi\sigma_g(x)}\oF_u f(gx,\xi)e^{- it\xi} +e^{- i\xi\sigma_g(x)}\oF_u f(gx,-\xi)e^{it\xi} \Big]\diff \xi,
	\end{align*}
where the last identity is valid because the last two functions in the right hand side are integrable.
	
 Put $\Phi(g,x,\xi,t):= e^{ i\xi\sigma_g(x)}\oF_u f(gx,\xi)e^{- it\xi}$. Recall that $\oP^n_\xi \psi  (x) = \int_G e^{i \xi \sigma_g(x)} \psi (gx) \, \diff \mu^{*n}(g)$ for $n \geq 0$. For all small $\ep>0$, we have
	\begin{align}
	\oR f(x,t)&={1\over {2\pi}}\sum_{n\geq 0}\int_G \int^{\infty}_0 \Big[ \Phi(g,x,\xi,t)+\Phi(g,x,-\xi,t)\Big] \diff \xi \diff \mu^{*n}(g) \nonumber	\\
	&={1\over {2\pi}}\sum_{n\geq 0} \int_{0}^\infty \int_G \Big[\Phi(g,x,\xi,t)+\Phi(g,x,-\xi,t)\Big] \diff \mu^{*n}(g) \diff \xi	\nonumber\\
	&={1\over {2\pi}}\sum_{n\geq 0} \int_{0}^\infty \Big[\oP^n_\xi \oF_u f(x,\xi)e^{- it\xi} +\oP^n_{-\xi} \oF_u f(x,-\xi)e^{ it\xi} \Big]\diff \xi \nonumber\\
	&={1\over {2\pi}}\sum_{n\geq 0} \Big(\int_0^\ep+\int_{\ep}^\infty \Big) \Big[ \oP^n_\xi \oF_u f(x,\xi)e^{- it\xi}+ \oP^n_{-\xi} \oF_u f(x,-\xi)e^{ it\xi}\Big] \diff \xi  
       \nonumber\\
	&=  {1\over {2\pi}}\sum_{n\geq 0}  \int_0^{\ep} \Big[ \oP^n_\xi \oF_u f(x,\xi)e^{- it\xi}+ \oP^n_{-\xi} \oF_u f(x,-\xi)e^{ it\xi}\Big] \diff \xi  \label{0-ep} \\
& \quad\quad + {1\over {2\pi}}\int_{\ep}^\infty \sum_{n\geq 0} \Big[ \oP^n_\xi \oF_u f(x,\xi)e^{- it\xi} +\oP^n_{-\xi} \oF_u f(x,-\xi)e^{ it\xi}\Big] \diff \xi  . \label{interchange}
	\end{align}
	Notice that, in the second equality above we have used Fubini's theorem to interchange the order of integration. This is justified by the fact that $\oF_uf(\cdot,\cdot)$ is bounded by $M=1$ and $\oF_uf(x,\cdot)$ is supported on $K$ for all $x$.  We have also used Fubini's theorem in \eqref{interchange}. In order to justify that, one needs to find an integrable function on $[\ep,\infty)$ that dominates $\sum_{n= 0}^k \big[ \oP^n_\xi \oF_u f(x,\xi)e^{- it\xi} +\oP^n_{-\xi} \oF_u f(x,-\xi)e^{ it\xi}\big]$ for every $k \geq 0$. As in the proof of Proposition \ref{prop:sum-Pxi-n-decomposition}, we have that $\|\oP^n_\xi \oF_u f\|_\infty$ is exponentially small as $n$ tends to infinity,  uniformly in $\xi\in K\setminus (-\ep,\ep)$. Since $\oF_u f(x,\xi)$ vanishes for $\xi$ outside $K$ by assumption, we can take $C \mathbf 1_{\xi\in K}$ for some suitable constant $C>0$ as the dominating function.
	
	\medskip
	
	\noindent \textbf{Claim 1.} The limit of  \eqref{0-ep} as $\ep\to 0$ is equal to ${1\over {2\gamma}} \int_{\P^1}\int_{-\infty}^\infty  f(y,u) \diff u\diff \nu(y)$.
		\proof[Proof of Claim 1]  We can assume that $\ep\leq \min(\ep_0,t^{-2},1)$, where $\ep_0>0$ is as in Corollary \ref{cor:P_t-decomp-W}. In particular, we can use the decomposition $\oP_\xi=\lambda_\xi \oN_\xi+\oQ_\xi$ on $\W$ given by that corollary.
	It follows from our assumptions that $\xi \mapsto \oF_uf(\cdot,\xi)$ is a $\Cc^1$-curve in $W^{1,2}\cap\Cc^{\log^1}$ whose $\Cc^1$-norm is bounded by $2$ (recall that $M=1$), so  $\big\|\oF_u f(\cdot,\xi) - \oF_u f(\cdot,0)\big\|_{W^{1,2}\cap\Cc^{\log^1}}  = O(|\xi|)$.
	Then, $\big|\oN_0 \oF_u f(x,\xi)e^{- it\xi}-\oN_0 \oF_u f(x,0)\big|$ is bounded by 
	\begin{align*}
	\big\|\oN_0 \oF_u f(\cdot,\xi)(e^{- it\xi}-1)\big\|_\infty+\big\|\oN_0\big( \oF_u f(\cdot,\xi)- \oF_u f(\cdot,0)\big)\big\|_\infty 
	\lesssim  t|\xi| + |\xi| .
	\end{align*}
	Using this and \eqref{N-xi-N-0} from the proof of Proposition \ref{prop:sum-Pxi-n-decomposition}, we obtain
	\begin{align*} 
	&\Big\| \frac{\oN_\xi \oF_u f(\cdot,\xi)e^{- it\xi}}{1-\lambda_\xi}  \,- \, \frac{\oN_0  \oF_u f(\cdot,0)}{-i\gamma\xi} \Big \|_{\infty} \\
	&= \Big\| \Big(\frac{\oN_\xi}{1-\lambda_\xi}  - \frac{\oN_0}{-i\gamma\xi}\Big)\oF_u f(\cdot,\xi)e^{- it\xi} \Big \|_{\infty} + \Big\| \frac{\oN_0}{-i\gamma\xi} \Big( \oF_u f(\cdot,\xi)e^{- it\xi} - \oF_u f(\cdot,0)  \Big) \Big \|_{\infty} \\
	&\lesssim  |\xi|^{-1/2}+ t+ 1\leq |\xi|^{-1/2}+ \ep^{-1/2} + \ep^{-1/2}\lesssim |\xi|^{-1/2},
	\end{align*}
	where the constants involved are independent of $\ep$. Recall that $0<\xi<\ep$ here.

	On the other hand, we have $|\lambda_\xi|\leq 1$ and $\sum_{n=0}^k \oQ_{\xi}^n \oF_u f(x,\xi)e^{ -it\xi}=O(1)$ where the constants involved are independent of $k$. This holds because, by Corollary \ref{cor:P_t-decomp-W}-(5),  $\|\oQ^n_\xi\|_{\W}$ is exponentially small when $n \to \infty$, uniformly in $\xi$. 
	
	From the above discussion and the identity $\sum_{n=0}^k \lambda_\xi^n = \frac{1-\lambda_\xi^{k+1}}{1-\lambda_\xi}$,  the sum of \eqref{0-ep} up to $k$  equals $(2\pi)^{-1}$ times 
 \begin{align}
\int_{0}^\ep & \sum_{n=0}^k \Big[ \oP^n_\xi \oF_u  f(x,\xi)e^{- it\xi} + \oP^n_{-\xi} \oF_u f(x,-\xi)e^{ it\xi} \Big] \diff \xi \nonumber\\
 & =\int_{0}^\ep \Big[	{(1-\lambda_\xi^{k+1})\oN_\xi \oF_u f(x,\xi)e^{- it\xi} \over 1-\lambda_\xi} + \sum_{n=0}^k \oQ_\xi^n \oF_u f(x,\xi)e^{- it\xi}  +\nonumber\\
 &\quad\quad\quad\quad\quad\quad\quad\quad\quad\quad\quad\,\,\,\,\,	{(1-\lambda_{-\xi}^{k+1})\oN_\xi \oF_u f(x,-\xi)e^{ it\xi} \over 1-\lambda_{-\xi}} +\sum_{n=0}^k \oQ_{-\xi}^n \oF_u f(x,-\xi)e^{ it\xi} \Big]\diff \xi   \nonumber\\
 &=\int_{0}^\ep  \Big[  {(1-\lambda_\xi^{k+1})\oN_0 \oF_u f(x,0) \over -i\gamma\xi}+ {(1-\lambda_{-\xi}^{k+1}) \oN_0 \oF_u f(x,0) \over i\gamma\xi}+O(|\xi|^{-1/2})+O(1)\Big] \diff \xi   \nonumber\\
 &= {\oN_0 \oF_u f(x,0)\over \gamma} \int_{0}^\ep \Big( {\lambda_\xi^{k+1} \over i\xi} +{\lambda_{-\xi}^{k+1} \over -i\xi}\Big) \diff \xi + O(\ep^{1/2}) + O(\ep),
 \end{align}
  where the implicit in $O(\ep)$ and $O(\ep^{1/2})$ are independent of $k$. We use here that $|\lambda_\xi| \leq 1$.

    Letting first $k \to \infty$ then $\ep \to 0$ and using Lemma \ref{lemma-b-ep}  gives that \eqref{0-ep}  tends to 
    $${1\over {2\pi}}\cdot {\pi\over \gamma}\oN_0\oF_u f(x,0) =  {1\over {2\gamma}} \int_{\P^1}\int_{-\infty}^\infty  f(y,u) \diff u\diff \nu(y)$$
    as $\ep \to 0$.   This proves the claim.\endproof

	We now treat the term  \eqref{interchange}. By Proposition \ref{prop:sum-Pxi-n-decomposition} and the definition of Cauchy principal value, when $\ep\to 0$, this term becomes 
	\begin{align*}
	&\lim_{\ep\to 0^+}{1\over {2\pi}}\int_{\ep}^\infty \Big[\Big( {\oN_0 \over {-\gamma i\xi}}+\oU_\xi\Big) \oF_u f(x,\xi)  e^{- it\xi} +\Big({\oN_0 \over {\gamma i\xi}}+\oU_{-\xi}\Big)  \oF_u f(x,-\xi) e^{ it\xi} \Big]\diff \xi \\
	& =	{1\over {2\pi}}\,\mathrm{p.v.}\int_{-\infty}^\infty  {\oN_0 \oF_u f(x,\xi)\over {-\gamma i\xi}}   e^{- it\xi} \diff \xi    +   {1\over {2\pi}}  \int_{-\infty}^\infty  \oU_\xi \oF_u f(x,\xi)  e^{- it\xi} \diff \xi.
	\end{align*}
By Lemma \ref{f/x} and Fubini's theorem, 	the first term above is equal to 
	\begin{align*}
-\frac {1}{2\gamma}\int_{\P^1}\int_{-\infty}^{-t} f(y,u) \diff u  \diff \nu(y)
	& +\frac {1}{2\gamma}\int_{\P^1}\int_{-t}^\infty f(y,u) \diff u  \diff \nu(y).
	\end{align*}
	 Combining with Claim 1 yields	
$$\oR f(x,t)={1\over \gamma}\int_{\P^1}\int_{-t}^\infty f(y,u) \,\diff u \diff \nu(y)+  {1\over {2\pi}}  \int_{-\infty}^\infty  \oU_\xi \oF_u f(x,\xi)  e^{- it\xi} \diff \xi.$$  
		In order to finish the proof, it remains to estimate the last integral. 
	\vskip 5pt

\noindent \textbf{Claim 2.}  The map $\xi \mapsto  \oU_\xi\oF_u f(x,\xi)$ is locally $\frac12$-H\"older continuous on $\R \setminus \{0\}$ and its $\Cc^{1/2}$-norm over $K \setminus (-\delta,\delta)$ is  $O(\delta^{-3/2})$.

\proof[Proof of Claim 2] In order to simplify the notation, denote $E_1:=W^{1,2} \cap \Cc^{\log^1}$, $E_2:=W^{1,2} \cap \Cc^0$ and $\phi_\xi(\,\cdot\,):= \oF_u f(\cdot,\xi)$. We want to estimate  
$$ \sup_{\xi\neq\eta \in K \setminus (-\delta,\delta)}{\|\oU_\xi \phi_\xi - \oU_\eta \phi_\eta\|_\infty    \over |\xi - \eta|^{1 \slash 2} }.$$ 
By assumption, $\xi \mapsto \phi_\xi \in E_1$ is a $\Cc^1$--curve whose $\Cc^1$-norm is bounded by $2$, so $\|\phi_\xi - \phi_\eta\|_{E_1}\leq 2 |\xi-\eta|$. By Proposition \ref{prop:sum-Pxi-n-decomposition}, we have $\|\oU_\xi\|_{E_1 \to E_2} = O(|\xi|^{- 1 \slash 2})$, and from Lemma \ref{holder-delta-2}, the map $\xi \mapsto \oU_\xi \in \text{Hom}(E_1,E_2)$ is  $\frac12$-H\"older continuous on $\R \setminus \{0\}$ and its $\Cc^{1/2}$-norm on $K \setminus (-\delta,\delta)$ is  $O(\delta^{-3/2})$. Therefore,  for $\xi, \eta \in K \setminus (-\delta,\delta)$, we have
\begin{align*}
\|\oU_\xi \phi_\xi - \oU_\eta \phi_\eta\|_\infty &\leq \|\oU_\xi \phi_\xi - \oU_\eta \phi_\eta\|_{E_2} \leq \|\oU_\xi \phi_\xi - \oU_\xi \phi_\eta\|_{E_2} + \|\oU_\xi \phi_\eta - \oU_\eta \phi_\eta\|_{E_2} \\ 
&\leq \|\oU_\xi\|_{E_1 \to E_2} \|\phi_\xi - \phi_\eta\|_{E_1} + \|\oU_\xi - \oU_\eta\|_{E_1 \to E_2} \|\phi_\eta\|_{E_1} \\ 
& \lesssim  \delta^{-1 \slash 2} |\xi - \eta|  +    \delta^{-3 \slash 2}  |\xi - \eta|^{1 \slash 2} \lesssim  \delta^{-3 \slash 2}  |\xi - \eta|^{1 \slash 2} .
\end{align*}
This proves the claim. \endproof

Since $\|\oU_\xi\|_{W^{1,2} \cap \Cc^{\log^1} \to W^{1,2} \cap \Cc^0} = O(|\xi|^{- 1 \slash 2})$ and $\norm{\oF_u f(\cdot,\xi)}_{W^{1,2} \cap \Cc^{\log^1}}\leq 1$, we have $\norm{\oU_\xi\oF_u f(\cdot,\xi)}_{\infty} =O(|\xi|^{-1\slash 2})$  when $|\xi|\to 0$.
On the other hand, by assumption,  $\supp\big(\oF_uf (x, \cdot)\big)$ is contained in the compact set $K$.  From these estimates and Claim 2, we can apply Lemma \ref{lemma:fourier-decay-holder} with $q=3/2$. This gives
	\begin{align*}
	\Big|\int_{-\infty}^\infty \oU_\xi\oF_u f(x,\xi)  e^{- it\xi} \,\diff \xi \Big|\leq {C_K\over (1+t)^{1/4}}.
	\end{align*}
	
	Notice that all the estimates are independent of $x$. This completes the proof of the proposition.
\end{proof}

\begin{remark}
When $\mu$ has a finite exponential moment, the analogue of Proposition \ref{renewal-1} was obtained by Li in \cite{li:fourier}. There, the error term is $O(t^{-1})$. This was later improved to an exponential error term $O(e^{-\ep t})$ in  \cite{li:fourier-2}. In both cases the family $\xi \to \oU_\xi$ is analytic in $\xi$, even across $\xi = 0$.
\end{remark}

As a consequence of Proposition \ref{renewal-1}, we obtain useful estimates for $\oR$ acting on indicator functions. These results will occupy the remainder of this subsection.

\begin{lemma}\label{renewal-3-lemma-1}
  Let $0<\delta<1$ be a constant and assume $b_2-b_1\geq 2\delta$. Then for all $x\in\P^1$ and $t>0$, we have
	$$\oR(\mathbf 1_{u\in[b_1,b_2]})(x,t)\leq C( b_2-b_1)+C_\delta(b_2-b_1)(|b_2|+|b_1|+1)(1+t)^{-1/4},$$ 
	where $C_\delta>0$ is a constant independent of $b_1,b_2,x,t$, and $C>0$ is a constant independent of $\delta,b_1,b_2,x,t$.
\end{lemma}

\begin{proof} Observe that $\mathbf 1_{u\in[b_1,b_2]}(x,u)$ is independent of $x$. This will be the case of all the function appearing in this proof. In order to use Proposition \ref{renewal-1}, we need to approximate $\mathbf 1_{u\in[b_1,b_2]}$ by a function having a compactly supported Fourier transform. For this, let $\vartheta$ and $\vartheta_\delta$, $0<\delta < 1$, be the functions from Lemma \ref{l:vartheta}. Recall from that lemma that the Fourier transform  of $\vartheta$ (resp. $\vartheta_\delta$) is supported by $[-1,1]$ (resp. $[-\delta^{-2},\delta^{-2}]$). 

	Consider the function 
	$$\phi(x,u):=\vartheta_\delta*\mathbf1_{u\in [b_1,b_2]}(x,u):=\int_{-\infty}^\infty \mathbf 1_{w\in [b_1,b_2]} (x,w) \vartheta_\delta (u-w) \diff w.$$

	Note that for $u\in[b_1,b_2]$, the interval $[u-b_2,u-b_1]$ contains either $[-\delta,0]$ or $[0,\delta]$. Since $\vartheta_\delta$ is positive and even, we have
	$$\phi(x,u)=\int_{b_1}^{b_2} \vartheta_\delta(u-w)\diff w\geq \int_0^\delta \vartheta_\delta (w)\diff w$$ 
	and the last quantity is bounded from below by a fixed positive constant by Lemma \ref{l:vartheta}.
	Therefore, one has $\mathbf1_{u\in [b_1,b_2]}(x,u)\lesssim \phi(x,u)$. We have that $\oF_u  \phi=\widehat{\vartheta_\delta}\cdot \oF_u \mathbf 1_{u\in[b_1,b_2]}$ by the convolution formula, so  the projection of  $\supp(\oF_u \phi)$ to $\R_\xi$ is contained in  $[-\delta^{-2},\delta^{-2}]$.
	
	We now estimate $\norm{\oF_u \phi(\cdot,\xi)}_{\Cc^1}$ and $\norm{\partial_\xi\oF_u \phi(\cdot,\xi)}_{\Cc^1}$. By the identities
	$$ \oF_u \phi(x,\xi)=\widehat{\vartheta_\delta}(\xi)\int_{-\infty}^\infty \mathbf 1_{u\in[b_1,b_2]} e^{-iu\xi} \diff u =\widehat{\vartheta_\delta}(\xi)\int_{b_1}^{b_2} e^{-iu\xi} \diff u $$ and the fact that $|\widehat{\vartheta_\delta}|\leq 1$, we deduce that $\norm{\oF_u \phi(\cdot,\xi)}_{\Cc^1}\leq b_2-b_1.$
	
	 In order to estimate $\norm{\partial_\xi\oF_u \phi(\cdot,\xi)}_{\Cc^1}$ we first note that, from Lemma \ref{l:vartheta}, $\|\widehat{\vartheta_\delta}\|_{\Cc^1}$ is bounded by a constant independent of $\delta$. Consequently, we have 
	$$\norm{\partial_\xi\oF_u \phi(\cdot,\xi)}_{\Cc^1}\lesssim \norm{\partial_\xi \oF_u \mathbf 1_{u\in[b_1,b_2]}(\cdot,\xi)}_{\Cc^1}+\norm{\oF_u \mathbf 1_{u\in[b_1,b_2]}(\cdot,\xi)}_{\Cc^1}.$$
		From the identity
	$\partial_\xi \oF_u \mathbf 1_{u\in[b_1,b_2]}(\cdot,\xi)=\int_{b_1}^{b_2} -iu e^{-iu\xi}\diff u$, we obtain  $$\norm{\partial_\xi\oF_u \phi(\cdot,\xi)}_{\Cc^1}\lesssim (b_2-b_1)(|b_2|+|b_1|+1).$$
	Using that the $W^{1,2}\cap\Cc^{\log^1}$-norm is bounded by $\Cc^1$-norm, we deduce that
	$$\norm{\oF_u \phi(\cdot,\xi)}_{W^{1,2}\cap\Cc^{\log^1}}\lesssim b_2-b_1 \quad\text{and}\quad \norm{\partial_\xi\oF_u \phi(\cdot,\xi)}_{W^{1,2}\cap\Cc^{\log^1}}\lesssim (b_2-b_1)(|b_2|+|b_1|+1).$$
	
 Applying Proposition \ref{renewal-1}, using that $$\int_{\P^1}\int_{-t}^\infty \phi \,\diff u\diff \nu = \int_{\P^1}\int_{-t}^\infty \vartheta_\delta*\mathbf 1_{u\in[b_1,b_2]} \,\diff u\diff \nu\leq \int_{\P^1}\int_{-\infty}^\infty \vartheta_\delta*\mathbf 1_{u\in[b_1,b_2]} \,\diff u\diff \nu=b_2-b_1$$  and recalling that $\mathbf1_{u\in [b_1,b_2]} \lesssim \phi$ gives that
	$$\oR(\mathbf 1_{u\in[b_1,b_2]})(x,t)\lesssim \oR \phi(x,t)\leq {{b_2-b_1}\over \gamma}+ C_\delta {(b_2-b_1)(|b_2|+|b_1|+1)\over (1+t)^{1/4}}$$ for a constant $C_\delta>0$. 
	The desired inequality follows.
\end{proof}

In this next result, we observe that negative values of $t$ are also allowed, which will useful later.

\begin{lemma}\label{R-b-b}
	Assume $ b>0$. For all $x\in \P^1$ and $t\in\R$, we have
	$$\oR(\mathbf 1_{u\in [-b,b]}) (x,t)\leq C (b+1)^2$$
	for some constant $C>0$ independent of $b,x,t$.
\end{lemma}

\begin{proof}
Since the left hand  is increasing in $b$, the case where $0<b<1$ can be deduced from the case $b=1$. So, we can assume $b\geq 1$. If $t\geq 1$, Lemma \ref{renewal-3-lemma-1} yields  $\oR(\mathbf 1_{u\in [-b,b]}) (x,t)\lesssim b^2+b$. Thus, we can assume $t<1$.
 Let $ t':=\max(0,b+t)$. Observe that $t\leq t'\lesssim b$. Using the definition of $\oR$ and Lemma \ref{large-n}, we have 
	$$\oR(\mathbf 1_{u\in [-b,b]}) (x,t)\leq\sum_{n<\lceil 2t'/\gamma\rceil}\mu^{*n}\{\sigma_g(x)<t' \}+ \sum_{n\geq \lceil 2t'/\gamma\rceil }\mu^{*n}\{\sigma_g(x)<t' \}\leq \frac{2t'}{\gamma}+\varepsilon_0(t')\lesssim b.$$
The proof of the lemma is complete.
\end{proof}

\begin{lemma}\label{integral-R-l}
	Let $l:=\lfloor t/(3\gamma ) \rfloor,0<\delta<1$ and assume $b\geq \delta$. Then, for all $x\in\P^1$ and $t>0$,  we have
	$$\int_{G} \oR(\mathbf 1_{u\in [-b,b]})\big(g x,t-\log\norm{g}\big) \diff \mu^{*l}(g)\leq C\big[C_l(b+1)^2+b\big]+C_\delta {(b+1)^2\over  (1+t)^{1/4}},$$
	where $C_l:=C_{l,\gamma/2}$ is the constant in Proposition \ref{prop:BQLDT} corresponding to $\ep=\gamma/2$, $C>0$ is a constant independent of $\delta,b,x,t$, and  $C_\delta>0$ is a constant independent of $b,x,t$.
\end{lemma}

\begin{proof}
Set  $\bB_l:=\big\{g\in G:\, \log\norm{g}\leq 3\gamma l/2\big\}$. Applying Proposition \ref{prop:BQLDT} with $\ep=\gamma/2$ gives that $\mu^{*l}(\bB_l)\geq 1-C_l$.  Since $l=\lfloor t/(3\gamma ) \rfloor$,  we have that  $t-\log\norm{g}\geq t/2$ for all $g \in \bB_l$. 
		Therefore, using Lemmas \ref{renewal-3-lemma-1} and \ref{R-b-b}, the integral we want to estimate is bounded by a constant times
		\begin{align*}
		& \Big( \int_{G \setminus \bB_l}+\int_{\bB_l} \Big) \Big[\oR(\mathbf 1_{u\in [-b,b]})\big(g x,t-\log\norm{g}\big)  \Big] \diff \mu^{*l}(g) \\
		&\lesssim C_l(b+1)^2+ \int_{\bB_l} \Big[ b+C_\delta{b(b+1)\over(1+t-\log\norm{g})^{1/4}} \Big] \diff \mu^{*l}(g)\\
		&\lesssim C_l(b+1)^2+b+ \int_{\bB_l} C_\delta {b(b+1)\over  (1+t)^{1/4}} \diff \mu^{*l}(g)
		\lesssim C_l(b+1)^2+b+C_\delta {(b+1)^2\over  (1+t)^{1/4}},
		\end{align*}
	which is the desired result.
\end{proof}

We stress the notation $C_\delta$ and $C_l$ for the constants appearing the above lemma are not conflicting, since $0 < \delta < 1$ while $l \geq 1$ is an integer.

\begin{lemma} \label{logs-lemma}
	There exists a decreasing rate function $\zeta(s)$ such that 
	$$\oR( \mathbf 1_{\D(y,e^{-s})\times [-\log s-4,\log s+4]})(x,t)\leq \zeta(s)\quad\text{for all} \quad s>1, \,\, t\geq s^{8} \,\,  \text{ and } \, x,y\in \P^1.$$
\end{lemma}

\begin{proof} We want to apply Proposition \ref{renewal-1}. For this, we need to regularize the function $\mathbf 1_{\D(y,e^{-s})}$ on $\P^1$ with a control on its $W^{1,2}\cap\Cc^{\log^1}$-norm. It is not hard to see that we can find a real smooth function $\phi$ on $\P^1$ supported by $\D(y,2e^{-s})$ such that $0 \leq \phi \leq 1$, $\phi=1$ on $\D(y,e^{-s})$, $\norm{\phi}_{W^{1,2}}\leq 1$ and $\norm{\phi}_{\Cc^1}\leq 4e^s$. In particular, we have $ \mathbf 1_{\D(y,e^{-s})\times [-\log s-4,\log s+4]}\leq\mathbf 1_{u\in [-\log s-4,\log s+4]}\cdot \phi(x)$. Consider the function 
	$$\Phi(x,u):=\vartheta_\delta*\mathbf 1_{u\in [-\log s-4,\log s+4]}(u)\cdot\phi(x),$$ 
	where $\delta=1/2$, $\vartheta_\delta$ is the function from Lemma \ref{l:vartheta} and the convolution acts only on $\R_u$. Observe that
	$\mathbf 1_{u\in [-\log s-4,\log s+4]}(u)\lesssim \vartheta_\delta*\mathbf 1_{u\in [-\log s-4,\log s+4]}(u)$ by  the same arguments as in the  proof of Lemma \ref{renewal-3-lemma-1}.	Therefore, we have
	$$\oR( \mathbf 1_{\D(y,e^{-s})\times [-\log s-4,\log s+4]})\leq \oR(\mathbf 1_{u\in [-\log s-4,\log s+4]}\cdot \phi)\lesssim \oR(\vartheta_\delta*\mathbf 1_{u\in [-\log s-4,\log s+4]}\cdot \phi)=\oR\Phi.$$ 
	
	We now estimate the ${W^{1,2}\cap\Cc^{\log^1}}$-norm of $\oF_u \Phi(\cdot,\xi)$ and $\partial_\xi\oF_u \Phi(\cdot,\xi)$. Recall that  $\norm{\phi}_{W^{1,2}}\leq 1$. We now estimate $\norm{\phi}_{\log^1}$ which by definition is the supremum of 
	$$ \big|\phi(x_0) - \phi(y_0)\big| \log^\star d(x_0,y_0) 
	\quad\text{over}\quad x_0 \neq y_0\in\P^1 .$$
	
	If $d(x_0,y_0) > e^{-2s}$,  using that $|\phi| \leq 1$ it easily follows that the above quantity is $\lesssim s$.  For $d(x_0,y_0) \leq e^{-2s}$,  the fact that $\norm{\phi}_{\Cc^1}\leq 4e^s$ and the mean value theorem give that the above quantity is $\lesssim 1$.
	This proves that $\norm{\phi}_{\Cc^{\log^1}} \lesssim s$, so  $\norm{\phi}_{W^{1,2}\cap\Cc^{\log^1}} \lesssim s$. By  computations analogous to the ones used in the proof of Lemma \ref{renewal-3-lemma-1}, one gets
	$$\norm{\oF_u \Phi(\cdot,\xi)}_{W^{1,2}\cap\Cc^{\log^1}}\lesssim s(\log s+1) \quad  \text{and} \quad \norm{\partial_\xi\oF_u \Phi(\cdot,\xi)}_{W^{1,2}\cap\Cc^{\log^1}}\lesssim s(\log s+1)^2.$$
	
	We can thus apply Proposition \ref{renewal-1} to $\Phi$, yielding
	\begin{equation}\label{PHI}
	\oR\Phi(x,t)\lesssim \nu\big(\D(y,2e^{-s})\big)(\log s+1)+s(\log s+1)^2 (1+t)^{-1/4}.
	\end{equation}

	By the regularity of $\nu$ (cf. Proposition \ref{regularity}), we have that $\nu\big(\D(y,2e^{-s})\big)\lesssim 1/s$. Hence when $t\geq s^{8}$,
	$$\oR( \mathbf 1_{\D(y,e^{-s})\times [-\log s-4,\log s+4]})(x,t)\lesssim (\log s+1)/s +(\log s+1)^2/s.$$  

 Taking  $\zeta(s):=C(\log s+1)^2/s$ for some large constant $C>0$ yields the desired estimate. This finishes the proof of the lemma.
\end{proof}

Consider now the following variation of the renewal operator $\oR$ defined by
$$\oL f(x,t):=\sum_{n=\lceil 2t/(3\gamma)\rceil }^{\lfloor 2t/\gamma\rfloor }\int_G f\big(gx,\log\norm{g}-t\big) \diff \mu^{*n}(g)  \quad \text{for}\quad t\geq 3\gamma.$$
In comparison with $\oR$, we use $\log\norm{g}$ instead of $\sigma_g(x)$ and the sum is truncated. By the large deviation estimates of Lemma \ref{diff-log}, the difference between $\oR$ and $\oL$ can be controlled, which yields the following estimate.

\begin{lemma} \label{operator-L}
	Let $b > 0$. Then $\oL (\mathbf 1_{u\in [-b,b]})(x,t)\lesssim 1+b^2$ for all $x\in\P^1$ and $t\geq 3\gamma$.
\end{lemma}

\proof	Let $l:=\lfloor t/(3\gamma)\rfloor$. By Lemma \ref{diff-log},  for every $n$ between $l$ and $100l$, we can find $\bS_{n,l,x}\subseteq G\times G$ depending on $x$ such  that  $\mu^{*(n-l)}\otimes \mu^{*l} (\bS_{n,l,x})\geq 1-D_{n}-D_l$, and for all $(g_2,g_1)\in \bS_{n,l,x}$, one has
\begin{equation}\label{dominate-equation-1}
\big| \log\norm{g_2g_1}-\sigma_{g_2}(g_1x)-\log\norm{g_1}\big| \leq 4e^{-\gamma l}\leq 4,
\end{equation}
where the constants $D_n$ satisfy $\sum_{n\geq 0} D_n<\infty$ and $\lim_{n\to \infty} nD_n=0$.

Using \eqref{dominate-equation-1}, writing $\mu^{*n} = \mu^{*(n-l)} \ast \mu^{*l}$ and treating the integral along $\bS_{n,l,x}$ and its complement separately,  we obtain that $\oL (\mathbf 1_{u\in [-b,b]})(x,t)$ is bounded by 
\begin{align*}
\sum_{n=\lceil 2t/(3\gamma)\rceil }^{\lfloor 2t/\gamma\rfloor} &\int_{\bS_{n,l,x}} \mathbf 1_{u\in [-b,b]}\big(g_2g_1x,\log\norm{g_2g_1}-t\big) \, \diff \mu^{*(n-l)}(g_2)\diff \mu^{*l}(g_1)+\sum_{n=\lceil 2t/(3\gamma)\rceil }^{\lfloor 2t/\gamma\rfloor }(D_n+D_l)\\
&\lesssim \sum_{n=\lceil 2t/(3\gamma)\rceil }^{\lfloor 2t/\gamma\rfloor}\int_{\bS_{n,l,x}}\mathbf 1_{-b\leq \log\norm{g_2g_1}-t\leq b}  \,\diff \mu^{*(n-l)}(g_2)\diff \mu^{*l}(g_1)+\sum_{n\geq 0} D_n +tD_l\\
&\lesssim \sum_{n=\lceil 2t/(3\gamma)\rceil }^{\lfloor 2t/\gamma\rfloor}\int_{\bS_{n,l,x}} \mathbf 1_{-b-4\leq \sigma_{g_2}(g_1x)-t+\log\norm{g_1}\leq b+4} \, \diff \mu^{*(n-l)}(g_2)\diff \mu^{*l}(g_1) +1 \\
&\leq \int_{G}   \oR(\mathbf 1_{u\in[-b-4,b+4]})\big(g_1x, t-\log\norm{g_1}\big)  \,   \diff \mu^{*l}(g_1) +1.
\end{align*}

Applying Lemma \ref{R-b-b}, we get that the last integral is bounded by a constant times $(b+5)^2$. The lemma follows.\endproof

\subsection{Residual process}

Let $f(x,v,u)$ be a bounded function on $\P^1\times \R_v \times \R_u$. Define the residual renewal operator by 
$$\oE f(x,t):=\sum_{n\geq 0}\int_{G^{2}}f\big(hgx,\sigma_h(gx),\sigma_g(x)-t\big)\diff \mu^{*n}(g)\diff \mu(h).$$
Notice that the above definition is similar to the one of $\oR$ from the previous section. Here we take into account one additional step in the random walk using $h \in G$. Observe that $\oE$ and $\oR$ coincide when acting on functions depending only on the $u$ variable.

Similarly as before, we define$$\oF_u f(x,v,\xi):=\int_{-\infty}^\infty f(x,v,u)e^{-iu\xi}\,\diff u,$$ which is the Fourier transform on the $\R_u$ factor.

The following result is the analogue of Proposition \ref{renewal-1} for the operator $\oE$.

\begin{proposition}\label{renewal-2}
Let $f$ be a  bounded continuous function on $\P^1\times \R_v \times \R_u$ such  that $\oF_u f(x,v,\cdot)$ is well-defined  and $\oF_u f(x,v,\cdot)\in \Cc^1$ for all $x,v$. Assume  that $\oF_u f(\cdot,\cdot,\xi)$ and $\partial_\xi\oF_u f(\cdot,\cdot,\xi)$ are of class $\Cc^1$ with $\Cc^1$-norms bounded by a positive constant $M$ for all $\xi$.  Assume moreover that the projection of $\supp(\oF_u f)$ to $\R_\xi$ is contained in a compact set $K$.  Then, for all $x\in\P^1$ and $t>0$, we have
\begin{align*}
\Big| \oE f(x,t)-{1\over \gamma}\int_{\P^1}\int_G\int_{-t}^\infty f\big(hy,\sigma_h(y),u\big) \diff u \diff \mu(h) \diff \nu(y) \Big| \leq {C_K M\over (1+t)^{1/4}},
\end{align*}
where $C_K>0$ is a constant independent of $f,M,x,t$.
\end{proposition}

\begin{proof}
Observe that $$\oE f(x,t)=\sum_{n\geq 0}\int_{G} \oS f\big(gx,\sigma_g(x)-t\big)\diff \mu^{*n}(g)=\oR(\oS f)(x,t),$$ 
where 
$$\oS f(x,u):=\int_G f\big(hx,\sigma_h(x),u\big)\diff \mu(h).$$  Therefore, the desired estimate for $\oE$ will follow directly from  Proposition \ref{renewal-1} once we show that $\oS f$ satisfies its hypotheses, possibly with a different constant $M' = O(M)$. 

Using Lebesgue's dominated convergence theorem, it is easy to see that the function $\oS f$ is bounded, continuous and $\oF_u \oS f(x,\cdot)$ is well-defined for all $x$.  By  Fubini's theorem, one has 
$$\oF_u \oS f(x,\xi)=\int_{-\infty}^\infty \int_G f\big(hx,\sigma_h(x),u\big)e^{-iu\xi} \, \diff \mu(h)\diff u=\int_G \oF_u f\big(hx,\sigma_h(x),\xi\big) \diff \mu(h).$$ 
Hence, the projection of $\supp(\oF_u \oS f)$ to $\R$ is contained in $K$. From the assumption that $\oF_u f(x,v,\cdot)\in \Cc^1$ and $\|\partial_\xi\oF_u f(\cdot,\cdot,\xi)\|_{\infty}\leq M$, we deduce that $\oF_u \oS f(x,\cdot)\in \Cc^1$.

It remains to bound $\norm{\oF_u \oS f(\cdot,\xi)}_{W^{1,2}\cap \Cc^{\log^1}}$  and $\norm{\partial_\xi\oF_u \oS f(\cdot,\xi)}_{W^{1,2}\cap \Cc^{\log^1}}$.
By the definition of $\oS f$, the triangle inequality and the assumption that $\mu$ has a finite second moment, it is enough to show that, for all $\xi$, we have
\begin{itemize}
\item[(i)] $\big\|\oF_u f\big(h \cdot \, ,\sigma_h( \, \cdot \, ),\xi \big) \big\|_{W^{1,2}}$ and  $\big\|\partial_\xi \oF_u f\big(h \cdot \, ,\sigma_h( \, \cdot \, ),\xi \big) \big\|_{W^{1,2}}$ are $\lesssim M \big(1 + \log^{1 \slash 2} \|h\|\big)$;
 \item[(ii)] $\big\|\oF_u f\big(h \cdot \, ,\sigma_h( \, \cdot \, ),\xi \big) \big\|_{\log^1}$ and  $\big\|\partial_\xi \oF_u f\big(h \cdot \, ,\sigma_h( \, \cdot \, ),\xi \big) \big\|_{\log^1}$ are $\lesssim M \big(1 + \log^2 \|h\|\big)$.
\end{itemize}

For (i), we observe first that $\big\|\oF_u f\big(h \cdot \, ,\sigma_h( \, \cdot \, ),\xi \big) \big\|_{L^1} \lesssim M$  since $\oF_uf $ is bounded by $M$ and the same reasoning applies to $\big\|\partial_\xi \oF_u f\big(h \cdot \, ,\sigma_h( \, \cdot \, ),\xi \big) \big\|_{L^1}$. Therefore, the only non-trivial estimates concern $\big \|\partial \oF_u f\big(h \cdot \, ,\sigma_h( \, \cdot \, ),\xi \big) \big \|_{L^2},\big \|\dbar \oF_u f\big(h \cdot \, ,\sigma_h( \, \cdot \, ),\xi \big) \big \|_{L^2}$ and $\big\| \partial \big[\partial_\xi \oF_u f\big(h \cdot \, ,\sigma_h( \, \cdot \, ),\xi \big) \big] \big \|_{L^2},\big\| \dbar \big[\partial_\xi \oF_u f\big(h \cdot \, ,\sigma_h( \, \cdot \, ),\xi \big) \big] \big \|_{L^2}$. We'll only estimate the terms involving $\partial$, the terms involving $\overline \partial$ being analogous.

From $h_* \sigma_h = -\sigma_{h^{-1}}$, we get $$\oF_u f\big(hx,\sigma_h(x),\xi\big) =h^*\big(\oF_u f(x,-\sigma_{h^{-1}}(x),\xi) \big).$$

Since $h$ acts unitarily on $L^2_{(1,0)}$, it is enough to estimate  $\big \|\partial \oF_u f \big(\, \cdot \, ,-\sigma_{h^{-1}}( \, \cdot \, ),\xi \big) \big \|_{L^2}$ and similarly for the $\xi$-derivative. By the chain rule, one has
\begin{align*}
(\partial_x \oF_u f)\big(x,-\sigma_{h^{-1}}(x),\xi\big)=\Big(\partial_x \oF_u f(x,v,\xi)+{\partial \oF_u f(x,v,\xi)\over \partial v}\partial_x(-\sigma_{h^{-1}}(x))\Big)\Big|_{v=-\sigma_{h^{-1}}(x)}.
\end{align*}
Using the assumption that $\norm{\oF_u f(\cdot,\cdot,\xi)}_{\Cc^1}\leq M$ and Cauchy-Schwarz inequality, we get
 \begin{align*}
 \big( i \partial \oF_uf \wedge \overline{ \partial\oF_uf } \big) \big(\, \cdot \, ,-\sigma_{h^{-1}}( \, \cdot \, ),\xi \big) \lesssim  M^2\omegaFS +M^2i\partial\sigma_{h^{-1}} \wedge \dbar\sigma_{h^{-1}}.
 \end{align*}
 
 From Lemma \ref{lemma:sigma-estimates}-(1) and the fact that $\|h\|=\|h^{-1}\|$, we get $\|\partial\sigma_{h^{-1}}\|^2_{L^2} \lesssim 1 + \log \|h\|$. Therefore, integrating the above inequality over $\P^1$ yields 
 $$\big \|\partial \oF_u f \big(\, \cdot \, ,-\sigma_{h^{-1}}( \, \cdot \, ),\xi \big) \big \|_{L^2} \lesssim M \big(1 + \log^{1 \slash 2} \|h\|\big).$$ This implies the first estimate in (i). Using that $\big\| \partial_\xi\oF_u f(\cdot,\cdot,\xi) \big\|_{\Cc^1}\leq M$, the second estimate is obtained analogously.

We now prove (ii). For the first estimate we need to prove that $$\big|\oF_u f\big(hx,\sigma_h(x),\xi\big) - \oF_u f\big(hy,\sigma_h(y),\xi\big)\big|  \lesssim M\big(1 + \log^2\|h\|\big) \cdot \big(\log^\star d(x,y) \big)^{-1}.$$
for all $x \neq y$ in $\P^1$. By the triangle inequality, it is enough to estimate 
$$\big|\oF_u f\big(hx,\sigma_h(x),\xi\big) - \oF_u f\big(hy,\sigma_h(x),\xi\big)\big|\quad\text{and}\quad \big|\oF_u f\big(hy,\sigma_h(x),\xi\big) - \oF_u f\big(hy,\sigma_h(y),\xi\big)\big|$$ separately. Since the $\Cc^1$-norm of $\oF_u f(\cdot,v,\xi)$ is bounded by $M$, the same is true for its $\Cc^{\log^1}$-norm. Using Lemma \ref{lemma:sigma-estimates}-(4), one concludes that the first term is $\lesssim M\big(1 + \log \|h\|\big)  \big(\log^\star d(x,y) \big)^{-1}$. Using again that $\|\oF_u f(\cdot,\cdot,\xi)\|_{\Cc^1}\leq M$ and Lemma \ref{lemma:sigma-estimates}-(3), we see that the second term is $\lesssim M\big(1 + \log^2 \|h\|\big) \big(\log^\star d(x,y) \big)^{-1}$. This proves the first estimate in (ii). The second estimate is proven similarly using the assumption $\big \| \partial_\xi\oF_u f(\cdot,\cdot,\xi) \big \|_{\Cc^1}\leq M$. The proof of the proposition is finished.
\end{proof}

\subsection{Residual process with cut-off} \label{subsec:E_1^+}

We now consider an operator obtained from $\oE$ which, for each $t>0$, takes into account only the trajectories for which the value of the cocycle jumps past $t$. More precisely, for a bounded function $f(x,v,u)$ on $\P^1\times \R_v\times \R_u$, we define  
$$\oE_1^+ f(x,t):=\sum_{n\geq 0}\int_{\sigma_g(x)<t\leq \sigma_{hg}(x)}f\big(hgx,\sigma_h(gx),\sigma_g(x)-t\big) \diff \mu^{*n}(g)\diff \mu(h).$$

Notice that, since $\sigma_g(x)+\sigma_h(gx)=\sigma_{hg}(x)$, the condition $\sigma_g(x)<t\leq \sigma_{hg}(x)$ is equivalent to $-\sigma_h(gx)\leq \sigma_g(x)-t <0$. Therefore, 
$$\oE_1^+ f(x,t)=\oE f_C(x,t), \quad \text{where} \quad f_C(x,v,u):=\mathbf 1_{-v\leq u<0}f(x,v,u)$$
 and $\oE$ is the residual operator introduced in the last subsection.
In particular, we have the following useful estimate, which says $\oE_1^+ $ is finite on bounded functions. We note that negative values of $t$ are allowed here.

	\begin{lemma}\label{renewal-3-lemma-2}
	For every $x\in\P^1$ and $t\in\R$, we have  	
	$$\big|\oE_1^+ f(x,t)\big|=\big|\oE(\mathbf 1_{-v\leq u<0}f)(x,t)\big|\leq C \norm{f}_{L^\infty},$$ where $C>0$ is a constant independent of $f,x,t$.
\end{lemma}

\begin{proof}
	It is enough to assume $f=\mathbf 1$.
	By definition, we have
	\begin{align*}
	&\oE_1^+ \mathbf 1(x,t)=\sum_{n\geq 0}\mu\otimes\mu^{*n}\big\{(h,g)\in G^2:\, -\sigma_h(gx)\leq  \sigma_g(x)-t<0\big\}\\
	&\leq \sum_{n\geq 0} \mu\otimes\mu^{*n}  \big\{(h,g)\in G^2:\, -\log\norm{h}\leq\sigma_g(x)-t<0\big\}=\int_G \oR(\mathbf 1_{u\in[-\log\norm{h},0)})(x,t)\,\diff \mu(h),
	\end{align*}
where $\oR$ is the first renewal operator introduced in Subsection \ref{subsec:R}.
	
 By Lemma \ref{R-b-b},  we have
 $$\oR(\mathbf 1_{u\in[-\log\norm{h},0)})(x,t) \leq \oR(\mathbf 1_{u\in[-\log\norm{h},\log\norm{h}]})(x,t) \lesssim (1+\log\norm{h})^2.$$
  Since $\int_G \log^2\norm{h} \diff \mu(h)<\infty$ by assumption,  the lemma follows.
\end{proof}

The asymptotic behavior of $\oE_1^+ f(x,t)$ for $t$ large is given by the next result. Notice that, in contrast to Propositions \ref{renewal-1} and \ref{renewal-2}, here we require conditions on $f$ rather than on its Fourier transform.

\begin{proposition}\label{renewal-3}
	Let $f$ be a $\Cc^1$ function on $\P^1\times \R_v \times \R_u$ such that $\norm{f}_{\Cc^1}\leq M$. Assume that  the projection of $\supp (f)$ to $\R_v$ is contained in $[-\kappa,\kappa]$ for some $\kappa>0$. Then, for all $x\in\P^1$,  $0<\delta< 1$ and $t>\kappa+\delta$, we have
	\begin{align*}
	\Big|\oE_1^+ f(x,t)-{1\over{\gamma}}\int_{\P^1}\int_G \int_{-\sigma_h^+(y)}^0 f\big(hy,\sigma_h(y),u\big) \diff u \diff \mu(h) \diff \nu(y) \Big|
	\leq  C_\delta{{M(\kappa+1)^2}\over (1+t)^{1/4}}+C \delta M,
	\end{align*} 
where $\sigma_h^+(\,\cdot\,):=\max\big(\sigma_h(\,\cdot\,),0\big)$, $C_\delta>0$ is a constant independent of $f,M,\kappa,x,t$, and $C>0$ is a constant independent of $\delta,f,M,\kappa,x,t$.
\end{proposition}

\begin{proof}
As noted above, we have that $\oE_1^+ f(x,t)=\oE f_C(x,t)$, where $f_C =\mathbf 1_{-v\leq u<0}f $. Notice also that
 $$  \int_{-\sigma_h^+(y)}^0 f\big(hy,\sigma_h(y),u\big) \diff u   = \int_{-\infty}^\infty \mathbf 1_{-\sigma_h^+(y)\leq u< 0} f\big(hy,\sigma_h(y),u\big) \diff u =\int_{-t}^\infty f_C \big(hy,\sigma_h(y),u\big) \diff u$$
 because $t>\kappa$ and $f(x,v,u)=0$ for $v\notin[-\kappa,\kappa]$. 
  Therefore, the desired inequality would follow directly from Proposition \ref{renewal-2} applied to $f_C$. However, we cannot apply this proposition directly,  because $\oF_uf_C(x,v,\cdot)$ is not $\Cc^1$ and it might not have compact support in $\R_\xi$. In order to remedy the situation we use the functions from Lemma \ref{l:vartheta} in order to approximate $f_C$ by a function with these properties. This approximation will give rise to the extra terms $ C_\delta (\kappa + 1)^2 M$ and $C\delta M$.
	
	For $0<\delta < 1$, let $\vartheta_\delta$ be the functions from Lemma \ref{l:vartheta}. Set
	 $$f_\delta(x,v,u):=f_C*\vartheta_\delta(x,v,u):=\int_{-\infty}^\infty f_C(x,v,w)\vartheta_\delta(u-w)\diff w.$$
		Notice that $f_\delta$ vanishes for $v$ outside $[0,\kappa]$.
		
	\vskip5pt
	
	\noindent	\textbf{Claim 1.} The function $f_\delta$ satisfies the hypotheses of Proposition \ref{renewal-2} with a new constant $M' \lesssim  (\kappa + 1)^2 M$ instead $M$.

	\proof[Proof of Claim 1] The convolution formula yields $\oF_u f_\delta=\widehat{\vartheta_\delta} \cdot\oF_u f_C$, so the projection of  $\supp(\oF_u f_\delta)$ to $\R_\xi$ is contained in  $[-\delta^{-2},\delta^{-2}]$. It is easy to see that $f_\delta$ is a  bounded continuous function and $\oF_u f_\delta(x,v,\cdot)$ is well-defined and $\Cc^1$ for all $x,v$.

	As $f_\delta$ vanishes for $v$ outside $[0,\kappa]$, in order to estimate $\norm{\oF_uf_\delta(\cdot,\cdot,\xi)}_{\Cc^1}$ and  $\norm{\partial_\xi\oF_uf_\delta(\cdot,\cdot,\xi)}_{\Cc^1}$, it is enough to consider $v\in [0,\kappa]$.
    From the identities
   $$\oF_uf_C(x,v,\xi)=\int f_C(x,v,u)e^{-iu\xi}\diff u=\int_{-v}^0f(x,v,u)e^{-iu\xi}\diff u,$$
   we get that $\norm{\oF_uf_C(\cdot,v,\xi)}_{\Cc^1} \leq $ $\int_{-v}^0 \norm{f}_{\Cc^1} \diff u$ $\leq \kappa \norm{f}_{\Cc^1}$ and $\norm{\oF_uf_C(x,\cdot,\xi)}_{\Cc^1} \leq$ $\int_{-\kappa}^0\norm{f}_{\Cc^1}\diff u+\norm{f}_{\Cc^0}\leq (\kappa+1)\norm{f}_{\Cc^1}$.  Since $\|\widehat{\vartheta_\delta}\|_{\Cc^1}$ is bounded independently of $\delta$ and $\widehat{\vartheta_\delta}$ is independent of $x$ and $v$, we obtain 
   $$\norm{\oF_uf_\delta(\cdot,\cdot,\xi)}_{\Cc^1}=\norm{\widehat{\vartheta_\delta} \oF_u f_C(\cdot,\cdot,\xi)}_{\Cc^1}\lesssim (\kappa+1)\norm{f}_{\Cc^1}.$$
   
Using again that $\|\widehat{\vartheta_\delta}\|_{\Cc^1}$ is bounded independently of $\delta$,  we get
   $$\norm{\partial_\xi \oF_uf_\delta(\cdot,\cdot,\xi)}_{\Cc^1}\lesssim \norm{\partial_\xi \oF_uf_C(\cdot,\cdot,\xi)}_{\Cc^1}+\norm{\oF_uf_C(\cdot,\cdot,\xi)}_\infty.$$
   Then, by the identity $\partial_\xi \oF_uf_C(x,v,\xi)=\int_{-v}^0 -iuf(x,v,u)e^{-iu\xi}\diff u$, the fact that $|u|\leq \kappa$ in the domain of integration and an argument similar to the one above, we obtain  $$\norm{\partial_\xi\oF_uf_\delta(\cdot,\cdot,\xi)}_{\Cc^1}\lesssim (\kappa^2+\kappa)\norm{f}_{\Cc^1}.$$ This proves the claim. \endproof

\noindent	\textbf{Claim 2.} We have $$\Big| \int_{\P^1}\int_G \Big( \int_{-t}^\infty f_\delta\big(hy,\sigma_h(y),u\big) \diff u - \int_{-\sigma_h^+(y)}^0  f\big(hy,\sigma_h(y),u\big) \diff u \Big)  \diff \mu(h) \diff \nu(y) \Big| \lesssim \delta M.$$

\proof[Proof of Claim 2] For $v\in[0,\kappa]$,	using $\int_\R \vartheta_\delta(u) \diff u =1$ and Fubini's theorem, we obtain
	\begin{align}
		\int_{-t}^\infty f_\delta (x,v,u) \diff u&=\int_{-t}^\infty \int_{-\infty}^\infty f_C(x,v,w)\vartheta_\delta(u-w)\diff w \diff u  = \int_{-t}^\infty  \int_{-v}^0   f(x,v,w)\vartheta_\delta(u-w) \diff w\diff u \nonumber\\
		&=\int_{-v}^0f(x,v,w)\diff w -\int_{-v}^0 f(x,v,w) \Big(\int_{-\infty}^{-t-w}\vartheta_\delta (u) \diff u \Big) \diff w. \label{renewal-3-1.5}
	\end{align}

	Since $t> \kappa+\delta$ by assumption, we have $-t-w\leq -t+v\leq -t+\kappa< -\delta$. Therefore,
	$$\int_{-\infty}^{-t-w}\vartheta_\delta (u) \diff u \leq \int_{-\infty}^{-\delta} \vartheta_\delta(u) \diff u   \lesssim \delta.$$ 
	Putting $x = hy$ and $v = \sigma_h(y)$ in \eqref{renewal-3-1.5}, we get that
	$$\Big| \int_{-t}^\infty f_\delta\big(hy,\sigma_h(y),u\big) \diff u - \int_{-\sigma_h^+(y)}^0  f\big(hy,\sigma_h(y),u\big) \diff u \Big| \lesssim \delta \int_{-\sigma_h^+(y)}^0 \big| f\big(hy,\sigma_h(y),u\big)\big| \diff u .$$
	Note that when $\sigma_h(y)<0$, the left hand side vanishes.
	
Since  $\sigma_h^+(y)\leq \log\norm{h}$ and $f$ is bounded by $M$, the last integral is bounded by $M \log\|h\|$. Integrating over $y$ and $h$ and using that $\int_G \log \norm{h} \diff \mu(h) < \infty$ gives the claim.  \endproof

Applying Proposition \ref{renewal-2} to $f_\delta$  and using Claims 1 and 2 yields
	\begin{align}
\Big|\oE (f_\delta)(x,t)-{1\over{\gamma}}\int_{\P^1}\int_G\int_{-\sigma_h^+(y)}^0 f\big(hy,\sigma_h(y),u\big) \diff u \diff \mu(h) \diff \nu(y)\Big|  
\lesssim \delta M+C_\delta {{M(\kappa+1)^2}\over (1+t)^{1/4}}  \label{renewal-3-4}
\end{align}
for some constant $C_\delta>0$ independent of $f,M,x,t$.
\vskip 5pt
	
	\noindent	\textbf{Claim 3.} 
There is a constant $C_\delta >0$ such that $\oE\big(|f_C-f_\delta|\big)(x,t)\lesssim \delta M+C_\delta M (1+t)^{-1/4}$.

\proof[Proof of Claim 3] We'll use  Lemmas \ref{renewal-3-lemma-3}  and \ref{renewal-3-lemma-2} above. Recall that $f_C(x,v,u)=\mathbf 1_{-v\leq u<0}f(x,v,u)$ and $\norm{f}_{\Cc^1}\leq M$. By Lemma \ref{renewal-3-lemma-3}, we have
\begin{align}
\oE\big(|f_C-f_\delta| \big)\lesssim M\oE\big(\delta \cdot\mathbf 1_{u\in[-v+\delta,-\delta]}\big) 
+M \oE &\big(\mathbf 1_{u\in[-v-\delta,-v+\delta]\cup [-\delta,\delta]}\big) \nonumber \\
&+ M \oE\big((  \mathbf 1_{u\in[-v,0]} *\vartheta_\delta)\cdot\mathbf 1_{u\notin[-v-\delta,\delta]}\big) .\label{renewal-3-lemma-3-1}
\end{align}
	
	For the first term in \eqref{renewal-3-lemma-3-1}, note that $u\in[-v+\delta,-\delta]$ implies $-v\leq u<0$. So  Lemma \ref{renewal-3-lemma-2} gives 
	$$\oE\big(\delta \cdot\mathbf 1_{u\in[-v+\delta,-\delta]}\big)\leq \oE\big(\delta\cdot\mathbf 1_{-v\leq u<0}\big)\lesssim \delta .$$
	
	For the second term in \eqref{renewal-3-lemma-3-1}, we observe that $\oE(\mathbf 1_{u\in [-\delta,\delta]})(x,t)=\oR(\mathbf 1_{u\in[-\delta,\delta]})(x,t)$ because $\mathbf 1_{u\in [-\delta,\delta]}$ is independent of $x$ and $v$. Thus,
	\begin{align*}
   &\oE\big(\mathbf 1_{u\in[-v-\delta,-v+\delta]\cup [-\delta,\delta]}\big)(x,t)\leq \oE(\mathbf 1_{u\in [-\delta,\delta]})(x,t)+\oE(\mathbf 1_{u\in[-v-\delta,-v+\delta]})(x,t)\\
	&=\oR(\mathbf 1_{u\in[-\delta,\delta]})(x,t)+\sum_{n\geq 0}\mu\otimes \mu^{*n}\big\{(h,g):\, -\sigma_h(gx)-\delta\leq \sigma_g(x)-t\leq -\sigma_h(gx)+\delta\big\}\\
	&=\oR(\mathbf 1_{u\in[-\delta,\delta]})(x,t)+\sum_{n\geq 0}\mu^{*(n+1)}\big\{g':\,-\delta\leq \sigma_{g'}(x)-t\leq \delta\big\}
	\\
	&\leq 2\oR(\mathbf 1_{u\in[-\delta,\delta]})(x,t)\lesssim \delta +C_\delta\delta(\delta+1) (1+t)^{-1/4}  \lesssim \delta+C_\delta (1+t)^{-1/4},
	\end{align*}
	where in the last line we have used Lemma \ref{renewal-3-lemma-1}, the assumption that $t > \kappa + \delta > \delta$ and that $\delta < 1$.
	
	For the last  term in \eqref{renewal-3-lemma-3-1}, we'll use the identity  $ \mathbf 1_{u\in[-v,0]}*\vartheta_\delta(u)=\int_u^{u+v}\vartheta_\delta (w) \diff w$. Notice that  $\sigma_g(x)< -\sigma_h(gx)-\delta$ if and only if $\sigma_{hg}(x)<-\delta$. We also have $w\notin [-\delta,\delta]$ when $u\notin [-v-\delta,\delta]$ and $u\leq w\leq u+v$. Therefore, using Fubini's theorem,  the last term in \eqref{renewal-3-lemma-3-1} is equal to
	\begin{align*}
&\sum_{n\geq 0}\int_{G^{2}} \mathbf 1_{\sigma_g(x)-t \notin [-\sigma_h(gx)-\delta,\delta]}\int_{\sigma_g(x)-t}^{\sigma_g(x)-t+\sigma_h(gx)} \vartheta_\delta(w) \,\diff w \diff \mu^{*n}(g) \diff \mu(h)  \\
& \leq \int_{w\notin [-\delta,\delta]} \vartheta_\delta(w) \sum_{n\geq 0}\int_{G^{2}}\mathbf 1_{\sigma_g(x)\leq w+t\leq \sigma_{hg}(x)} \,\diff \mu^{*n}(g) \diff \mu(h) \diff w \\
&=\int_{w\notin [-\delta,\delta]} \vartheta_\delta(w) \cdot\oE_1^+ \mathbf 1 (x,w+t)\,\diff w \lesssim \delta,
	\end{align*}
	where in the last inequality we used Lemma \ref{renewal-3-lemma-2} and the fact that $\int_{|w| \geq \delta}  \vartheta_\delta (w)\diff w \lesssim \delta$. We conclude that the right hand side of  \eqref{renewal-3-lemma-3-1} is $\lesssim \delta M+C_\delta M (1+t)^{-1/4}$, thus proving the claim. \endproof

The proof the proposition is completed by writing $\oE^+_1f = \oE f_C = \oE(f_C - f_\delta) + \oE(f_\delta)$ and coupling \eqref{renewal-3-4} with Claim 3.
\end{proof}

     Symmetrically,	we define $$\oE_1^-f(x,t):=\sum_{n\geq 0}\int_{\sigma_{hg}(x)<t\leq \sigma_g(x)} f\big(hgx,\sigma_h(gx),\sigma_g(x)-t\big)\diff \mu^{*n}(g)\diff \mu(h),$$
 which enjoys a similar asymptotic as $\oE_1^+$.
	
\begin{proposition} \label{renewal-3-remark}
		Let $f$ be a $\Cc^1$ function on $\P^1\times \R_v \times \R_u$ such  that $\norm{f}_{\Cc^1}\leq M$ and the projection of $\supp (f)$ to $\R_v$ is contained in $[-\kappa,\kappa]$ for some $\kappa>0$. Then, for all  $x\in\P^1$, $0<\delta< 1$ and $t>\kappa+\delta$, we have
	\begin{align*}
	\Big|\oE_1^-f(x,t)-{1\over{\gamma}}\int_{\P^1}\int_G \int_0^{\sigma_h^-(y)} f\big(hy,\sigma_h(y),u\big) \diff u \diff \mu(h) \diff \nu(y) \Big|
	\leq C_\delta{{M(\kappa+1)^2}\over (1+t)^{1/4}}+C\delta M,
	\end{align*} 
	where $\sigma_h^-(\,\cdot\,):=\max\big(-\sigma_h(\,\cdot\,),0\big)$, $C_\delta>0$ is a constant independent of $f,M,\kappa,x,t$, and $C>0$ is a constant independent of $\delta,f,M,\kappa,x,t$.
\end{proposition}
	
	\begin{proof}
		Define $f_C^-:=\mathbf 1_{0\leq u<-v} f$, then $\oE_1^- f(x,t)=\oE f_C^-(x,t)$. We can thus repeat the proof of Proposition \ref{renewal-3} in the current situation and obtain the desired asymptotic. 
		\end{proof}

\subsection{Second residual process with  cut-off}\label{subseq:E_2^+}
We now introduce a second residual operator with cut-off. There are a few important differences between this new operator and the operator $\oE^+_1$ from last section. Firstly, we consider the jumps of the norms $\log\norm{g}$ instead of jumps of the norm cocycle $\sigma_g(x)$. Secondly, this new operator also takes into account the reverse random walk on $\P^1$ induced by the probability measure $\check \mu$, which is by definition the image of $\mu$ under the inversion map $g \mapsto g^{-1}$.

Let $f(\check x,x,v,u)$ be a bounded function on $\P^1\times\P^1\times\R_v\times\R_u$. Define 
\begin{align*}
\oE_2^+ f(\check x,x,t):&=\sum_{n\geq 0}\int_{\log\norm{g}<t\leq \log\norm{hg}} f\Big((hg)^{-1}\check x,hgx,\log{\norm{hg}\over \norm{g}},\log\norm{g}-t\Big) \diff \mu^{*n}(g)\diff \mu(h).
\end{align*} 

Observe that, since $\mu$ is non-elementary and has a finite second moment, the same properties hold for $\check \mu$. We will denote by $\check \nu$ the corresponding stationary measure of $\check\mu$. In the next proposition, as in Proposition \ref{renewal-3}, we require conditions on $f$ and not on its Fourier transform.

\begin{proposition}\label{renewal-4}
		Let $f$ be a  $\Cc^1$ function on $\P^1\times\P^1\times \R_v \times \R_u$ such that $\norm{f}_{\Cc^1}\leq M$. Assume that the projection of $\supp(f)$ to $\R_v$ is contained in $[-\kappa,\kappa]$ for some $\kappa>0$.   Then, for all $\check x,x\in\P^1$, $0<\delta< 1$ and $t>2\kappa+2\delta$, we have
	\begin{align*}
	\Big|\oE_2^+ f(\check x,x,t)-{1\over{\gamma}}\int_{\P^1}\int_{\P^1}\int_G \int_{-\sigma_h^+(y)}^0 f\big(\check y,hy,\sigma_h(y),u\big) \diff u \diff \mu(h) \diff \nu(y)\diff \check\nu(\check y)\Big| \\
	\leq C_\delta{M(\kappa+1)^2\over (1+t)^{1/4}} +C\delta M+\varrho(t)M,
	\end{align*}
	 where $\varrho(t)$ is some rate function independent of $\delta,f,M,\kappa,\check x, x$, $C_\delta>0$ is a constant independent of $f,M,\kappa,\check x,x,t$, and $C>0$ is a constant independent of $\delta,f,M,\kappa,\check x, x,t$. 
\end{proposition}

The proof of Proposition \ref{renewal-4} is rather long and will follow several intermediate steps. We will prove the result by approximating $\oE_2^+$ by a number of operators, denoted below by $\widetilde\oE_2$, $\oE_T$ and $\oE_Q$. The last operator $\oE_Q$ is closely related to the operator $\oE_1^+$ and its asymptotic will be obtained using Proposition \ref{renewal-3} (see Lemma \ref{lemma:E_T^l-asymptotic} below). This will give the main term in the asymptotic of $\oE_2^+$. The error terms will come from the error term in Proposition \ref{renewal-3} and the intermediate approximations by $\widetilde\oE_2$, $\oE_T$ and $\oE_Q$.

\vskip5pt

From now until the end to this section, we fix $f$ as in the statement of Proposition \ref{renewal-4}. After dividing $f$ by $M$ we may assume that $\norm{f}_{\Cc^1}\leq 1$. From Lemma \ref{large-n}, it is easy to see that $\sum_{n\geq 0}\mu^{*n}\otimes \mu \big\{ \log\norm{g}<t\leq \log\norm{hg} \big\}$ is bounded   for $t\leq 3\gamma$. It follows that $\oE_2^+ f(\check x,x,t)$ is bounded for $t\leq 3\gamma$. On the other hand, the integral in the proposition is finite because $f(\check x,x,v,u)=0$ for $v\notin[-\kappa,\kappa]$. Hence we can assume   $t\geq 3\gamma$. 

Define the sets
\begin{equation*} 
\bN_t^+:=\bigsqcup_{n\geq 0} \big\{(h,g_n,\dots,g_1)\in G^{n+1}:\, \log\norm{\g}<t\leq \log\norm{h\g}\big\} \quad \text{and} \quad \bN_{t,n}^+:=\bN_t^+\cap G^{ n+1}.
\end{equation*}
From now on, we will write $\mu^{\otimes n}$ as $\mu^n$ to ease the notation.
Consider the operator
$$\widetilde\oE_2 f(\check x,x,t):=\sum_{n=\lceil 2t/(3\gamma)\rceil}^{ \lfloor 2t/\gamma\rfloor}\int_{\bN_{t,n}^+} f\Big((h\g)^{-1}\check x,h\g x,\log{\norm{h\g}\over \norm{\g}},\log\norm{\g}-t\Big) \diff \mu^{n}(\g)\diff \mu(h).$$

Observe that $\widetilde\oE_2$ is obtained from $\oE_2^+$ by considering only the sum for $n$   between $2t/(3\gamma)$ and $2t/\gamma$. Out of this range, the probability of falling in the set $\bN_{t,n}^+$ is small. More precisely, by Lemma \ref{large-n} applied to both $t$ and $t + \gamma / 2 $, there exists a rate function $\widetilde \varepsilon_0(t)$ such that $$ \sum_{n\leq \lfloor 2t/(3\gamma)\rfloor }\mu^{n+1}(\bN_{t,n}^+)+\sum_{n\geq \lceil 2t/\gamma\rceil}\mu^{n+1}(\bN_{t,n}^+)\leq \widetilde \varepsilon_0(t) \quad\text{for all} \quad t>0.$$ 
Since $|f|\leq 1$,  we deduce that 
\begin{equation}\label{diff-E_L}
\big|\oE_2^+ f(\check x,x,t)-\widetilde\oE_2 f(\check x,x,t)\big|\leq \norm{f}_{\infty} \, \widetilde \varepsilon_0(t)\leq \widetilde \varepsilon_0(t).
\end{equation}
Therefore, we can replace $\oE_2^+$ by $\widetilde\oE_2$ in the statement of Proposition \ref{renewal-4} without affecting the desired error term.  

\vskip5pt

We now set some more notation that will be used until the end of this section. Let $l:=\lfloor t/(3\gamma)\rfloor$ and, for $\g=(g_n,\cdots,g_1)\in G^{n}$, denote
$$\g_1:=(g_l,\dots, g_1) \quad\text{and} \quad \g_2:=(g_n,\dots,g_{l+1}).$$ 
We'll write the product $g_{n}\cdots g_1$ as $\g$ or $\g_2\g_1$ if there is no ambiguity. Define the set  $$\bT_n:=\bT_{n,x,t}:=\big\{(h,\g_2,\g_1)\in G^{n+1}:\, \sigma_{\g_2}(\g_1x)< t-\log\norm{\g_1} \leq\sigma_{h\g_2}(\g_1x)\big\},$$ 
and the operator
$$\oE_T f(\check x,x,t):=\sum_{n=\lceil 2t/(3\gamma)\rceil}^{ \lfloor 2t/\gamma\rfloor}\int_{\bT_{n}} f\big((h\g)^{-1}\check x,h\g x,\sigma_h(\g x),\sigma_{\g_2}(\g_1x)-t+\log\norm{\g_1}\big) \diff \mu^{n}(\g)\diff \mu(h).$$ 
Note that in the definition of $\oE_T$ we consider  the cocycle instead  the norms that appear in the expression of $\widetilde\oE_2$.  Also, in $\oE_T$ there's a shift by $\log \|\g_1\|$ both in the definition of $\bT_n$ and in the argument of $f$.

\begin{lemma}\label{renewal-4-lemma-1}
There exists a rate function $\varrho_1(t)$ such that, for all $\check x,x\in\P^1$, $0<\delta< 1$ and $t \geq 3\gamma$, we have
 $$\big|\widetilde\oE_2 f(\check x,x,t)-\oE_T f(\check x,x,t)\big|\lesssim \delta + C_\delta (1+t)^{-1/4}+\varrho_1(t),$$
 where $C_\delta>0$ is the constant in Lemma \ref{integral-R-l}.
\end{lemma}

\begin{proof}
In abridged notations, we have that $$\widetilde\oE_2 f  = \sum_{n=\lceil 2t/(3\gamma)\rceil}^{ \lfloor 2t/\gamma\rfloor} \int_{\bN_{t,n}^+} f_1 \, \diff \mu^{n}(\g)\diff \mu(h) \,\, \text{ and } \,\, \oE_T f  = \sum_{n=\lceil 2t/(3\gamma)\rceil}^{ \lfloor 2t/\gamma\rfloor}  \int_{\bT_n} f_2 \, \diff \mu^{n}(\g)\diff \mu(h),$$
where
$$ f_1 :=  f\big((h\g)^{-1}\check x,h\g x,\log\norm{h\g}-\log \norm{\g},\log\norm{\g}-t\big),$$ 
 $$ f_2: = f\big((h\g)^{-1}\check x,h\g x,\sigma_h(\g x),\sigma_{\g_2}(\g_1x)-t+\log\norm{\g_1}\big).$$

 The idea  is to find a set $\bS_n$ having high probability where the difference between $f_1$ and $f_2$ can be effectively compared and use the elementary estimate
\begin{equation} \label{eq:int-f1-int-f2}
\Big|\int_{\bN_{t,n}^+}f_1-\int_{\bT_n}f_2\Big|\leq \int_{\bN_{t,n}^+\setminus (\bN_{t,n}^+ \cap \bS_n \cap \bT_n)}|f_1|+\int_{\bT_n  \setminus(\bN_{t,n}^+ \cap \bS_n  \cap \bT_n)}|f_2|+\int_{\bN_{t,n}^+\cap \bS_n  \cap \bT_n} |f_1-f_2|.
\end{equation}

We will use Lemma \ref{diff-log}. Let $\bS_{n+1,l,x}\subseteq G^{n+1-l}\times G^{ l}$, $\bS_{n,l,x}\subseteq G^{n-l}\times G^{ l}$, $\bS_{n+1,n,x}\subseteq G\times G^n$ be as in that lemma. Notice that $\mu^{*n}$ corresponds to $\mu^n$ in our situation. Define $$\bS_n:=\bS_{n+1,l,x} \cap (G\times \bS_{n,l,x})\cap \bS_{n+1,n,x} \subset G^{n+1}.$$ Then, by Lemma \ref{diff-log}, $\mu^{n+1}(\bS_n)\geq 1-2D_{n+1}-2D_n-2D_l$ and  for all $(h,\g_2,\g_1)\in \bS_n$ one has 
	\begin{align}
	&\big|\log\norm{\g_2\g_1}-\sigma_{\g_2}(\g_1x)-\log\norm{\g_1}\big|\leq 4e^{-\gamma l}, \label{S-n-1}\\
    &\big|\log\norm{h\g_2\g_1}-\sigma_{h\g_2}(\g_1x)-\log\norm{\g_1}\big|\leq 4e^{-\gamma l}, \label{S-n-2} \\
    &\big| \log\norm{h\g}-\sigma_{h}(\g x)-\log\norm{\g} \big| \leq 4e^{-\gamma n}\leq  4e^{-\gamma l}. \label{S-n-3} 
	\end{align}

	From \eqref{S-n-1} and \eqref{S-n-2}, it follows that
	$$\bN_{t,n}^+\cap \bS_n\subseteq \big\{\sigma_{\g_2}(\g_1x)< t+4e^{-\gamma l}-\log\norm{\g_1},\, \sigma_{h\g_2}(\g_1x)\geq t-4e^{-\gamma l}-\log\norm{\g_1}\big\},$$ 
	$$\bN_{t,n}^+ \supseteq \bS_n\cap \big\{\sigma_{\g_2}(\g_1x)< t-4e^{-\gamma l}-\log\norm{\g_1},\, \sigma_{h\g_2}(\g_1x)\geq t+4e^{-\gamma l}-\log\norm{\g_1}\big\}.$$
	Hence, we have	$(\bN_{t,n}^+\cap \bS_n) \setminus \bT_n\subseteq \bA_n $, where 
	\begin{align*}
	\bA_n:=\big\{t-\log\norm{\g_1}\leq\sigma_{\g_2}(\g_1x)&\leq t-\log\norm{\g_1}+4e^{-\gamma l}\big\} \cup \\
	     &\big\{t-\log\norm{\g_1}- 4e^{-\gamma l}\leq      \sigma_{h\g_2}(\g_1x)    \leq t-\log\norm{\g_1}       \big\},
	\end{align*}
	and $\bT_n\cap \bS_n\setminus \bN_{t,n}^+\subseteq \bB_n$, where
		\begin{align*}
	\bB_n:=\big\{t-\log\norm{\g_1}-4e^{-\gamma l}\leq\sigma_{\g_2}&(\g_1x)\leq t-\log\norm{\g_1}\big\} \cup \\
	&\big\{t-\log\norm{\g_1}\leq      \sigma_{h\g_2}(\g_1x)    \leq t-\log\norm{\g_1}  + 4e^{-\gamma l}     \big\}.
	\end{align*}
	
	Therefore, we obtain
	\begin{align*}
	\mu^{n+1}\big(\bN_{t,n}^+\setminus (\bN_{t,n}^+\cap \bS_n\cap \bT_n)\big)&\leq \mu^{n+1} (\bN_{t,n}^+\setminus \bS_n) +       \mu^{n+1}  \big(\bN_{t,n}^+\cap \bS_n\setminus \bT_n\big)      \\ &\leq 2D_{n+1}+2D_n+2D_l+  \mu^{ n+1}(\bA_n)
	\end{align*}
	because  $\mu^{n+1}(\bS_n)\geq 1-2D_{n+1}-2D_n-2D_l$. Similarly,  
	$$\mu^{n+1}\big(\bT_n\setminus (\bN_{t,n}^+\cap \bS_n\cap \bT_n)\big)\leq 2D_{n+1}+2D_n+2D_l+          \mu^{n+1}(\bB_n).$$
	
	Now, using the mean value theorem, the fact that $\|f\|_{\Cc^1} \leq 1$ and estimates \eqref{S-n-1} and \eqref{S-n-3} we get that, for all $(h,\g_2,\g_1)\in \bS_n$ and every $\check x,x,t$,
\begin{align*}
|f_1-f_2|= \big|f\big((h\g)^{-1}\check x,h\g x,&\log\norm{h\g}-\log \norm{\g},\log\norm{\g}-t\big)\,-  \\
   &f\big((h\g)^{-1}\check x,h\g x,\sigma_h(\g x),\sigma_{\g_2}(\g_1x)-t+\log\norm{\g_1}\big)\big|\lesssim e^{-\gamma l}.
 \end{align*}
 By using inequality  \eqref{eq:int-f1-int-f2}, we deduce that $\big|\widetilde\oE_2 f(\check x,x,t)-\oE_T f(\check x,x,t)\big|$ is bounded by a  constant times
 \begin{align*}
 &   \sum_{n=\lceil 2t/(3\gamma)\rceil}^{\lfloor 2t/\gamma\rfloor}  \Big(e^{-\gamma l} +  D_{n+1}+D_n+D_l+          \mu^{n+1}(\bA_n)       +    \mu^{n+1}(\bB_n) \Big) \\
&\lesssim te^{-\gamma l}+\sum_{n\geq\lceil 2t/(3\gamma)\rceil}D_n+tD_l+\sum_{n\geq\lceil 2t/(3\gamma)\rceil}\big(\mu^{n+1}(\bA_n)+\mu^{n+1}(\bB_n) \big).
 \end{align*}

  For the last term, notice that $\mu^{n+1}(\bA_n)$ and $\mu^{n+1}(\bB_n)$ both are bounded by
 \begin{align*}
\mu^{ n}  \big\{(\g_2,\g_1):\, -4e^{-\gamma l}\leq  &\sigma_{\g_2}(\g_1x) -t+  \log\norm{\g_1} \leq   4e^{-\gamma l}  \big\}\,+ \\
 &\mu^{n+1}  \big\{(h,\g_2,\g_1):\,-4e^{-\gamma l}\leq  \sigma_{h\g_2}(\g_1x) -t+  \log\norm{\g_1} \leq   4e^{-\gamma l}  \big\}.
 \end{align*}
 Hence, 
 $$\sum_{n\geq l}\big(\mu^{n+1}(\bA_n)+\mu^{n+1}(\bB_n)\big)\leq 4\int_{G^l} \oR(\mathbf 1_{u\in [-4e^{-\gamma l},4e^{-\gamma l}]})\big(\g_1x,t-\log\norm{\g_1}\big) \diff \mu^{l}(\g_1),$$
where $\oR$ is the first renewal operator introduced in Subsection \ref{subsec:R}. 
Using that $\oR(\mathbf 1_{u\in [-4e^{-\gamma l},4e^{-\gamma l}]})\leq \oR(\mathbf 1_{u\in [-4e^{-\gamma l}-\delta,4e^{-\gamma l}+\delta]})$ and applying Lemma \ref{integral-R-l}, we get that last integral is bounded by some constant times
 $$C_l(e^{-\gamma l}+\delta+1)^2+e^{-\gamma l}+\delta +C_\delta{ (e^{-\gamma l}+\delta+1)^2 \over  (1+t)^{1/4}}\lesssim C_l +e^{-\gamma l}+\delta + {C_\delta\over (1+t)^{1/4}}.$$

Finally, set $\varrho_1(t):=te^{-t/3}+\sum_{n\geq\lceil 2t/(3\gamma)\rceil}D_n+tD_l+C_l$. This is a rate function by Proposition \ref{prop:BQLDT} and Lemma \ref{diff-log}. Recall that $l:=\lfloor t/(3\gamma)\rfloor$. From the preceding estimates we have that 
$$\big|\widetilde\oE_2 f(\check x,x,t) -\oE_T f(\check x,x,t) \big|\lesssim \delta + C_\delta  (1+t)^{-1/4}+\varrho_1(t).$$ This concludes the proof of the lemma.
\end{proof}

Consider now the sets $$\bQ_{l}:=\bQ_{l,\check x}:=\big\{\g_1\in G^{l}:\,  \big|\log\norm{\g_1} -l\gamma \big|\leq \gamma l/4,\, d(z^M_{\g_1},\check x)> e^{-\gamma l/4} \big\},  $$ 
$$ \bQ_{n,l}:=\bQ_{n,l,\check x,x,t}:=(G^{n-l+1}\times \bQ_{l}) \cap \bT_{n},$$
where $z^M_{\g_1}$ is the density point defined in Section \ref{sec:prelim}.

Observe that $\bQ_{n,l}$ contains only the words $h\g_2 \g_1$ in $\bT_n$ for which $\log \norm{\g_1}$ is close to the expected value $l\gamma$ and $z^M_{\g_1}$  is not to too close to $\check{x}$. By Proposition \ref{prop:BQLDT}, this happens with high probability.
 
Define also the operator
$$\oE_Q f(\check x,x,t):=\sum_{n=\lceil 2t/(3\gamma)\rceil}^{ \lfloor 2t/\gamma\rfloor}\int_{\bQ_{n,l}} f\big((h\g)^{-1}\check x,h\g x,\sigma_h(\g x),\sigma_{\g_2}(\g_1x)-t+\log\norm{\g_1}\big) \diff \mu^{n}(\g)\diff \mu(h),$$ which differs from $\oE_T$ only in its domain of integration.

\begin{lemma} \label{lemma:diff-E_L^T}
There exists a rate function $\varrho_2(t)$ such that
\begin{equation*}
\big|\oE_Q f(\check x,x,t)-\oE_T f(\check x,x,t)\big| \lesssim \varrho_2(t) \quad\text{for all}\quad \check x,x\in\P^1,\, t>0.
\end{equation*}
\end{lemma}

\begin{proof}
As mentioned above, $\oE_Q$ and  $\oE_T$  differ only by their domain of integration ($\bQ_{n,l}$ and $\bT_n$ respectively). Note that $\bQ_{n,l} \subset \bT_n$. By the definition of these sets, one has
 $$\bT_n\setminus \bQ_{n,l}=\big\{(h,\g_2,\g_1)\in G^{n+1}:\,\sigma_{\g_2}(\g_1x)< t-\log\norm{\g_1} \leq\sigma_{h\g_2}(\g_1x),\, \g_1\notin \bQ_{l}\big\}.$$
By Proposition \ref{prop:BQLDT}, we get that $\mu^l(\bQ_{l})\geq 1-2C_l$, where the constants $C_l$ correspond to the choice $\ep=\gamma/4$.
We now can use the residual operator $\oE_1^+$ from Subsection \ref{subsec:E_1^+} and Lemma \ref{renewal-3-lemma-2} from the same section to obtain the estimate
$$\sum_{n\geq l}\mu^{n+1}(\bT_n\setminus \bQ_{n,l})\leq \int_{\g_1\notin \bQ_{l}}\oE_1^+ \mathbf 1\big(\g_1 x, t-\log\norm{\g_1}\big) \diff \mu^{l}(\g_1)\lesssim \mu^l(G^l\setminus \bQ_{l})\lesssim C_l.$$

Recall that $l=\lfloor t/(3\gamma)\rfloor$, so $C_l$ tends to zero as $t$ tends to infinity. Setting
$\varrho_2(t):=C_l$ and recalling that $|f| \leq 1$, we get that $\big|\oE_Q f(\check x,x,t)-\oE_T f(\check x,x,t)\big|\lesssim \varrho_2(t)$ and  $\varrho_2(t)$ is a rate function. This completes the proof of the lemma.
\end{proof}

The next step in the proof of Proposition \ref{renewal-4} is to obtain the asymptotic of the operator $\oE_Q$.  We'll do so by relating it with the operator  $\oE_1^+$ from last subsection and using Proposition \ref{renewal-3}.  
 In order to achieve that, define the operator
$$\widetilde\oE_3 f(\check x,x,t):=\sum_{n'=\lceil 2t/(3\gamma)\rceil-l}^{\lfloor 2t/\gamma\rfloor-l} \int_{\sigma_\g(x)<t\leq \sigma_{h\g}(x)} f\big((h\g)^{-1}\check x,h\g x,\sigma_h(\g x),\sigma_\g(x)-t \big)\diff \mu^{n'}(\g)\diff \mu(h). $$
In comparison with $\oE_1^+$, in the definition of $\widetilde\oE_3$, the sum over $n'$ is shifted by $l$ and we also take into account the inverse action $(h\g)^{-1}$ on the variable $\check x$. 

\medskip

For each $\check x\in \P^1$, we denote by $f_{\check x}$ the function on $\P^1\times \R_v \times \R_u$ given by $f_{\check x}(x,v,u):=f(\check x,x,v,u)$. We also define the function $F_{\g}$ on  $\P^1\times\P^1\times \R_v \times \R_u$ by $F_{\g}(\check x,x,v,u):=f(\g\check x,x,v,u)$. By replacing $\g,n'$  by $\g_2, n-l$ in the definition of $\widetilde\oE_3$, we obtain
\begin{equation}\label{integral-G_l}
\oE_Q f(\check x,x,t)=\int_{\bQ_{l}} \widetilde\oE_3 F_{\g_1^{-1}}\big(\check x,\g_1 x, t-\log\norm{\g_1}\big) \diff \mu^l(\g_1).
\end{equation}
We warn that $\widetilde\oE_3 F_{\g_1^{-1}}\big(\check x,\g_1 x, t-\log\norm{\g_1}\big)$ is not equal to $\widetilde\oE_3 f\big(\g_1^{-1}\check x,\g_1 x, t-\log\norm{\g_1}\big)$.

\begin{lemma}\label{diff-E_C}
There exists a rate function $\varrho_3(t)$ such that for all  $\g_1\in \bQ_l$, $\check x,x\in\P^1$ and $t>0$,  one has
	$$\big|\widetilde\oE_3 F_{\g_1^{-1}}\big(\check x,\g_1x,t-\log\norm{\g_1}\big)-\oE_1^+ f_{\g_1^{-1} \check x}\big(\g_1x,t-\log\norm{\g_1}\big)\big|\lesssim  \varrho_3(t).$$ 
\end{lemma}

\begin{proof}
	Set
	$$\widetilde\oE_1 f(x,t):=\sum_{n'=\lceil 2t/(3\gamma)\rceil-l}^{ \lfloor 2t/\gamma\rfloor-l} \int_{\sigma_{\g}(x)<t\leq \sigma_{h\g}(x)} f\big(h\g x,\sigma_h(\g x),\sigma_{\g}(x)-t\big) \diff \mu^{n'}(\g)\diff \mu(h),$$
	which is a truncation $\oE_1^+ f(x,t)$.
	
	Fix a $\g_1\in \bQ_l$. Arguing as in the proof of the estimate \eqref{diff-E_L}, we obtain 
	\begin{equation}\label{d_1}
	\big|\widetilde\oE_1 f_{\g_1^{-1} \check x}\big(\g_1 x,t-\log\norm{\g_1}\big)-\oE_1^+ f_{\g_1^{-1} \check x}\big(\g_1 x,t-\log\norm{\g_1}\big)\big|\leq \eta_1(t),
	\end{equation}
	for some rate function $\eta_1(t)$.
	
	Consider the following sets, which depend on $\g_1$:
	$$\bQ^{\g_1}_{n-l}:=\big\{(h,\g_2)\in G^{n+1-l}:\, (h,\g_2,\g_1)\in \bQ_{n,l}\big\},$$
	$$\bP^{\g_1}_{n-l}:=\bQ^{\g_1}_{n-l}\cap \big\{(h,\g_2)\in G^{n+1-l}:\, d\big(z^m_{\g_1^{-1}},(h\g_2)^{-1}\check x\big)> e^{-\gamma (n+1-l)/24}\big\}.$$
	Then in abridged notations, we have that for every fixed $\g_1\in\bQ_l$,
	$$\widetilde\oE_3 F_{\g_1^{-1}}\big(\check x,\g_1x,t-\log\norm{\g_1}\big) = \sum_{n=\lceil 2t/(3\gamma)\rceil}^{ \lfloor 2t/\gamma\rfloor} \int_{\bQ^{\g_1}_{n-l}} f_3 \, \diff \mu^{n-l}(\g_2)\diff \mu(h),$$  
	$$ \widetilde\oE_1 f_{\g_1^{-1} \check x}\big(\g_1 x,t-\log\norm{\g_1}\big) = \sum_{n=\lceil 2t/(3\gamma)\rceil}^{ \lfloor 2t/\gamma\rfloor}  \int_{\bQ^{\g_1}_{n-l}} f_4 \, \diff \mu^{n-l}(\g_2)\diff \mu(h),$$
	where
	$$ f_3 :=  f\big((h\g_2\g_1)^{-1}\check x,h\g_2\g_1 x,\sigma_h(\g_2\g_1 x),\sigma_{\g_2}(\g_1x)-t+\log\norm{\g_1}\big),$$ 
	$$   f_4 := f\big(\g_1^{-1}\check x,h\g_2\g_1 x,\sigma_h(\g_2\g_1 x),\sigma_{\g_2}(\g_1x)-t+\log\norm{\g_1}\big).$$
	
	Similarly to the proof of Lemma \ref{renewal-4-lemma-1}  we'll see that $\bP_{n-l}^{\g_1}$ has high probability in  $\bQ^{\g_1}_{n-l}$ and $f_3$ and $f_4$ can be compared over $\bP_{n-l}^{\g_1}$. We'll then apply the elementary estimate
	\begin{equation} \label{eq:int-f1-int-f2-2}
	\Big|\int_{\bQ^{\g_1}_{n-l}}f_3-\int_{\bQ^{\g_1}_{n-l}}f_4\Big|\leq \int_{\bP^{\g_1}_{n-l}}|f_3-f_4|+ \int_{\bQ^{\g_1}_{n-l}\setminus \bP^{\g_1}_{n-l}}|f_3|+ \int_{\bQ^{\g_1}_{n-l}\setminus \bP^{\g_1}_{n-l}}|f_4|.
	\end{equation}

	Applying Proposition \ref{prop:BQLDT} with $\ep=\gamma/24$, we get $\mu^{n+1-l}\big(\bQ^{\g_1}_{n-l}\setminus \bP^{\g_1}_{n-l}\big)\leq C_{n+1-l}$. 
	Since $\g_1\in \bQ_{l}$, we have
	$ \big|\log\norm{\g_1} -l\gamma\big|\leq \gamma l/4$ and $ d(z^m_{\g_1^{-1}},\check x)\geq e^{-\gamma l/4}$. Let $\check x=[\check w]\in\P^1$ with $\check w \in \C^2 \setminus \{0\}$. Then, by \eqref{d-gxgy} and \eqref{g^-2}, we obtain that, for $(h,\g_2)\in \bP^{\g_1}_{n-l}$, one has
	\begin{align*}
	&d\big((h\g_2\g_1)^{-1}\check x,\g_1^{-1}\check x\big)=  d\big((h\g_2)^{-1}\check x,\check x\big)\cdot{\norm{(h\g_2)^{-1}\check w}\over\norm{\g_1^{-1}(h\g_2)^{-1}\check w}}\cdot {\norm{\check w}\over\norm{\g_1^{-1}\check w}}  \\
	&\leq 1\cdot{1\over d(z^m_{\g_1^{-1}},(h\g_2)^{-1}\check x)}\cdot{1\over d(z^m_{\g_1^{-1}},\check x)}\cdot{1\over \norm{\g_1}^2}\leq e^{\gamma (n+1-l)/24}e^{\gamma l/4}e^{-3\gamma l/2}\lesssim e^{-t/(3\gamma)},
	\end{align*}
where in the last step we have used that $n\leq 2t/\gamma$ and $l=\lfloor t/(3\gamma)\rfloor$.
	
	By the mean value theorem and the fact that $\norm{f}_{\Cc^1}\leq 1$, we conclude that $|f_3-f_4| \lesssim e^{-t/(3\gamma)}$ over $\bP^{\g_1}_{n-l}$.	Then, using \eqref{eq:int-f1-int-f2-2}, we deduce that
	\begin{align*}
	&\big|\widetilde\oE_3 F_{\g_1^{-1}}\big(\check x,\g_1x,t-\log\norm{\g_1}\big)-\widetilde\oE_1 f_{\g_1^{-1} \check x}\big(\g_1 x,t-\log\norm{\g_1}\big)\big| \\
	&\lesssim \sum_{n= \lceil 2t/(3\gamma)\rceil} ^{ \lfloor 2t/\gamma\rfloor}    \Big(e^{-t/(3\gamma)}  +\mu^{n+1-l}\big(\bQ^{\g_1}_{n-l}\setminus \bP^{\g_1}_{n-l} \big)\Big)     
	\lesssim te^{-t/(3\gamma)}+\sum_{n\geq \lceil 2t/(3\gamma)\rceil}  C_{n+1-l}=:\eta_2(t).
	\end{align*}
	Recall that $l=\lfloor t/(3\gamma)\rfloor$, so  we have $\eta_2(t) \to 0$ as $t\to \infty$. 
	Combining with \eqref{d_1}, we get 
	$$\big|\widetilde\oE_3 F_{\g_1^{-1}}\big(\check x,\g_1x,t-\log\norm{\g_1}\big)-\oE_1^+ f_{\g_1^{-1} \check x}\big(\g_1 x,t-\log\norm{\g_1}\big)\big|\lesssim \eta_1(t)+\eta_2(t).$$
	Taking $\varrho_3(t):=\eta_1(t)+\eta_2(t)$ gives the desired estimate and finishes the proof of the lemma.
\end{proof}

We can now obtain the desired asymptotic for $\oE_Q$.

\begin{lemma} \label{lemma:E_T^l-asymptotic}
Under the assumptions of Proposition \ref{renewal-4} with $M=1$, we have
\begin{align*}
 \Big| \oE_Q f(\check x,x,t)- {1\over{\gamma}} \int_{\bQ_{l}} \int_{\P^1}\int_G \int_{-\sigma_h^+(y)}^0 f\big(\g_1^{-1}\check x,hy,\sigma_h(y),u\big)  \diff u \diff \mu(h) \diff \nu(y) \diff \mu^l(\g_1)  \Big|\\
  \lesssim C_\delta{(\kappa+1)^2\over (1+t)^{1/4}} +\delta +\varrho_3(t)
 \end{align*}
 for all $\check x,x\in\P^1$ and $0<\delta< 1, t>2\kappa+2\delta$,  where $\varrho_3(t)$ is the rate function from Lemma \ref{diff-E_C} and $C_\delta>0$ is a constant independent of $f,\kappa,\check x,x,t$.
\end{lemma}

\begin{proof}
Recall that $\norm{f}_{\Cc^1}\leq 1$. Then $\norm{f_{\g_1^{-1}\check x}}_{\Cc^1}\leq 1$ for all $\g_1\in \bQ_{l}$ and $\check x\in \P^1$. In particular, $f_{\g_1^{-1}\check x}$ satisfies all the conditions of Proposition \ref{renewal-3}.  Moreover, for $\g_1\in \bQ_{l}$, one has $\log\norm{\g_1}\leq 5\gamma l/4\leq t/2$. Combining with $t>2\kappa+2\delta$ gives $t-\log\norm{\g_1}>\kappa+\delta$. Therefore, we  can apply  Proposition \ref{renewal-3}  and get that, for every $\g_1\in\bQ_l$ and $\check x\in\P^1$,
\begin{align*}
\Big| \oE_1^+ f_{\g_1^{-1} \check x}  \big(\g_1 x,t-\log\norm{\g_1}\big)- {1\over{\gamma}}\int_{\P^1}\int_G \int_{-\sigma_h^+(y)}^0 f_{\g_1^{-1}\check x} \big(hy,\sigma_h(y),u\big) \diff u \diff \mu(h) \diff \nu(y)     \Big|\\
\lesssim C_\delta{(\kappa+1)^2\over (1+t-\log\norm{\g_1})^{1/4}}+\delta \lesssim  C_\delta{(\kappa+1)^2\over (1+t)^{1/4}} +\delta .
\end{align*}

Applying Lemma \ref{diff-E_C}, integrating over $\g_1\in \bQ_{l}$ and using the identity \eqref{integral-G_l} yields the desired result.
\end{proof}

We are now in position to finish the proof of Proposition \ref{renewal-4}.

\begin{proof}[End of the proof of Proposition \ref{renewal-4}] Notice that the main term in the asymptotic given by Lemma \ref{lemma:E_T^l-asymptotic} almost coincides with the main term in the statement of Proposition \ref{renewal-4}. The difference is the integral of $\g_1^{-1}\check x$ against $\mu^l$ in the lemma in contrast with the desired integral of $\check y$ against $\check \nu$. This is handled by the  equidistribution property of Theorem \ref{thm:equi-dis}.

Applying that theorem to $\check\mu,\check\nu$ instead of $\mu,\nu$, we get 
 $$\Big|\int_{G^l}\varphi(\g_1^{-1}\check x) \,\diff \mu^l(\g_1) - \int_{\P^1} \varphi \, \diff \check\nu \Big|\lesssim \lambda^l\norm{\varphi}_{\Cc^1}$$
 for every $\Cc^1$ function $\varphi$ on $\P^1$, where $0<\lambda<1$ is independent of $\varphi$. Together with the fact that $\mu^l(\bQ_{l}) \geq 1 - 2C_l$ this gives that 
$$\Big|\int_{\bQ_{l}}\varphi(\g_1^{-1}\check x) \,\diff \mu^l(\g_1) - \int_{\P^1} \varphi \,\diff \check\nu \Big|\lesssim \lambda^l\norm{\varphi}_{\Cc^1}+C_l\norm{\varphi}_{\infty}.$$
 Applying the above estimate to the function $\varphi(\,\cdot\,):= f\big(\,\cdot\,,hy,\sigma_h(y),u\big)$, whose $\Cc^1$-norm is bounded by $1$ uniformly,  gives that, for every $h,y,u$, one has
 $$\Big| \int_{\bQ_l} f\big(\g_1^{-1}\check x,hy,\sigma_h(y),u\big) \diff \mu^l(\g_1) - \int_{\P^1}f\big(\check y,hy,\sigma_h(y),u\big) \diff \nu(\check y)\Big|\lesssim \lambda^l+C_l.$$
We deduce, using Lemma \ref{lemma:E_T^l-asymptotic}, that
 \begin{align}
 \Big| \oE_Q f(\check x,x,t)- {1\over{\gamma}}  \int_{\P^1}\int_{\P^1}\int_G \int_{-\sigma_h^+(y)}^0 f\big(\check y,hy,\sigma_h(y),u\big) \diff u \diff \mu(h) \diff \nu(y) \diff \check\nu(\check y)  \Big|\nonumber\\
 \lesssim C_\delta{(\kappa+1)^2\over (1+t)^{1/4}} +\delta +\varrho_3(t)+\kappa (\lambda^l+C_l) \label{diff-last} ,
 \end{align}
  where the extra factor $\kappa$ comes from the fact that $f\big(y',hy,\sigma_h(y),u\big) =0$ when  $\sigma_h(y)\notin [-\kappa,\kappa]$, so only the case $\sigma_h^+(y)\leq\kappa$ contributes to the integral.

  Coupling \eqref{diff-last} with Lemmas \ref{lemma:diff-E_L^T}  and  \ref{renewal-4-lemma-1}, inequality \eqref{diff-E_L} and recalling that  $t>2\kappa$, finally gives

  \begin{align*}
  \Big| \oE_2^+ f(\check x,x,t)- {1\over{\gamma}}  \int_{\P^1}\int_{\P^1}\int_G \int_{-\sigma_h^+(y)}^0 f\big(\check y,hy,\sigma_h(y),u\big) \diff u \diff \mu(h) \diff \nu(y) \diff \check\nu(\check y)  \Big|\nonumber\\
  \lesssim C_\delta{(\kappa+1)^2\over (1+t)^{1/4}} +\delta +\widetilde \varepsilon_0(t)+\varrho_1(t)+\varrho_2(t)+\varrho_3(t)+t\lambda^l+tC_l.
  \end{align*}
  
  Setting  $\varrho(t):=\widetilde \varepsilon_0(t)+\varrho_1(t)+\varrho_2(t)+\varrho_3(t)+t\lambda^l+tC_l,$ it is clear that $\lim_{t\to\infty} \varrho(t) =0$ and the desired asymptotic follows.  This completes the proof of Proposition \ref{renewal-4}. 
\end{proof}

Before ending this section, we introduce the ``symmetric'' version of $\oE_2^+ $ as we did with with $\oE_2^+$ earlier.
Define 
\begin{align*}\oE_2^- f(\check x,x,t):&=\sum_{n\geq 0}\int_{ \log\norm{hg}<t\leq \log\norm{g}} f\Big((hg)^{-1}\check x,hgx,\log{\norm{hg}\over \norm{g}},\log\norm{g}-t\Big) \diff \mu^{*n}(g)\diff \mu(h) .
\end{align*} 

We have the following analogue of Proposition \ref{renewal-4}.

\begin{proposition} \label{E_L^-}

	Let    $f$ be a $\Cc^1$ function on $\P^1\times\P^1\times \R_v\times \R_u$ such that  $\norm{f}_{\Cc^1}\leq M$. Assume the projection of $\supp (f)$ to $\R_v$ is contained in $[-\kappa,\kappa]$ for some $\kappa>0$. Then for all $\check x,x\in\P^1$ and $0<\delta< 1,t>2\kappa+2\delta$, we have
	\begin{align*}
	\Big|\oE_2^-f(\check x,x,t)-{1\over{\gamma}}\int_{\P^1}\int_{\P^1}\int_G \int_0^{\sigma_h^-(y)} f\big(\check y,hy,\sigma_h(y),u\big) \diff u \diff \mu(h) \diff \nu(y)\diff \check\nu(\check y)\Big| \\
	\leq C_\delta{M(\kappa+ 1)^2 \over (1+t)^{1/4}} +C\delta M+\varrho(t)M,\end{align*}
	 where  $\varrho(t)$ is a rate function independent of $\delta,f,M,\kappa,\check x,x$, $C_\delta>0$ is a constant independent of $f,M,\kappa,\check x,x,t$, and $C>0$ is a constant independent of $\delta,f,M,\kappa,\check x, x,t$.  
\end{proposition}

\begin{proof}
	Define
	\begin{equation*} 
	\bN_t^-:=\bigsqcup_{n\geq 0} \big\{(h,g_n,\dots,g_1)\in G^{n+1}:\, \log\norm{h\g}<t\leq \log\norm{\g}\big\} .
	\end{equation*}

 We repeat the proof of Proposition \ref{renewal-4} by using $\oE_1^-,\bN_t^-$ instead of $\oE_1^+,\bN_t^+$ and applying Proposition \ref{renewal-3-remark} instead of Proposition \ref{renewal-3}.
\end{proof}

\section{Decay of Fourier coefficients} \label{sec:fourier}

This section is devoted to the proof of Theorem \ref{thm:main-fourier}. We first present the main part of the proof, up to the proof of Proposition \ref{x-y>es}, which is the central and most difficult estimate. The proof of this last result is rather long and is discussed separately in Subsection \ref{subsec:fourier-B} below.

\vskip5pt

We'll obtain here the following more general version of Theorem \ref{thm:main-fourier}. Recall from the introduction that we identify $\R \P^1 \subset \P^1$ with a circle $\bC$ and that the action of $g\in\SL_2(\R)$ preserves $\bC$. 

\begin{theorem}\label{thm:fourier-general}
Let $\mu$ be a non-elementary probability measure on $G_\R = \SL_2(\R)$. Assume $\mu$ has a finite second moment, that is, $\int_{G_\R} \log^2 \|g\| \, \diff \mu(g) < \infty$. Let $\nu$ be the associated stationary measure on $\bC \subset \P^1$. Let $c>0$ be a constant and $\varphi\in\Cc^2(\bC )$, $\psi\in\Cc^1(\bC )$ be real-valued functions such that
	$$|\varphi'|\geq 1/c>0 \,\, \text{ on } \,\,  \supp (\psi), \quad \norm{\varphi}_{\Cc^2(\bC )}\leq c \quad \text{ and }   \quad \norm{\psi}_{\Cc^1(\bC )}\leq c.$$
Then we have for $k\in\R$,
	$$\lim_{|k|\to\infty} \int_{\bC }e^{ik\varphi}\psi \, \diff \nu = 0.$$
Moreover, the convergence is uniform in $\varphi,\psi$ satisfying the above conditions.
\end{theorem} 

\subsection{The proofs of Theorem \ref{thm:fourier-general} and Theorem \ref{thm:main-fourier}}  \label{subsec:fourier-A}

Before getting to the actual proof of Theorem \ref{thm:fourier-general}, let us highlight its main ideas. We follow \cite{li:fourier}.  Our goal is to estimate the integral $\int_{\bC }e^{ik\varphi(x)}\psi(x) \, \diff \nu(x)$. The analysis consists of two main steps. First, we use the invariance of $\nu$ and a parameter $t>0$ in order to express the last integral as a difference between two integrals of the type $$\int_{\P^1} \int_{ \bM^\pm_t} e^{ik \varphi(\g x)} \psi(\g x) \,\diff \mu^\N(\g) \diff \nu(x),$$ along sets $\bM^\pm_t$ where the quantities $\log\|g_n \cdots g_1\|$ cross the value $t$. See Proposition \ref{decomp} below. The second step is to use Cauchy-Schwarz inequality to estimate these integrals in terms of an integral of the form $$\int_{\bC}\int_{\bC} \int_{\bM^\pm_t} e^{ik\varphi(\g x)-ik\varphi(\g y)}\psi(\g x)\psi(\g y)\, \diff \mu^\N(\g)\diff \nu(x)\diff \nu(y),$$
see \eqref{inequality-Nt} below.

The most technical part of the proof consists in estimating this last integral. This is done in Subsection \ref{subsec:fourier-B}. The key idea is to replace the integrand by a suitable approximation allowing us to  translate the resulting quantity in terms of the operator $\oE_2^+$ studied in Section \ref{sec:renewal}. The desired estimate will then follow by applying Proposition \ref{renewal-4}, after carefully choosing the parameter $t$ and other auxiliary parameters.

\medskip

We now begin  the  proof of Theorem \ref{thm:fourier-general}. We start by setting some notation. Recall that $G = \SL_2(\C)$. Let $\mu$ be a non-elementary probability measure on $G$ with finite second moment, not necessarily supported by $G_\R$. An element $(g_n,\dots,g_1)$ of $G^{n}$ for some $n\geq 1$ will be denoted by $\g$. When there is no ambiguity, we'll also denote by $\g$ the product $g_n \cdots g_1$. For $\g=(g_n,\dots,g_1) \in G^n$, we define $L\g:=(g_{n-1},\dots,g_1)$, which deletes the leftmost entry of $\g$, and $\g^{-1}:=(g_1^{-1},\dots,g_n^{-1})$. As above, we'll also write $L\g$ to represent the product $g_{n-1}\cdots g_1$. For convenience, we let  $L\g=\mathrm {Id}$ for $\g\in G^1$.

For $t > 0$, define
$$\bN_t^+:=\bigsqcup_{n\geq 1} \big\{\g\in G^n:\, \log\norm{L\g}<t\leq \log\norm{\g}\big\},$$
$$\bN_t^-:=\bigsqcup_{n\geq 1} \big\{\g\in G^n:\, \log\norm{\g}<t\leq \log\norm{L\g}\big\}$$
and 
$$\bM_t^+:=\{\g^{-1}:\, \g\in \bN_t^+\} , \quad     \bM_t^-:=\{\g^{-1}:\, \g\in \bN_t^-\}.$$
The above notations are compatible with the ones from Subsection \ref{subseq:E_2^+}.
\vskip 3pt

Let $\check \mu $ be the image $\mu$ under the map $g\mapsto g^{-1}$. Notice that $\check \mu$ is also non-elementary and has a finite second moment. Let  $\check\nu$ be the stationary measure corresponding to $\check \mu $. Denote by $\mu^{\N}$ (resp. $\check \mu ^\N$) the positive measure on $\bigsqcup_{n\geq 1} G^n$ given by $\sum_{n\geq 1} \mu^{\otimes n}$ (resp. $\sum_{n\geq 1} \check \mu ^{\otimes n}$).  By definition, $\mu^{\N}(\bM^+_t)=\check \mu ^\N(\bN^+_t)$ and $\mu^{\N}(\bM^-_t)=\check \mu ^\N(\bN^-_t)$. Observe that $\mu^{\N}$ and $\check \mu ^\N$ might have infinite mass over some sets. However, we'll only work with sets having finite mass. In particular, $\mu^\N (\bM_t^+)$ and $\mu^\N (\bM_t^-)$ are finite for every $t > 0$ after Lemma \ref{large-n}.

We actually have a better estimate on the mass of the above sets.

\begin{lemma}\label{dominate-s1}
	Let $\mu$ be a non-elementary probability measure on $G$ with  a finite second moment.
	Then there exists a decreasing rate function $\varepsilon_1(s)$ such that $$\mu^\N\big\{\g\in \bN_t^\pm:\,  \big|\log\norm{\g}-\log\norm{L\g}\big| \geq s\big\}\leq \varepsilon_1 (s) \quad \text{for all}\quad t\geq s\geq 0.$$
	 In particular $\mu^\N (\bN_t^\pm)$ are bounded uniformly in $t$. The same estimates hold after replacing $\bN_t^\pm$  by $\bM_t^\pm$, or $\mu$ by $\check\mu$.
\end{lemma}

\begin{proof} 
We will only prove that
 $$\mu^\N\big\{\g\in \bN_t^+:\,  \log\norm{\g}-\log\norm{L\g} \geq s\big\}\leq \varepsilon_1 (s) \quad\text{for all}\quad t\geq s\geq 0$$ 
  for some decreasing rate function $\varepsilon_1(s)$.  The other cases can be obtained in the same way. From Lemma \ref{large-n}, it is easy to see that $\mu^\N(\bN_t^+)$ is bounded for  $t\leq 3\gamma$. Hence we can assume   $t\geq 3\gamma$. 

From Lemma \ref{large-n} applied to both $t$ and $t + \gamma / 2 $, we have 
\begin{equation} \label{head-tail-2}
\sum_{n\leq \lfloor 2t/(3\gamma)\rfloor }\mu^{ n}(\bN_t^+)+\sum_{n\geq \lceil 2t/\gamma\rceil}\mu^{ n}(\bN_t^+)   \leq \widetilde \varepsilon_0(t) 
\end{equation}
for some decreasing rate function $\widetilde \varepsilon_0(t)$.
Let $\oL $ be the operator defined before Lemma \ref{operator-L}. Using   the fact that $\big|\log\norm{hg}-\log\norm{g}\big|\leq \log\norm{h}$ and the assumption $t \geq s$, we get that
\begin{align*}
&\mu^\N\big\{\g\in \bN_t^+:\, \log\norm{\g}-\log\norm{L\g}\geq s\big\} =\sum_{n\geq 0}\int_{\log\norm{g}<t\leq \log\norm{hg}} \mathbf 1_{\log\norm{hg}-\log\norm{g}\geq s}\, \diff \mu^{*n}(g)\diff \mu(h) \\
&\leq \widetilde \varepsilon_0(t)+\sum_{n=\lceil 2t/(3\gamma)\rceil }^{\lfloor 2t/\gamma\rfloor} \int_{G^2} \mathbf 1_{-\log\norm{h}\leq \log\norm{g}-t <0,\,\log\norm{h}\geq s}\, \diff \mu^{*n}(g)\diff \mu(h) \\
&\leq \widetilde \varepsilon_0(s)+\int_{\log\norm{h}\geq s} \oL (\mathbf 1_{u\in [-\log\norm{h},0]}) (x,t) \,\diff \mu(h) 
\lesssim \widetilde\varepsilon_0(s)+\int_{\log\norm{h}\geq s}\big(1+\log^2\norm{h} \big)\, \diff \mu(h),
\end{align*}
where in the last inequality we have used Lemma \ref{operator-L}. 
Since $\int_G\big(1+\log^2\norm{h} \big) \diff \mu(h)$ is finite by assumption, the last  integral above defines a decreasing rate function  $\eta(s)$.  This gives the desired result. 
\end{proof}

The first step, as in \cite{li:fourier}, is to use the sets $\bM^\pm_t$ to split the integral against $\nu$. 

\begin{proposition}\label{decomp}
		Let $\mu$ be a non-elementary probability measure on $G$ with  a finite second moment. Let $\nu$ be the associated stationary measure. Then, for every continuous function $f$ on $\P^1$ and  $t>0$, we have
	$$\int_{\P^1}f(x)\,\diff \nu(x) =\int_{\P^1}\Big(\int_{ \bM^+_t}f(\g x) \,\diff \mu^\N(\g) -\int_{ \bM^-_t} f(\g x) \,\diff \mu^\N(\g) \Big)\,\diff \nu(x).$$ 
	Moreover, if $\supp(\mu)\subset \SL_2(\R)$, the condition that $f$ is continuous on $\P^1$ can be replaced by $f$ being continuous on $\bC$ and the above integrals can be restricted to $\bC$. 
\end{proposition}

\begin{proof}
When $\mu$ has a finite exponential moment this is \cite[Proposition 3.5]{li:fourier}.  The proof given there still holds under the second moment assumption thanks to Proposition \ref{prop:BQLDT} and Lemma \ref{large-n} above. We warn that there is a slight difference in the notation used in the definition of $\bM^\pm_t$ given here and the one from \cite{li:fourier}. However, the integrals we consider are the same.  If $\supp(\mu)\subset \SL_2(\R)$, its elements preserve $\bC$ and $\supp(\nu)\subset \bC$. Hence, all the quantities involved only depend on the values of $f$ on $\bC$.
\end{proof}

\proof[Proof of Theorem \ref{thm:fourier-general}] First observe that $e^{ik\varphi}=\overline{e^{-ik\varphi}}$, so it is enough to prove the theorem for $k>0$. We'll show that for every $\ep>0$, there exists a large constant $s$ and another large constant $t_0$ depending on $s$ such that, if $k = e^{2t+s}$ and $t > t_0$ then $|\int_{\bC } e^{ik\varphi}\psi \,\diff \nu |<\ep$. This clearly implies the result.
\vskip 3pt

Applying Proposition \ref{decomp} to $e^{ik\varphi}\psi$, we have 
\begin{align}
\Big|\int_{\bC} e^{ik\varphi}\psi \,\diff \nu \Big|=\Big|     \int_{\bC}\Big(\int_{ \bM^+_t}e^{ik\varphi(\g x)}\psi(\g x) \,\diff \mu^\N(\g) -\int_{ \bM^-_t} e^{ik\varphi(\g x)}\psi(\g x) \,\diff \mu^\N(\g) \Big)\diff \nu(x)       \Big|  \nonumber\\
\leq  \Big|     \int_{ \bM^+_t}\int_{\bC}e^{ik\varphi(\g  x)}\psi(\g x) \,\diff \nu(x) \diff \mu^\N(\g)  \Big| + \Big|     \int_{ \bM^-_t}\int_{\bC}e^{ik\varphi(\g x)}\psi(\g x)\, \diff \nu(x) \diff \mu^\N(\g)  \Big| . \label{main-inequality}
\end{align}
We have used that 	$\mu^\N(\bM_t^\pm)$ is finite (cf. Lemma \ref{dominate-s1}) and that the integrands are bounded functions, in order to apply Fubini's theorem.
\vskip3pt

We will first bound the first term of \eqref{main-inequality} for $k$, $s$ and $t$ as above. The second term will be treated similarly in the end of the proof. Recall that $\check\nu$ is the stationary measure corresponding to $\check \mu $. Using Cauchy-Schwarz inequality and the fact that $\mu^\N(\bM^+_t)$ is bounded uniformly in $t$ (cf. Lemma \ref{dominate-s1}), the  first term of \eqref{main-inequality} is bounded by 
\begin{align}
&\mu^\N(\bM^+_t)^{1/2}\Big(\int_{\bM^+_t}\Big|\int_{\bC} e^{ik\varphi(\g x)}\psi(\g x)\,\diff \nu(x) \Big|^2\diff \mu^\N(\g)\Big)^{1/2}\nonumber\\
&\lesssim \Big(\int_{\bC}\int_{\bC} \int_{\bM^+_t} e^{ik\varphi(\g x)-ik\varphi(\g y)}\psi(\g x)\psi(\g y)\, \diff \mu^\N(\g)\diff \nu(x)\diff \nu(y) \Big)^{1/2}\nonumber\\
&=\Big(\int_{\bC}\int_{\bC} \int_{\bN^+_t} e^{ik(\varphi(\g^{-1}x)-\varphi(\g^{-1}y))}\psi(\g^{-1}x)\psi(\g^{-1}y)\, \diff \check \mu ^\N(\g)\diff \nu(x)\diff \nu(y) \Big)^{1/2}. \label{inequality-Nt}
\end{align}

Let $s>20$ and $t >  s^{20}$ be large constants whose values will be determined later.  Set $k:=e^{2t+s}$. By the regularity of $\nu$ given by Proposition \ref{regularity}, we have 
$$\iint_{d(x,y)\leq e^{-s/9}} \mathbf 1 \,\diff \nu(x)\diff \nu(y) \lesssim 1/s.$$ 
Since the  integrand in \eqref{inequality-Nt} is a bounded function and $\check\mu^\N (\bN_t^+)$ is bounded uniformly in $t$ by Lemma \ref{dominate-s1}, we conclude that 
\begin{equation*}
\Big|\iint_{d(x,y)\leq e^{-s/9}} \int_{\bN^+_t} e^{ik(\varphi(\g^{-1}x)-\varphi(\g^{-1}y))}\psi(\g^{-1}x)\psi(\g^{-1}y)\, \diff \check \mu ^\N(\g)\diff \nu(x)\diff \nu(y)\Big| \leq C /s
\end{equation*}
for some constant $C$.

When  $d(x,y)> e^{-s/9}$, we have by Proposition \ref{x-y>es} below that
\begin{align*}
\Big|\iint_{d(x,y)> e^{-s/9}} \int_{\bN^+_t} e^{ik(\varphi(\g^{-1}x)-\varphi(\g^{-1}y))}\psi(\g^{-1}x)\psi(\g^{-1}y)\, \diff \check \mu ^\N(\g)\diff \nu(x)\diff\nu(y)\Big| 
\leq \varepsilon(s)+\varepsilon_{s}(t),
\end{align*}
where $\varepsilon(s)$ (resp. $\varepsilon_{s}(t)$) tends to zero when $s$ (resp. $t$) tends to infinity. We conclude that \eqref{inequality-Nt} is bounded by  $C/s+\varepsilon(s)+\varepsilon_{s}(t)$.

Now, fix $\ep>0$. Take $s$ large enough so that $C/s+\varepsilon(s) \ll \ep^2/8$. Then, choose  $t_0$ depending on $s$ such that $\varepsilon_{s}(t) \ll \ep^2/8$ for all $t\geq t_0$. The above estimates show that the first term of \eqref{main-inequality} is smaller than $\ep \slash 2$ for $k = e^{2t+s}$ and $t > t_0$.

The second term of \eqref{main-inequality} is treated analogously. In this case we apply Proposition \ref{x-y>es} for $\bN_t^-$ instead of $\bN_t^+$. All the other estimates remain unchanged. We conclude that \eqref{main-inequality} is smaller than $\ep$ for $k = e^{2t+s}$ and $t > t_0$ chosen as before. As $\ep >0$ is arbitrary, this gives that $\int_{\bC} e^{ik\varphi}\psi \,\diff \nu \to 0$ as $k \to \infty$, concluding the proof of the theorem. Observe that all the constants involved the proof are uniform in $\varphi$ and $\psi$ satisfying the conditions of the theorem.
\endproof

\begin{proposition} \label{x-y>es}
	Let $\mu,\varphi,\psi$ be as in Theorem \ref{thm:fourier-general} and let
	let $k:=e^{2t+s}$.  There exist a rate function $\varepsilon(s)$ and, for each $s$, a rate function $\varepsilon_{s} (t)$ such that 
	$$\Big| \int_{\bN^\pm_t} e^{ik(\varphi(\g^{-1}x_0)-\varphi(\g^{-1}y_0))}\psi(\g^{-1}x_0)\psi(\g^{-1}y_0)\, \diff \check \mu ^\N(\g) \Big| \leq \varepsilon(s)+\varepsilon_{s}(t) $$
	for all $x_0,y_0\in \bC $ with $d(x_0,y_0)> e^{-s/9}$ and $s>20,t> s^{20}$.
\end{proposition}

\begin{proof}
The proof is given in Subsection \ref{subsec:fourier-B} below.
\end{proof}

Theorem \ref{thm:main-fourier} can now be easily deduced from Theorem \ref{thm:fourier-general}.

\begin{proof}[Proof of Theorem \ref{thm:main-fourier}]
	We parametrize $\mathbf C$ by $e^{i\theta}$ with $0\leq \theta<2\pi$.
Consider a partition of unity given by smooth functions $\chi_1,\chi_2$ on $\bC$ such that $\chi_1+\chi_2=1$ and  each of them is supported on an arc of $\bC$. Choose $\varphi_1,\varphi_2$ smooth on $\mathbf C$ such that $\varphi_j(\theta)=\theta$ or $\theta\pm 2\pi$ on $\supp(\chi_j)$. Applying Theorem \ref{thm:fourier-general} to $\varphi_1,\chi_1$ and $\varphi_2,\chi_2$ respectively instead of $\varphi,\psi$ gives that
$$\int_{\bC} e^{ik \varphi_1} \chi_1 \,\diff\nu \longrightarrow 0 \quad\text{and}\quad \int_{\bC} e^{ik \varphi_2} \chi_2 \,\diff\nu \longrightarrow 0    \quad\text{as}\quad     |k|\to\infty.$$
Observe that $\widehat\nu (k)$ is the sum  of these two integrals (recall that we assume here that $k\in\Z$). The theorem follows.
\end{proof}

\subsection{ The proof of Proposition \ref{x-y>es}} \label{subsec:fourier-B}

We now discuss the proof of Proposition \ref{x-y>es} above. As discussed in the previous subsection, the idea is to replace the integral in the statement of the proposition by a quantity that can be encoded by a suitable renewal operator. This will allow us to use the analysis carried out in Section \ref{sec:renewal}. The main difference from our work and \cite{li:fourier} is that the estimates given there, which are usually of exponential decay, are no longer valid in our setting and must replaced by their weaker analogues obtained in Section \ref{sec:renewal}.

\medskip

We start by fixing two points $x_0,y_0 \in \bC$ such that $d(x_0,y_0)> e^{-s/9}$. We'll prove Proposition \ref{x-y>es} for these points and ensure that all the estimates are uniform in $x_0$ and $y_0$.
Recall that $k:=e^{2t+s}$ and $\supp(\check \mu)\subset \SL_2(\R)$. From now until the end of the paper, we assume that all matrices are in $\SL_2(\R)$.

  Our goal is to estimate the integral of
\begin{equation} \label{eq:def-Gamma}
\Gamma(\g):= e^{ik(\varphi(\g^{-1}x_0)-\varphi(\g^{-1}y_0))}\psi(\g^{-1}x_0)\psi(\g^{-1}y_0)
\end{equation} 
 with respect to $\check \mu^\N$ over $\g\in\bN_t^\pm$ when $t> 2s^8$ and $s$ tends to infinity. 
 Using the mean value theorem, the term $e^{ik(\varphi(\g^{-1}x_0)-\varphi(\g^{-1}y_0))}$ can be estimated from the derivative $\varphi'$ and the length of the smallest arc joining $\g^{-1}x_0$ and $\g^{-1}y_0$. In order to give a proper meaning to that, we need to take into account the orientation of $\bC$. For this, we consider the \textit{sign function} defined by 
 \begin{align*}
 \sgn (x,y,z) &=
  \begin{cases}
   0        & \text{if } x,y \text{ and } z \text{ are not pairwise distinct,} \\
   1        & \text{if } x,y \text{ and } z \text{ are in counterclockwise order,}   \\
   -1        & \text{otherwise.}
  \end{cases}
\end{align*}

The length of the arc joining $\g^{-1}x_0$ and $\g^{-1}y_0$ is controlled by the derivative of $\g^{-1}$, which is in turn related to the cocycle $\sigma_{\g^{-1}}$ and the quantity $\|\g^{-1}\|^2 = \| \g \|^ 2$ (cf. \eqref{d-gxgy}). Moreover, with high probability, the points $\g^{-1}x_0$ and $\g^{-1}y_0$ are very close.
This motivates the introduction of the following function in order to approximate $\Gamma(\g)$:
 \begin{equation} \label{eq:def-Lambda}
 \Lambda(\g):=\exp\Big[ik\cdot \sgn(\g x_0,x_0,y_0)\varphi'(\g^{-1}x_0){{d(x_0,y_0)}\over {\norm{\g}^2d(\g x_0,x_0)d(\g x_0,y_0)}}\Big]\psi^2(\g^{-1}x_0).
 \end{equation}

The precise approximation is given by Proposition \ref{prop:li-approximation} below. It offers a good estimate on the difference between $\Gamma(\g)$ and $\Lambda(\g)$ on a large subset $\bN^\pm_{t,s}(x_0,y_0)$ of $\bN^\pm_t$ defines as follows. For each $x,y\in \P^1$, we set 
\begin{align}
\bN^\pm_{t,s}(x,y):=\Big\{\g\in \bN^\pm_t:\, \big| \log\norm{\g}-\log\norm{L\g}\big|&<s/9, \,d(\g x,z^M_\g)<e^{-t},  \nonumber\\
&d(\g x,x)>2e^{-s/9},\,d(\g x,y)>2e^{-s/9}\Big\}. \label{N-t-s}
\end{align}

\begin{proposition}[\cite{li:fourier}--Proposition 3.6] \label{prop:li-approximation}
Assume $t > 2s$.
Let $\Gamma(\g)$ and $\Lambda(\g)$ be the functions introduced above. Then, there are constants $C>0$ and $\beta>0$ independent of $x_0,y_0$ such that, for $\g$ real	
$$\big|\Gamma(\g)-\Lambda(\g)\big|\leq C e^{-\beta s} \quad \text{for all} \quad \g \in \bN^\pm_{t,s}(x_0,y_0).$$
\end{proposition}

The set $\bN^\pm_{t,s}(x,y)$ is a large subset of $\bN^\pm_t$ in the following sense.

\begin{lemma} \label{lemma:size-N(x,y)}
There exists a decreasing rate function $\varepsilon_2(s)$  such that $$\mu^\N\big(\bN^\pm_t\setminus \bN^\pm_{t,s}(x,y)\big)\leq \varepsilon_2(s) \quad\text{for all} \quad s>20, \,\, t\geq s^{20}  \,\, \text{ and } \,\, x,y\in\P^1.$$
	The same estimate holds for $\check\mu$ instead of $\mu$.
\end{lemma}

\begin{proof}
	We will only prove the lemma for $\bN^+_t$ and $\mu$. The other cases can be obtained in the same way.
	From the definition of $\bN^+_{t,s}(x,y)$ it is clear that the desired estimate will follow if we show that the quantities
	\begin{equation} \label{eq:N(x,y)-1}
	\mu^\N\big\{\g\in \bN_t^+:\, \log\norm{\g}-\log\norm{L\g} \geq s\big\},
	\end{equation}
	\begin{equation} \label{eq:N(x,y)-2}
	\mu^\N\big\{\g\in \bN^+_t:\,d(\g x,z_\g^M)\geq e^{-t}\big\},
	\end{equation}
	\begin{equation} \label{eq:N(x,y)-3}
	\mu^\N\big\{\g\in \bN^+_t:\,d(\g x,y)\leq e^{-s}\big\}
	\end{equation}
	decrease to zero when $s>1, t\geq 2s^8$ and $s \to \infty$, uniformly in $x,y \in \P^1$. 
	
	The fact that \eqref{eq:N(x,y)-1} tends to  zero as $s \to \infty$ is the content of Lemma \ref{dominate-s1}.
	
	Let us estimate  \eqref{eq:N(x,y)-2}. 
	Using \eqref{head-tail-2}, it is enough to show that $$\sum_{n=\lceil 2t/(3\gamma)\rceil }^{\lfloor 2t/\gamma\rfloor} \mu^{*n} \big\{  d(gx,z_g^M)\geq e^{-t}    \big\} \longrightarrow 0 \quad\text{as}\quad t\to \infty.$$
	Notice that when $\lceil 2t/(3\gamma)\rceil\leq n$, we have $t\leq 3\gamma n/2$. It follows that  $e^{-t}\geq e^{-3\gamma n/2}.$ Applying Proposition \ref{prop:BQLDT} with $\ep=\gamma/2$, we obtain $\mu^{*n}\big\{d(gx,z_g^M)\geq e^{-3\gamma n/2}\big\}\leq C_n$. Hence, 
	$$\sum_{n=\lceil 2t/(3\gamma)\rceil }^{\lfloor 2t/\gamma\rfloor} \mu^{*n} \big\{  d(gx,z_g^M)\geq e^{-t}    \big\} \leq \sum_{n=\lceil 2t/(3\gamma)\rceil }^{\lfloor 2t/\gamma\rfloor}C_n,$$ which tends to $0$ as $t$ tends to infinity since $\sum C_n$ is finite. We conclude that  \eqref{eq:N(x,y)-2} tends to  zero as $s \to \infty$.
	
	It now remains to estimate  \eqref{eq:N(x,y)-3}. That quantity is bounded by
	\begin{align}
	\mu^\N\big\{\g\in \bN_t^+:\, &\log\norm{\g}-\log\norm{L\g}\geq \log s\big\} \, + \nonumber\\
	&  \mu^\N\big\{\g\in \bN^+_t:\,d(\g x,y)
	\leq e^{-s},\, \log\norm{\g}-\log\norm{L\g}\leq \log s\big\} .\label{eq:N(x,y)-3-2}
	\end{align}
	By Lemma \ref{dominate-s1}, the first term above is bounded by $\varepsilon_1(\log s)$, which decreases to zero as $s \to \infty$. Therefore, it is enough to bound the last term. As before, after \eqref{head-tail-2}, it is enough to consider $n$ in the interval between $\lceil 2t/(3\gamma)\rceil$ and $\lfloor 2t/\gamma\rfloor$. Using the definition of $\bN^+_t$ we see that, in this range, the quantity in  \eqref{eq:N(x,y)-3-2}  is bounded by 
	\begin{align*}
	&   \sum_{n=\lceil 2t/(3\gamma)\rceil -1}^{\lfloor 2t/\gamma\rfloor-1}\int_{\log\norm{g}<t\leq \log\norm{hg}}\mathbf 1_{\log\norm{hg}-\log\norm{g}\leq \log s} \cdot\mathbf 1_{d(hgx,y)\leq e^{-s}}\,\diff \mu^{*n}(g)\diff \mu(h) \\
	&\leq \sum_{n=\lceil 2t/(3\gamma)\rceil -1}^{\lfloor 2t/\gamma\rfloor-1}\int_{0\leq \log\norm{hg}-t\leq \log s} \mathbf 1_{d(hgx,y)\leq e^{-s}}\,\diff \mu^{*n}(g)\diff \mu(h) = \oL (\mathbf 1_{\D(y,e^{-s})\times [0,\log s]})(x,t),
	\end{align*}
	where $\oL $ is defined before Lemma \ref{operator-L}.

	Let  $l:=\lfloor t/(3\gamma)\rfloor$,  $n$ and $\bS_{n,l,x}$ be as in the proof of Lemma \ref{operator-L}. Notice that $\sum_{n=\lceil 2t/(3\gamma)\rceil }^{\lfloor 2t/\gamma\rfloor}(D_n+D_l)$ is bounded by a decreasing rate function $\rho(t)$. Using  \eqref{dominate-equation-1} and that  $t\geq 2s^8 \geq s$, we get that $\oL (\mathbf 1_{\D(y,e^{-s})\times [0,\log s]})(x,t)$ is bounded by 
	\begin{align*}
	&\sum_{n=\lceil 2t/(3\gamma)\rceil }^{\lfloor 2t/\gamma\rfloor}\int_{\bS_{n,l,x}} \mathbf  1_{\D(y,e^{-s})\times [0,\log s]}\big(g_2g_1x,\log\norm{g_2g_1}-t\big) \,\diff \mu^{*(n-l)}(g_2)\diff \mu^{*l}(g_1)  + \rho(t)\\
	&\leq \sum_{n=\lceil 2t/(3\gamma)\rceil }^{\lfloor 2t/\gamma\rfloor}\int_{\bS_{n,l,x}}\mathbf 1_{d(g_2g_1x,y)<e^{-s},\,-4\leq\sigma_{g_2}(g_1x)+\log\norm{g_1}-t\leq\log s+4} \,\diff \mu^{*(n-l)}(g_2)\diff \mu^{*l}(g_1) + \rho(s) \\
	&\leq  \int_{G} \oR( \mathbf 1_{\D(y,e^{-s})\times [-4,\log s+4]})\big(g_1x,t-\log\norm{g_1}\big) \,\diff \mu^{*l}(g_1)  + \rho(s).
	\end{align*}
	
	As in the proof of Lemma \ref{integral-R-l}, consider the set $\bB_l:=\big\{g_1\in G:\, \log\norm{g_1}\leq 3\gamma l/2\big\}$. Then $\mu^{*l}(\bB_l)\geq 1-C_l$ and $t-\log\norm{g_1}\geq t/2\geq s^{8}$ for $g_1\in \bB_l$. Lemma \ref{logs-lemma} implies that
	\begin{equation*}\label{inside-Bl}
	\oR( \mathbf 1_{\D(y,e^{-s})\times [-4,\log s+4]})\big(g_1x,t-\log\norm{g_1}\big)\leq \zeta(s)\quad\text{for all} \quad g_1 \in \bB_l,
	\end{equation*}
	where $\zeta(s)$ is a decreasing  rate function.
	
	On the other hand,   Lemma \ref{R-b-b}  gives
	\begin{equation*}\label{outside-Bl}
	\oR(\mathbf 1_{\D(y,e^{-s})\times [-4,\log s+4]})\leq    \oR( \mathbf 1_{u\in [-\log s-4,\log s+4]})\lesssim (\log s+5)^2.
	\end{equation*}

	Hence, splitting $G$ into $\bB_l$ and its complement and using the last two inequalities, we get
	\begin{align*}
	\int_{G} \oR(\mathbf 1_{\D(y,e^{-s})\times [-4,\log s+4]})\big(g_1x,t-\log\norm{g_1}\big) \diff \mu^{*l}(g_1)
	\lesssim \zeta
	(s) +   C_l  (\log s+5)^2 . 
	\end{align*}
	
	We have $C_l(\log s+5)^2\leq C_l(\log t+5)^2\lesssim C_l(t+5)$, since $t\geq 2s^8 \geq s$. Recall that  $l=\lfloor t/(3\gamma)\rfloor$, so the last quantity which tends to $0$ as $s\to\infty$. Thus, we can find a decreasing rate function dominating   $C_l(\log s+5)^2$.
	
	We conclude that \eqref{eq:N(x,y)-3-2} tends to $0$ as $s\to\infty$, which yields the same property for \eqref{eq:N(x,y)-3}.
	This completes the proof of the lemma.
\end{proof} 

The following result is stated in \cite[Lemma 3.8]{li:fourier} and can be obtained by using integration by parts.
\begin{lemma}\label{inequality-exp}
	For $b_1<b_2$ and $w\in\R\setminus \{0\}$, we have
	$$\Big| \int_{b_1}^{b_2} \exp\big[iw e^{-u}\big] \,\diff u \Big|\leq {{2e^{b_1}+2e^{b_2}}\over |w|}.$$
\end{lemma}

We will need the following simple extension lemma for $\Cc^1$ functions on $\bC\times\bC\times\R\times \R$. The proof is left to the reader.

\begin{lemma}\label{extend}
	Let $F$ be a $\Cc^1$ function on $\bC \times\bC\times\R\times\R $. Then, there exists a $\Cc^1$ function $\widetilde F$ on $\P^1 \times\P^1\times\R\times\R$ such that $\widetilde F=F$ on $\bC\times\bC\times\R\times \R $,     
	$\norm{\widetilde F}_{L^\infty(\P^1\times\P^1\times\R\times\R)}\leq \norm{F}_{L^\infty(\bC\times\bC\times\R \times\R)}$, $\norm{\widetilde F}_{\Cc^1(\P^1\times\P^1\times\R\times\R)}\lesssim \norm{F}_{\Cc^1(\bC\times\bC\times\R\times\R )}$, and the projection of  $\supp(\widetilde F)$ to $\R\times\R$ is the same as the projection of  $\supp(F)$ to $\R\times\R$.
\end{lemma}

We now proceed to the proof of Proposition \ref{x-y>es}.

\begin{proof}[Proof of Proposition \ref{x-y>es}] Recall that $x_0,y_0\in \bC $ are fixed and $d(x_0,y_0)> e^{-s/9}$. We first consider the case of $\bN_t^+$.  Recall that $\mu ^\N(\bN_t^+)$ is bounded uniformly in $t >0$ by Lemma \ref{dominate-s1}. Let $\bN^+_{t,s}(x_0,y_0)$ be the set defined in \eqref{N-t-s}.  Remind that $s>20,t> s^{20}$ by assumption and our goal is to prove that $$\Big|\int_{\bN^+_t} \Gamma(\g) \, \diff \check \mu ^\N(\g) \Big| \leq \varepsilon(s)+\varepsilon_{s}(t),$$ where $\Gamma(\g)$ is the function defined in \eqref{eq:def-Gamma} and  $\varepsilon(s),\varepsilon_{s}(t)$ are  rate functions as in the statement of the proposition.
	\vskip 3pt

 Lemma \ref{lemma:size-N(x,y)} applied to $\check\mu$ yields  a rate function $\varepsilon_2(s)$  such that $$\check\mu^\N\big(\bN^+_t\setminus \bN^+_{t,s}(x_0,y_0)\big)\leq \varepsilon_2(s).$$
Let $\Lambda(\g)$ be the function defined in \eqref{eq:def-Lambda}. Since $\psi$ is bounded, so are $\Gamma(\g)$ and $\Lambda(\g)$. Using this and the fact that $\supp(\check \mu)\subset \SL_2(\R)$ we deduce, from Proposition \ref{prop:li-approximation},  that 
	\begin{align}
	\Big|\int_{\bN^+_t} \big[\Gamma(\g) -\Lambda(\g) \big] \diff \check \mu ^\N(\g) \Big| 
	\lesssim \varepsilon_2(s) +e^{-\beta s}. \label{inequality-Lambda}
	\end{align}
	
	Since $k=e^{2t+s}$, we can write $\Lambda(\g)$ as
	$$\Lambda(\g)=\exp\Big[ie^s \cdot  \sgn(\g x_0,x_0,y_0)\varphi'(\g^{-1}x_0)e^{-2(\log\norm{\g}-t)}{{d(x_0,y_0)}\over {d(\g x_0,x_0)d(\g x_0,y_0)}}\Big]\psi^2(\g^{-1}x_0).$$
	This can be further rewritten as $$\Lambda(\g) = f\big(\g^{-1}x_0,\g x_0,\log\norm{\g}-\log\norm{L\g},\log\norm{L\g}-t\big),$$
	where  $f$ is the function on  $\bC  \times \bC \times \R_v \times \R_u$ defined by
	$$f(\check x,x,v,u):=\exp  \Big[ie^s \cdot \sgn (x,x_0,y_0)\varphi'(\check x)e^{-2(u+v)}{{d(x_0,y_0)}\over {d(x,x_0)d(x,y_0)}} \Big]\psi^2(\check x).$$
	
	The expected decay for $\Lambda(\g)$ will be achieved by applying Proposition \ref{renewal-4}. In order to apply that proposition, we need to extend $f$ to a function on  $\P^1  \times \P^1 \times \R_v \times \R_u$ with a good control on the $\Cc^1$-norm and modify it so that the projection of its support to $\R_v$ becomes compact.
	
 Consider the function	
	$$F(\check x,x,v,u):=f(\check x,x,v,u)\cdot \chi\big(e^{s/9}d(x,x_0)\big) \chi\big(e^{s/9}d(x,y_0)\big)\cdot \tau
	\Big({9v\over s}\Big)\tau\Big({9u\over2s}\Big),$$ 
	where $\chi$ is a increasing smooth function on $\R_{\geq 0}$ such that $\chi=0$ on $[0,1]$, $\chi=1$ on $[2,\infty)$, and $\norm{\chi}_{\Cc^1}\leq 4$, and $\tau$ is a real smooth cut-off function supported by $[-2,2]$ such that $|\tau|\leq 1$,   $\tau= 1$ on $[-1,1]$,  and $\norm{\tau}_{\Cc^1}\leq 4$.
	
		Observe that $f$ is discontinuous at $x=x_0$ and $x=y_0$ due to the discontinuity of  $\sgn (x,x_0,y_0)$. However, when $d(x,x_0)\leq e^{-s/9}$ or $d(x,y_0)\leq e^{-s/9}$, a factor involving $\chi$ in the definition of $F$ vanishes, thus canceling the singularity of $f$. This implies that $F$ is $\Cc^1$ on $\bC  \times \bC \times \R_v \times \R_u$. From the assumptions $\norm{\varphi}_{\Cc^2}\leq c, \norm{\psi}_{\Cc^1}\leq c$, it is not hard to see that  $$\norm{F}_{\Cc^1(\bC  \times \bC \times \R_v \times \R_u)}\lesssim e^{2s}.$$
		
		It is clear that the projection of $\supp(F)$ to $\R_v$ is contained in $[-2s/9,2s/9]$. 
	Moreover, note that for $\g\in \bN^+_{t,s}(x_0,y_0)$, we have 
	 $0<\log\norm{\g}-\log\norm{L\g}<s/9,\, -s/9<\log\norm{L\g}-t<0,\, d(\g x_0,x_0)>2e^{-s/9},\,d(\g x_0,y_0)>2e^{-s/9}
	 $.
	 So in this case, 
	 $$\chi\big(e^{s/9}d(\g x_0,x_0)\big) =\chi\big(e^{s/9}d(\g x_0,y_0)\big)=\tau\Big({9(\log\norm{\g}-\log\norm{L\g})\over s}\Big)=\tau\Big({9(\log\norm{L\g}-t)\over2s}\Big)=1.$$ 
	We conclude that 	for all  $\g\in \bN^+_{t,s}(x_0,y_0)$ real,
	\begin{equation}\label{f-tau-lambda}
	F\big(\g^{-1}x_0,\g x_0,\log\norm{\g}-\log\norm{L\g},\log\norm{L\g}-t\big)= \Lambda(\g)  .  
	\end{equation}

Using Lemma \ref{extend}, we can take a $\Cc^1$ function $\widetilde F$ such that $\widetilde F=F$ on $\bC  \times \bC \times \R_v \times \R_u$, $\norm{\widetilde F}_{\Cc^1(\P^1 \times \P^1\times \R_v \times \R_u)}\lesssim e^{2s}$ and the projection of the support of $\widetilde F$ to $\R_v$ is contained in $[-2s/9,2s/9]$. Notice that \eqref{f-tau-lambda} still holds with $\widetilde F$ instead of $F$. Therefore, we have
	\begin{align}\Big|\int_{\bN^+_t} \Big[\widetilde F\big(\g^{-1}x_0,\g x_0,\log\norm{\g}-\log\norm{L\g},\log\norm{L\g}-t\big) -\Lambda(\g)  \Big]\diff \check \mu ^\N(\g)\Big| 
	\lesssim \varepsilon_2(s). \label{inequality-f'}
	\end{align}
	
Observe now that with our notations
 $$\int_{\bN^+_t}   \widetilde F\big(\g^{-1}x_0,\g x_0,\log\norm{\g}-\log\norm{L\g},\log\norm{L\g}-t\big)\diff \check \mu ^\N(\g)=\oE_2^+ \widetilde F(x_0,x_0,t) ,$$ 
 where $\oE_2^+$ is the residual operator introduced in Subsection \ref{subseq:E_2^+}, with respect to $\check\mu$ instead of $\mu$.
We now apply Proposition \ref{renewal-4} to $\widetilde F$. Since  $\norm{\widetilde F}_{\Cc^1(\P^1 \times \P^1\times \R_v \times \R_u)}\lesssim e^{2s}$ and  the projection of $\supp(\widetilde F)$ to $\R_v$ is contained in $[-2s/9,2s/9]$, we obtain that $\big|\oE_2^+ \widetilde F(x_0,x_0,t)\big|$ is bounded by a constant times (remind that we are using $\check \mu$ instead of $\mu$)
	\begin{align}
	 \label{renewal-5-f'}  
	\Big| \int_{\P^1}\int_{\P^1}\int_G \int_{-\sigma_h^+(y)}^0 \widetilde F\big(\check y,hy,\sigma_h(y),u\big) \diff u \diff \check \mu (h) \diff \check\nu(y)\diff \nu(\check y)\Big|+C_\delta{ s^2e^{2s}\over (1+t)^{1/4}}+\delta e^{2s}+\varrho(t) e^{2s},
	\end{align} 
	for $t>4s/9+2$ and $0<\delta < 1$, where $\varrho(t)$ is a rate function and $C_\delta>0$ is independent of $\widetilde F,x_0,y_0,t,s$.

	Now, for $\check x,x\in\bC$, we put $$w(\check x,x):=e^{s}\cdot\sgn(x,x_0,y_0)\varphi'(\check x){{d(x_0,y_0)}\over {d(x,x_0)d(x,y_0)}}$$
	and
	 $$q(\check x,x):=\chi\big(e^{s/9}d(x,x_0)\big) \chi\big(e^{s/9}d(x,y_0)\big)\psi^2(\check x),$$
so that
$$\widetilde F(\check x,x,v,u) =  \exp \big[i e^{-2(u+v)} w(\check x,x) \big] q(\check x,x) \cdot\tau
\Big({9v\over s}\Big)\tau\Big({9u\over2s}\Big) \quad \text{for}\quad  x,\check x\in \bC.$$ 

		Since $d(x_0,y_0)>e^{-s/9},d(x,x_0)\leq 1,d(x,y_0)\leq 1$, the assumptions $|\varphi'|\geq 1/c$ on $\supp(\psi)$ and $|\psi|\leq c$  imply that
	$$|w(\check x,x)|\geq e^{8s/9}/c \quad \text{and} \quad |q(\check x,x)|\leq c^2 \quad \text{on} \quad  \supp(\psi) \times \bC .$$ 
	
	As $\widetilde F$ vanishes for $v$ outside $[-2s/9,2s/9]$,  the integrand in \eqref{renewal-5-f'}  is non-zero only when $\sigma^+_h(y)\leq 2s/9$, or equivalently $-2s/9\leq -\sigma_h^+(y)\leq u\leq 0$. In this case, $\tau({9u\over 2s})=1$. 
Therefore, using  that $\nu$ and $\check\nu $ are supported by $\bC $, the first term of \eqref{renewal-5-f'} is equal to
	\begin{align*}
	&\Big|\int_{\bC}\int_{\bC}\int_G\int_{-\sigma_h^+(y)}^0  \exp\big[i e^{-2(u+\sigma_h(y))} w(\check y,hy) \big] q(\check y,hy)\cdot \tau
	\Big({9\sigma_h(y)\over s}\Big) \,\diff u \diff \check \mu (h) \diff \check\nu(y)\diff \nu(\check y) \Big|\\
	&\leq\int_{\bC}\int_{\bC}\int_G \Big|q(\check y,hy) \tau
	\Big({9\sigma_h(y)\over s}\Big)\Big|\cdot \Big| \int_{-\sigma_h^+(y)}^0  \exp\big[i e^{-2(u+\sigma_h(y))} w(\check y,hy) \big]  \, \diff u\Big| \,\diff \check \mu (h) \diff \check\nu(y)\diff \nu(\check y) \\
	&\lesssim \int_{\supp(\psi)}\int_{\bC}\int_G\Big|\int_{-\sigma_h^+(y)}^0 \exp\big[ie^{-2(u+\sigma_h(y))}w(\check y,hy) \big]\, \diff u\Big| \,\diff \check \mu (h) \diff \check\nu(y)\diff \nu(\check y) \\
	&\leq \int_{\supp(\psi)}\int_{\bC}\int_G  \Big|\int_0^{\sigma_h^+(y)} \exp\big[ie^{-2u}w(\check y,hy) \big]\, \diff u \Big| \,\diff \check \mu (h) \diff \check\nu(y)\diff \nu(\check y)   \lesssim {1+e^{4s/9}\over  e^{8s/9}} \lesssim e^{-s/3},
	\end{align*}
	where in the last line we applied  Lemma \ref{inequality-exp} and used that $|w| \gtrsim e^{8s/9}$ and $0 \leq \sigma^+_h(y) \leq 2s/9$.
		Combining with \eqref{inequality-Lambda}, \eqref{inequality-f'} and \eqref{renewal-5-f'}, we obtain that
	\begin{equation*}\label{Nt+}
	\Big|\int_{\bN^+_t} \Gamma(\g) \, \diff \check \mu ^\N(\g) \Big|  \lesssim \varepsilon_2(s) +e^{-\beta s}+e^{-s/3}+C_\delta s^2e^{2s}(1+t)^{-1/4}+\delta e^{2s}+\varrho(t)e^{2s}.	
    \end{equation*}

	Now set $\delta:=e^{-3s}$. Then,  the last quantity is equal to  $\varepsilon_2(s) +e^{-\beta s}+e^{-s/3}+C_{e^{-3s}} s^2e^{2s}(1+t)^{-1/4}+ e^{-s}+\varrho(t)e^{2s}$. Finally, we take
	$$\varepsilon(s)=C\big[\varepsilon_2(s) +e^{-\beta s}+e^{-s/3}+ e^{-s} \big]\quad\text{and} \quad \varepsilon_s (t)=C\big[ C_{e^{-3s}} s^2e^{2s}(1+t)^{-1/4}+\varrho(t)e^{2s}\big]$$
	for some large constant $C$. This yields the desired result for $\bN^+_t$.
	\vskip 3pt
	
	For the case of $\bN^-_t$, we repeat the above proof by using the set $\bN^-_{t,s}(x_0,y_0)$ defined in  \eqref{N-t-s} instead of $\bN^+_{t,s}(x_0,y_0)$, and applying Proposition \ref{E_L^-} instead of Proposition \ref{renewal-4}. 
	It is clear that all estimates we used are uniform in $x_0$ and $y_0$. This concludes the proof of the proposition.
\end{proof}

\end{document}